\documentclass[11pt, reqno]{amsart}

\allowdisplaybreaks
\usepackage{color}
\usepackage{amssymb,amsthm,amsmath,bm}
\usepackage[numbers,sort&compress]{natbib}
\usepackage{xcolor}
\usepackage[unicode,breaklinks=true,colorlinks=true]{hyperref}
\title[ ]{Global Regularity of weak solutions to the generalized Leray equations and its applications}

\author[B. Lai]{Baishun Lai }

\address{LCSM (MOE) and School of Mathematics and Statistics, Hunan Normal University, Changsha 410081, Hunan,China}

\email{laibaishun@hunnu.edu.cn}

\author[C. Miao]{Changxing Miao}

\address{Institute of Applied Physics and Computational Mathematics, P.O. Box 8009, Beijing 100088, P.R. China }

\email{miao\_{}changxing@iapcm.ac.cn}

\author[X. Zheng]{Xiaoxin Zheng}

\address{ School of Mathematics and Systems Science, Beihang University, Beijing 100191, P.R. China}

\email{xiaoxinzheng@buaa.edu.cn}

 \date{\today}

\theoremstyle{plain}
\newtheorem{corollary}{Corollary}[section]
\newtheorem{theorem}{Theorem}[section]
\newtheorem{lemma}{Lemma}[section]
\newtheorem{proposition}{Proposition}[section]

\theoremstyle{definition}
\newtheorem{definition}{Definition}[section]

\newtheorem{remark}{Remark}[section]

\def\R{\mathbb R}

\newcommand{\beq}{\begin{equation}}
\newcommand{\eeq}{\end{equation}}
\newcommand{\ben}{\begin{eqnarray}}
\newcommand{\een}{\end{eqnarray}}
\newcommand{\beno}{\begin{eqnarray*}}
\newcommand{\eeno}{\end{eqnarray*}}
 \newcommand{\ba}{\begin{aligned}}
 \newcommand{\ea}{\end{aligned}}

\begin{document}


\begin{abstract}
We investigate a regularity for weak solutions of  the following generalized Leray equations
\begin{equation*}
(-\Delta)^{\alpha}V- \frac{2\alpha-1}{2\alpha}V+V\cdot\nabla V-\frac{1}{2\alpha}x\cdot \nabla V+\nabla P=0,
\end{equation*}
which arises from the study of self-similar solutions to the generalized Naiver-Stokes equations in $\mathbb R^3$.
Firstly, by making use of the vanishing viscosity and developing non-local effects of the fractional diffusion operator, we prove uniform estimates for weak solutions $V$ in the weighted Hilbert space $H^\alpha_{\omega}(\R^3)$. Via the differences characterization of Besov spaces and the bootstrap argument, we  improve the regularity for weak solution from $H^\alpha_{\omega}(\R^3)$ to $H_{\omega}^{1+\alpha}(\R^3)$.
This regularity result, together linear theory for the non-local Stokes system, lead to pointwise estimates
of $V$ which allow us to obtain a natural pointwise property of the self-similar solution constructed in \cite{LXZ}. In particular, we obtain an optimal decay estimate of the self-similar solution to the classical Naiver-Stokes equations by means of the special structure of Oseen tensor.
 This answers the question proposed by  Tsai \cite[Comm. Math. Phys., 328 (2014), 29-44]{T}.

\end{abstract}
\maketitle

\section{Introduction and main results}

Several works have been devoted to  the study of the existence  and the partial regularity of weak solutions to the  three dimensional incompressible Navier-Stokes equations with the fractional diffusion in  $\R^{3}\times (0,+\infty)$
\begin{align}\label{NS}
\left\{
\begin{aligned}
&u_{t}+(-\Delta)^{\alpha} u+u\cdot\nabla u+\nabla p=0, \\
&\mbox{div}\, u=0,\\
\end{aligned} \right.
\end{align}
complemented with the initial condition
\begin{equation}\label{NS-i}
u(\cdot,0)=u_{0}\ \ \ \ \ \mbox{in}\ \ \ \R^{3}.
\end{equation}
When $0<\alpha<1$, the fractional Laplacian $(-\Delta)^{\alpha}$ corresponds to the L\'{e}vy operator
$$
(-\Delta)^{\alpha} u(x)\triangleq c_{\alpha} {\rm p.v.}\int_{\R^{3}}\frac{u(x)-u(z)}{|x-z|^{3+2\alpha}}\;{\rm d}z,
\;\;\; c_{\alpha}
=\frac{\alpha(1-\alpha)4^{\alpha}\Gamma(\frac{3}{2}+\alpha)}{\Gamma(2-\alpha)\pi^{\frac{3}{2}}}.$$
From the stochastic process point of view, the fractional Laplacian
$(-\Delta)^{\alpha}$ is an infinitesimal generator of the spherically
symmetric $2\alpha$-stable L\'{e}vy process, denoted by $X_{t}$,
which possesses the self-similarity properties such that
$$
X_{1}=t^{-\frac{1}{2\alpha}}X_{t}\qquad \ \ \mbox{for all}\ \ t>0.
$$

In the framework of stochastic process,  \eqref{NS} can be deduced  via the following stochastic representation
\cite{Cons,Zhang}
\begin{equation*}
\left \{\begin{aligned}
&X_{t}(x)=x+\int_{0}^{t}u(X_{s}(x),s)\,{\rm d}s+\sqrt{2}{\rm B}_{t},\quad\; t\geq0, \\
&u(\cdot,t)=\mathbb{P}\mathbb{E}\left[(\nabla^{t}X_{t}^{-1})(u(\cdot,0)\circ X_{t}^{-1})\right],
\end{aligned}\right.
\end{equation*}
where $\textrm{B}_{t}$ is a three-dimensional Brownian motion,
$X_{t}^{-1}(x)$ denotes the inverse of the mapping $x\rightarrow X_{t}$, and
$$\mathbb P={\rm Id}-\nabla (-\Delta)^{-1}{\rm div}$$
is called the Helmholtz projection onto the divergence-free vector fields.

System \eqref{NS} was first proposed by Frisch-Lesieur-Brissaud
\cite{Fr} in the study of Markovian random coupling model for turbulence,
and it  was used to describe a fluid motion with
internal friction interaction by physicists in \cite{Me}. From the viewpoint of PDEs, the fractional  Navier-Stokes equations \eqref{NS}  corresponds to
 a hyper-dissipative (or hypo-dissipative) model for the case of $\alpha>1$ (or
$\alpha<1$).  For $\alpha\geq \frac{5}{4}$, the system \eqref{NS}-\eqref{NS-i}  is subcritical under the energy scaling,  and both in $\R^{3}$ and in a periodic setting $\mathbb{T}^{3}$, has been proved to be global  well-posedness for smooth initial data with decay in infinite, for details see \cite{Lions,Mat}.
For  $\alpha\in [0,\frac{5}{4})$, the global well-posedness of smooth solution for \eqref{NS}-\eqref{NS-i} is still an open problem.
 Usually the $\alpha=\frac{5}{4}$ is  called Lions' critical exponent. For some hypo-dissipative model,   Colombo-De Lellis-De Rosa in \cite{Colombo}, by making use of   the  convex integration methods introduced by De Lellis-Szekelyhidi Jr. in \cite{De Lellis},   shown the ill-posedness of Leray-Hopf solutions to hypodissipative Navier-Stokes system \eqref{NS}-\eqref{NS-i} with $\alpha<\frac{1}{5}$ under the  periodic setting $\mathbb{T}^{3}$.  Subsequently,  De Rosa \citep{Rosa} improved Colombo-De Lellis-De Rosa's results by extending  $\alpha<1/5$ to $\alpha<1/3$. However, for the dissipative model in range $\alpha\in [\frac{1}{3},\frac{5}{4})$, the uniqueness of the energy week solution  for \eqref{NS}-\eqref{NS-i} is still an open problem. Besides,
 an analogue of the Caffarelli-Kohn-Nirenberg result was established in \cite{Colombo-1, Katz} for the hyper-dissipative range $\alpha\in (1,\frac{5}{4})$, and in \cite{Tang} for the hypo-dissipative range $\alpha\in(\frac{3}{4},1)$.

The motivation of the present paper arises from  the study on the regularity and asymptotics   of a velocity field $u(x,t)$ and pressure $p(x,t)$ of  self-similar form:
$$
u(x,t)= \lambda^{2\alpha-1}(t)U\big(\lambda(t)x\big),\ \  p(x,t)=\lambda^{2(2\alpha-1)}(t)P\big(\lambda(t)x\big).
$$
where $\lambda(t)=t^{-\frac{1}{2\alpha}}$ with  $t>0$ (or $\lambda(t)=(-t)^{-\frac{1}{2\alpha}}$ with  $t<0$)
associate with  forward self-similar solutions (or backward self-similar solutions, resp.) to system \eqref{NS}.  One easily verifies that the profile pair $(U,P)$  fulfills either the generalized Leray equations in $\R^{3}$
\begin{align}\label{E-forwad}
\left\{
\begin{aligned}
&(-\Delta)^{\alpha}U- \frac{2\alpha-1}{2\alpha}U-\frac{1}{2\alpha}x\cdot \nabla U+\nabla P+U\cdot\nabla U=0 \\
&\mbox{div}\, U=0\\
\end{aligned}\ \right.
\end{align}
or
\begin{align}\label{E-back}
\left\{
\begin{aligned}
&(-\Delta)^{\alpha}U+\frac{2\alpha-1}{2\alpha}U+\frac{1}{2\alpha}x\cdot \nabla U+\nabla P+U\cdot\nabla U=0 \\
&\mbox{div}\, U=0\\
\end{aligned}\ \right..
\end{align}

For the classical Naiver-Stokes equations \eqref{NS}, the question on the existence of backward self-similar solutions was already raised by Leray \cite{Ler} in 1934 who observed that if such a nontrivial solution exists then it would necessarily lead to the phenomenon of finite-time blow up.
Ne\v{c}as-R{\aa u}\v{z}i\v{c}ka-\v{S}ver\'{a}k \cite{NRS}  proved that any $U\in L^{3}(\R^{3})$ which solves \eqref{E-back} has to vanish. Later, this celebrated  result was extended to $U\in L^{q}(\R^{3})$ for some $q\in(3,+\infty)$ by Tsai~\cite{Tsai}. However, the story on  forward self-similar solutions is different.
Using harmonic analysis method together with  perturbation argument,  Cannone-Meyer-Planchon \cite{Cannon-Meyer-Planchon, C2} firstly proved the  existence and uniqueness of the small forward self-similar solutions in the framework of  homogeneous Besov spaces, see also Koch and Tataru \cite{Koch} in $\text{BMO}^{-1}(\R^3)$.
It is necessary to point out that the perturbation argument can not work for large forward self-similar solutions.
  Jia and \v{S}ver\'{a}k \cite{JS} constructed a scale-invariant solution with large initial values by developing so
  called local-in-space regularity near the initial time.

In this paper, we mainly focus on  the regularity and asymptotics  of the forward self-similar solutions
 of system \eqref{NS}. Due to singularity arising from  self-similarity, we can not directly construct a solution to \eqref{E-forwad} in the Sobolev spaces. Thus, one decomposes
$$
U=U_{0}+V
$$
and considers the difference part $V$, which  satisfies in $\R^{3}$
\begin{align}\label{E}
\left\{
\begin{aligned}
&(-\Delta)^{\alpha}V- \frac{2\alpha-1}{2\alpha}V(x)-\frac{1}{2\alpha}x\cdot \nabla V+\nabla P=F_{0}+F_{1}(V), \\
&\textnormal{div}\,V=0,
\end{aligned}\ \right. ,
\end{align}
where
\begin{equation*}\left\{
\begin{aligned}&F_{0}=-U_{0}\cdot\nabla U_{0},\ \  F_{1}(V)=-(U_{0}+V)\cdot\nabla V-V\cdot\nabla U_{0},\\
&U_{0}=e^{-(-\Delta)^{\alpha}}u_{0}.
\end{aligned}\right.\qquad\end{equation*}
The readers will find that compared with nonlinearity $F_{1}(V)$, the force $F_{0}$ plays a dominant role in  the study of the decay estimates of $V$.

It is clear that finding a forward self-similar solution to system \eqref{NS}-\eqref{NS-i} is equivalent to prove the existence of the profile $V(x)$ to \eqref{E}.
For $\alpha=1$, Jia and \v{S}ver\'{a}k in \cite{JS} proved  \eqref{E} admits at least one solution $V(x)$
 by developing the local-in-space regularity near the initial time and the {\em topological degree} method under the framework of the weighted H\"{o}lder space.
 It is obvious from
 $$
 u(x,t)=\frac{1}{\sqrt{t}}(V+U_{0})\Big(\frac{x}{\sqrt{t}}\Big)
 $$
  that the asymptotic behaviour of $V$ at infinity is closely related to the regularity of $u$ near $|x|=1,t=0$.
  On the other hand, the local regularity of $u_{0}$ propagates for a short time, this fact helps us to obtain better  decay estimate of $V$ at infinite if $u_{0}$ is smooth in $\R^{3}\setminus\{0\}$, see \cite{JS}. That is to say, the asymptotics of $V$ at infinity essentially depend on the regularity of the initial data $u_{0}(x)$. To be precise, Jia and \v{S}ver\'{a}k in \cite{JS} proved for $0<\gamma<1$
\begin{equation}\label{add-01}
|V|(x)\lesssim \left \{
\begin{aligned}
&(1+|x|)^{-2}\ \ \ &\mbox{if}\ \ u_{0}(x)\in C_{{\rm loc}}^{\gamma}(\R^{3}\setminus\{0\}),\\
&(1+|x|)^{-3}\log(2+|x|) \ \ \  &\mbox{if}\ \ u_{0}(x)\in C_{{\rm loc}}^{1,\gamma}(\R^{3}\setminus\{0\}),\\
&(1+|x|)^{-3} \ \ \  &\mbox{if}\ \ u_{0}(x)\in C_{{\rm loc}}^{3,\gamma}(\R^{3}\setminus\{0\}),
\end{aligned}
\right.
\end{equation}
 see also \cite{T}.  When $u_{0}(x)\in C_{{\rm loc}}^{1,\gamma}(\R^{3}\setminus\{0\})$, there is a logarithmic  loss caused by roughly potential estimates, to be more precise
$$
|V|(x)\leq C\int_{0}^{1}\int_{\R^{3}}\mathcal{O}(x-y,1-s)s^{-\frac{3}{2}}U_{0}\cdot\nabla U_{0}(y/\sqrt{s})\,{\rm d}y{\rm d}s
\leq C(1+|x|)^{-3}\log(2+|x|).
$$
When $u_{0}(x)$ is smooth enough away from 0, then  $V$  possesses the optimal pointwise decay
$$|V|\lesssim |U_{0}\cdot\nabla U_{0}|\sim (1+|x|)^{-3}, \qquad\quad \forall \;\; |x|\gg1,$$
by making use of  local-in-space regularity near the initial time. In this case, one can  obtain local regularity estimate
$$
\big\|\partial_{t}\big(u(x,t)-e^{t\Delta}u_{0}\big)\big\|_{L^{\infty}(\mathbb B_{1/2}(x_{0})\times[0,T])}\leq C(u_{0})
 $$
for some $T\ll1$ and $|x_{0}| >1$.  This estimate together  with  scaling invariance leads  to
 the optimal decay  without using potential estimates,
 and thus  avoid the logarithmic loss, for details  see \cite[Theorem 4.1]{JS}. It was asked in \cite{T} that one may be able to obtain the optimal decay  of $V$ under the low regularity condition on $u_{0}$.

Later, Korobkov-Tsai \cite{KT} used the blow-up argument to construct a solution to Leray equation \eqref{E} for $\alpha=1$.
The advantage of this method can be used to deal with the case of half-space, but this method
can not provide the pointwise decay estimates. Recently,  Lai-Miao-Zheng \cite{LXZ} used the $L^{2}$ weighted estimate to prove the solution constructed in \cite{KT} satisfying the pointwise decay.
   Besides, for the hypo-dissipative case, i.e., $\frac{5}{8}<\alpha<1$, the authors  proved an existence of  weak solution to \eqref{E} in Sobolev space  $H^{\alpha}(\R^{3})$, which is stated as follows:

\begin{theorem}[Theorem 3.6, \cite{LXZ}]\label{thm-1}
Let $\alpha\in(\frac58,1]$.	Then system \eqref{E} admits at least one weak  solution $V\in H^{\alpha}(\R^{3})$ defined (see  Definition \ref{def-weak} below) such that
$$
\|V\|_{H^{\alpha}(\R^{3})}\leq C(U_{0}).
$$
Moreover,  if $u_{0}\in C^{0,1}(\R^{3}\setminus\{0\})$, then for $\alpha=1$
\begin{equation}\label{add-02}
|V|(x)\leq C(1+|x|)^{-3}\log(2+|x|).
\end{equation}
\end{theorem}

\vskip0.35cm

The keypoint of constructing  a solution  in \cite{LXZ} by making use of the {\em topological degree} theory
consists of  the blowup argument and viscous approximation method.
In contrary to the framework of weighted H\"{o}lder space used in \cite{JS}, Theorem \ref{thm-1} only provide uniformly $H^{\alpha}(\R^{3})$ bound of solution $V$,
  and donot provide any information on pointwise behavior for the case $\alpha<1$.  By the way,  the authors also established
  a higher local regularity of $V$ via the classical elliptic regularity theory and bootstrap technique in \cite{LXZ}.

The main goal includes two aspects, one is to establish the global regularity of the weak solution constructed in
Theorem \ref{thm-1} in weighted Sobolev space.  Another is to establish  the decay estimates of $V$ at infinity according to the global regularity established by the previous step.
 This helps us to recover the natural pointwise decay estimate of self-similar solutions to
generalized Navier-Stokes equations \eqref{NS}.
In particular, we obtain the optimal decay estimate of the self-similar solution to the
classical Naiver-Stokes equations under the lower regularity condition for $u_0$.

\begin{theorem}\label{thm1.2}
	    Let $\alpha\in(\frac56,1]$ and $V\in H_{\sigma}^{\alpha}(\R^{3})$ be the weak solution  established in Theorem~\ref{thm-1}.
		Then  $V\in H^{1+\alpha}_{\omega}(\R^3)$ $($ see Definition \ref{def-2} below$)$ satisfies
		\begin{equation*}
	\big	\|\sqrt{1+|\cdot|}V\big\|_{H^{1+\alpha}(\R^{3})}\leq C(U_{0}).
		\end{equation*}
 As a product, one has
 $$|V(x)|\leq  C(U_{0})(1+|x|)^{-\frac{1}{2}}.$$
 The Hilbert space  $H_{\sigma}^{s}(\R^{3})$ with $s\in\mathbb{R}$ consists of tempered distributions $u$ such that $\textnormal{div}\,u=0$, endowed with the following norm
 $$
 \|u\|^2 _{H_{\sigma}^s(\mathbb{R}^3)}= \int_{\mathbb{R}^3} \left(1+|\xi|^2\right)^s|\hat{u}(\xi)|^2\,\mathrm{d}\xi<\infty,
 $$
where  $\hat{u}$ denotes the Fourier transform of $u$.

 \end{theorem}

\begin{remark}
The restriction on $\alpha\in (5/6,1]$ is caused by convection  term $V\cdot\nabla V$.
The fact $V\in H^{\alpha}(\R^{3})$ implies that $V\cdot\nabla V\in W^{-1,\frac{3}{3-2\alpha}}(\R^{3})$.
 Roughly speaking, this fact  together with  the elliptic regularity theory yields that
$$
V\in W^{2\alpha-1,\frac{3}{3-2\alpha}}(\R^{3})\hookrightarrow L^{\frac{6}{3-2\alpha}}(\R^{3})
$$
where the imbedding need condition $\alpha>\frac{5}{6}$.
It is natural to ask whether the similar result  holds for the critical case $\alpha=5/6$.
\end{remark}

 To illuminate the motivations of this paper in detail, we sketch the proof  of Theorem~\ref{thm1.2}.
Compared with the case $\alpha=1$, the main obstacle in establishing global regularity originates  from the fractional diffusion operator.
To overcome this difficulty, with the method of vanishing viscosity, we will establish $H_{\omega}^\alpha(\R^3)$-estimate for weak solution $V$ by  choosing the suitable test function  $\varphi=\sqrt{\frac{1+|x|}{1+\varepsilon|x|^2}}V$ in the weak sense. Next, in view of difference characterisation  of  Besov  space, we improve the regularity of $V$ to a new level
$$
\big\| V\big\|_{B_{2,\infty}^{2\alpha}(\R^3)}<C\left(U_0\right),
$$
which  affords us  the $L^2$-bound of $\nabla V$. This bound can help us to choose
$
-D_k^{-h}\big(g^2_\varepsilon D^h_kV\big)$ with  $g_\varepsilon=\frac{1}{\sqrt{1+\varepsilon|x|^2}}
$  as the test function $\varphi$ in  equality \eqref{eq.weak-ESTIMATE} to obtain $V\in H^{1+\alpha}(\R^3).$ With this regularity in hand, the bootstrap argument enables us to conclude that
\begin{equation*}
\big\|V\big\|_{H_{\omega}^{1+\alpha}(\R^3)}<C\left(U_0\right),
\end{equation*}
by developing commutator estimates.

\smallskip
The second main result of this work is to improve the order of decay of $V$ at infinity by
establishing the corresponding linear theory for the non-local Stokes system with the singularity force and the basic properties of the non-local Oseen kernel.
In particular, for the classical Navier-Stokes equations, by means of  the special structure of Oseen kernel with the decay estimates of the second order derivatives of $V$,
 we can remove logarithmic loss, and establish  the  optimal decay estimates of $V$ under the lower regularity condition for $u_{0}$.
It is stated as follows.
\begin{theorem}\label{T1.3}
Let  $u_{0}=\frac{\sigma(x)}{|x|^{2\alpha-1}} \text{ with }\sigma(x)=\sigma(x/|x|)\in C^{0,1}(\mathbb{S}^{2})$, and $\frac{5}{6}<\alpha\leq1$, which satisfies $\mathrm{div}\, u_{0}=0$ in $\R^{3}$ in the distribution sense, then the solution $V(x)$ constructed by Theorem \ref{thm-1} fulfills
$$
|V|(x)\leq C(1+|x|)^{2-4\alpha}.
$$
Moreover, if $\alpha=1$, we can obtain the optimal
 decay estimate of $V$ under the condition $\sigma(x)\in C^{1,1}(\mathbb{S}^{2})$, {\rm i.e.}
$$
|V|(x)\leq C(1+|x|)^{-3}.
$$
\end{theorem}

\begin{remark}  From the proof of Theorem \ref{T1.3}, it follows that the solution $V$ constructed by Jia and \v{S}ver\'{a}k in \cite{JS} possess  the optimal
 decay estimate under the lower regularity condition i.e.,  $u_{0}\in C^{1,1}(\R^{3}\setminus\{0\})$.  This answers the question proposed by  Tsai in \cite{T}.
\end{remark}

\begin{remark}
The essential ingredient in establishing the optimal decay estimate  is how to gain  the decay estimate of the second order derivative of $V$.
It seems to us that for $5/6<\alpha<1$, it is impossible to  obtain any  pointwise estimate of $D^{2}V$,
this fact prevents us from establishing the optimal decay of $V$ for the generalized Leray equations.
\end{remark}

The proof of Theorem \ref{T1.3} goes roughly as follows.  First, we establish the corresponding linear theory for the non-local Stokes system with the singularity force,  which is of independent interest.  Thus, we can represent $V$ as
\begin{equation}\label{V-Represent}
V(x)=\int_{0}^{1}\int_{\R^{3}}\mathcal{O}(x-y,1-s)s^{\frac{1}{2\alpha}-2}\{F_{0}+F_{1}(V)\}(y/s^{\frac{1}{2\alpha}})\,{\rm d}y{\rm d}s,
\end{equation}
where $\mathcal{O}$ is the Oseen kernel of the the non-local Stokes operator.
Combining this representation with the roughly decay of $V$ obtained in Theorem \ref{thm1.2},
we can improve the decay estimate of $V$ to
$$
|V(x)|\leq C(1+|x|)^{2-4\alpha}.
$$
Next, we will remove the logarithmic  loss in \eqref{add-02} to obtain the optimal decay estimate of $V$  for $\alpha=1$.  Here, we decompose  \eqref{V-Represent} into
$$
\int_{0}^{1}\int_{|y|<|x|/2}\cdots\,{\rm d}y{\rm d}s+\int_{0}^{1}\int_{|x|/2<|y|<2|x|}\cdots\,{\rm d}y{\rm d}s+\int_{0}^{1}\int_{|y|>2|x|}\cdots\,{\rm d}y{\rm d}s.
$$
The latter two terms are easily controlled by $(1+|x|)^{-3}$, but  the first term is controlled by
$$
\int_{0}^{1}\int_{|y|<|x|/2}(|x-y|+\sqrt{1-s})^{-3}(\sqrt{s}+|y|)^{-3}\,{\rm d}y{\rm d}s\lesssim (1+|x|)^{-3} \log (2+|x|)
$$
by traditional argument which causes a logarithmic loss.
To avoid this loss, we first  obtain  the decay estimate as
$$
|D^{2}V|\leq C(1+|x|)^{-3+\delta}, \;\;\;\text{for}\;\;0<\delta\ll 1$$
by establishing the regularity estimate of the second order derivative of $V$ in weighted Sobolev space.
 This estimate, together with the special structure of the Oseen tensor,
 finally leads to the optimal controlled estimate on  the first term.  We finally  obtain the optimal estimate of $V$:
$$
|V(x)|\thicksim |U_{0}\cdot\nabla U_{0}|\leq C(1+|x|)^{-3}.
$$

As a direct application of Theorem \ref{T1.3}, one can  use Lemma \ref{Lfreeterm}  to  obtain the natural pointwise property of the
self-similar solution constructed in \cite{LXZ}.

\begin{corollary}
	Assume $\frac{5}{6}<\alpha\leq1$. Let  $u_{0}=\frac{\sigma(x)}{|x|^{2\alpha-1}} \text{ with }\sigma(x)=\sigma(x/|x|)\in C^{0,1}(\mathbb{S}^{2})$, which satisfies $\mathrm{div}\, u_{0}=0$ in $\R^{3}$ in the distribution sense. Then problem \eqref{NS} admits at least one forward self-similar
	solution $u\in \mathrm{BC}_{w}\big ([0,+\infty),\,L^{\frac{3}{2\alpha-1},\infty}(\R^3)\big)$  such that
		\begin{equation*}
		|u(x,t)|\leq\frac{C}{(|x|+t^{\frac{1}{2\alpha}})^{2\alpha-1}}\quad\text{and}\quad \left|u(x,t)-e^{-t(-\Delta)^{\alpha}}u_{0}\right|\leq \frac{Ct^{\frac{2\alpha-1}{2\alpha}}}{(|x|+t^{\frac{1}{2\alpha}})^{4\alpha-2}}
		\end{equation*}
		for all $(x,t)\in \R^{3}\times(0,+\infty)$. Moreover, if $\alpha=1$ and $\sigma(x)\in C^{1,1}(\mathbb S^{2})$, we have
$$
\left|u(x,t)-e^{t\Delta}u_{0}\right|\leq \frac{Ct}{(|x|+t^{\frac{1}{2}})^{3}}\ \ \mbox{for all}\ \  (x,t)\in \R^{3}\times(0,+\infty).
$$
\end{corollary}

Now we conclude this section with a definition of a the weak solution of the generalized Leray equations \eqref{E} and some notations.

\begin{definition}[Weak Solutions]\label{def-weak}
	We say that the couple $(V,\,{P})$ is a weak solution to problem~\eqref{E} if
	\begin{itemize}
		\item $V\in  H^\alpha(\R^3)$ and ${P}\in L^{2}(\R^3) ;$
		\item $(V,\,{P})$ satisfies \eqref{E} in the sense of distribution in $\R^3,$ i.e.  for all vector fields $\varphi\in H^1(\mathbb{R}^3)$ satisfying $\big\||\cdot |\varphi(\cdot)\big\|_{H^1(\mathbb{R}^3)}<+\infty,$
		the couple $(V,P)$ fulfills
		\begin{align}\label{eq.weak-ESTIMATE}
		\begin{split}
		&\int_{\mathbb{R}^3}(-\Delta)^\frac{\alpha}{2} V:(-\Delta)^\frac{\alpha}{2}\varphi\,\mathrm{d}x+\frac{1}{2\alpha}\int_{\mathbb{R}^3}x\cdot \nabla\varphi\cdot V\,\mathrm{d}x-\frac{\alpha-2}{\alpha}\int_{\mathbb{R}^3} V\cdot\varphi\,\mathrm{d}x\\
		=&\int_{\mathbb{R}^3}P\, \mathrm{div}\,\varphi\,\mathrm{d}x+\int_{\mathbb{R}^3}V\cdot \nabla \varphi\cdot V\,\mathrm{d}x-\int_{\mathbb{R}^3} V\cdot U_0\cdot\nabla\varphi\,\mathrm{d}x-\int_{\mathbb{R}^3}(V+U_0)\cdot \nabla  U_0\cdot \varphi\,\mathrm{d}x.
		\end{split}
		\end{align}
	\end{itemize}
\end{definition}

\noindent\textbf{Notation:} We first agree that  $\langle x\rangle =(1+|x|^2)^{\frac12}$ and $\Lambda=\sqrt{-\Delta}$. We define
\[\mathbb{B}_R(x)=\big\{y\in\R^3\big|\,|x-y|<R\big\}\]
and $\mathbb{B}^c_R(x)=\R^3 \backslash \mathbb{B}_R(x).$
Define the fractional difference quotient  $\vartriangle_k^{h^\alpha}u$ by
$$\vartriangle_k^{h^\alpha}u:=\frac{u(x+h\mathbf{e}_k)-u(x)}{|h|^\alpha}\qquad h\in\mathbb{R},\,\,h\neq0.$$
As usual, we define the commutator as
$$
[A,B]=AB-BA.
$$

This paper is organized as follows: In the next section, we give some preliminary lemmas which will be used in the forthcoming proof. In Section 3, we  devote to  the proof of Theorem \ref{thm1.2}. Finally, decay estimates of $V$ are established in Section 4.

\section{Preliminaries}
\setcounter{equation}{0}

In this section, let's us begin with the so-called Littlewood-Paley decomposition , see e.g.,  \cite{BCD11}.
Suppose that  $(\chi,\varphi)$ be a couple of smooth functions with  values in $[0,1]$
such that  $\text{supp}\,\chi\subset \big\{\xi\in\mathbb{R}^{3}\big||\xi|\leq\frac{4}{3}\big\}$,
$\text{supp}\,\varphi\subset\big\{\xi\in\mathbb{R}^{3}\,\big|\,\frac{3}{4}\leq|\xi|\leq\frac{8}{3}\big\}$ and
\begin{equation*}
\chi(\xi)+\sum_{j\in \mathbb{N}}\varphi(2^{-j}\xi)=1\qquad \forall\,\xi\in \mathbb{R}^{3}.
\end{equation*}
For any $u\in \mathcal{S}'(\mathbb{R}^{3})$, let us define
\begin{equation*}
\Delta_{-1}u\triangleq\chi(D)u\quad\text{and}\quad   {\Delta}_{j}u\triangleq\varphi(2^{-j}D)u\qquad \forall\,j\in\mathbb{N}.
\end{equation*}
Moreover, we can define the low-frequency cut-off:
\begin{equation*}
{S}_{j}u\triangleq\chi(2^{-j}D)u.
\end{equation*}
So, we easily find  that
\begin{equation*}
u=\sum_{j\geq-1}{\Delta}_{j}u\qquad \text{in}\quad\mathcal{S}'(\mathbb{R}^{3}),
\end{equation*}
which  corresponds to  the \emph{inhomogeneous Littlewood-Paley decomposition}.
In usual, we always use the following properties of quasi-orthogonality:
\begin{equation*}
{\Delta}_{j}{\Delta}_{j'}u\equiv 0\quad \text{if}\quad |j-j'|\geq 2.
\end{equation*}
\begin{equation*}
{\Delta}_{j}({S}_{j'-1}u{\Delta}_{j'}v)\equiv0\quad \text{if}\quad |j-j'|\geq5.
\end{equation*}
We shall also use the \emph{homogeneous Littlewood-Paley} operators governed by
\begin{equation*}
\dot{S}_{j}u\triangleq\chi(2^{-j}D)u \quad\text{and}\quad {\Delta}_{j}u\triangleq\varphi(2^{-j}D)u\qquad\forall\,j\in\mathbb{Z}.
\end{equation*}
We denoted by  $\mathcal{S}_{h}'(\mathbb R^3)$ the space of tempered distributions $u$ such that
\begin{equation*}
\lim_{j\rightarrow-\infty}\dot{S}_{j}u=0\quad \text{in} \quad \mathcal{S}'(\mathbb R^3).
\end{equation*}
The \emph{homogeneous Littlewood-Paley decomposition}  can be written as
\begin{equation*}
u=\sum_{j\in\mathbb Z}{\Delta}_{j}u,\quad \text{in}\quad\mathcal{S}'_h(\mathbb{R}^{3}).
\end{equation*}
 Based on the above decomposition in frequency space, we will give definition of the  homogeneous Besov space $	\dot{B}^{s}_{p,q}(\mathbb{R}^{3})$.

\begin{definition}\label{def2.2}
	Let $s\in \mathbb{R}$, $(p,q)\in [1,+\infty]^{2}$ and $u\in \mathcal{S}'_h(\mathbb{R}^{3})$.
Then  the \emph{ homogeneous Besov spaces} can be defined as
	\begin{equation*}
	 \dot{B}^{s}_{p,q}(\mathbb{R}^{3})\triangleq\big\{u\in\mathcal{S}'(\mathbb{R}^{3})\big|\,\,\|u\|_{\dot{B}^{s}_{p,q}(\mathbb{R}^{3})}<+\infty\big\},
	\end{equation*}
	where
	\begin{equation*}
	\|u\|_{{\dot{B}}^{s}_{p,q}(\mathbb{R}^{3})}\triangleq\begin{cases}
	\Big( \displaystyle{\sum_{j\in\mathbb Z}}\,2^{jsq}\|{\Delta}_{j}u\|_{L^{p}(\mathbb{R}^{3})}^{q}\Big)^{\frac{1}{q}}
	\quad&\text{if}\quad q<+\infty,\\
	\displaystyle{\sup_{j\in\mathbb Z}}\,2^{js}\| {\Delta}_{j}u\|_{L^{p}(\mathbb{R}^{3})}\quad&\text{if}\quad q=+\infty.
	\end{cases}
	\end{equation*}
\end{definition}
Now we recall some useful properties of the  homogeneous Besov space $	\dot{B}^{s}_{p,q}(\mathbb{R}^{3})$ defined in Definition \ref{def2.2} .
\begin{lemma}[\cite{BCD11, Miao-book}\label{lem-equi}]
There hold that
	\begin{itemize}
		\item[(i)] For $
		1 \leqslant p_{1} \leqslant p_{2} \leqslant \infty,$ $1 \leqslant r_{1} \leqslant r_{2} \leqslant \infty$ and $ s \in \mathbb{R}$,
		\[\dot{B}_{p_{1}, r_{1}}^{s} (\R^3)\hookrightarrow \dot{B}_{p_{2}, r_{2}}^{s-3\left(\frac{1}{p_{1}}-\frac{1}{p_{2}}\right)} (\R^3).\]
		\item [(ii)]For $s<\frac{3}{2},$ Besov spaces $\dot{B}^s_{2,2}(\R^3)$ coincide  with Hilbert spaces $\dot{H}^s(\R^3).$
		\item[(iii)] For $s\in (0,1)$ 	and $p,\,r\in[1,\infty]$,
		$$
		C^ {-1} \|u\|_{\dot {B}_ {p,r} ^ {s}(\R^3) }
		\leqslant \left\| \frac { \left\| \tau_{-h }u-u\right\|_{ L^p(\R^3 }} {|h|^ {s} } \right\|_ {L ^ {r} \left( \mathbb{R}^{3}, \frac{\textnormal{d}h}{|h|^{3}}\right)}\leq C\| u\|_{\dot {B}_{p,r}^{s}(\R^3)},$$
		where $\tau_{h}u = u(x-h).$
	\end{itemize}
\end{lemma}
Next, we introduce the definitions of the weighted Hilbert space. Firstly we recall the weighted $L^2$-space.
For a non-negative locally integrable function  $\|\cdot\|_{2, \,{\omega}}$  denotes the norm on  $ L_{\omega}^{2}(\R^3)$,  i.e.,
\begin{equation*}\|g\|_{2, \,\omega}=\left(\int_{\mathbb{R}^{3}}|g(x)|^{2} \omega(x) \,\mathrm{d} x\right)^{1 / 2},\end{equation*} where $\omega(x)$ is the nonzero weighted function.
\begin{definition}\label{def-2}
	Let $s>0,$ and $\omega=\langle x\rangle$. Then we define the weighted  space $H_{\omega}^s(\R^3)$ as
	\[H^s_{\omega}(\R^3)\triangleq\big\{u\in\mathcal{S}'(\mathbb{R}^{3})\,\big|\; \|u\|_{H^s_{\omega}(\R^3)} <+\infty\big\},\]
	where
	\[\|u\|_{H^s_{\omega}(\R^3)}\triangleq \|u\|_{L^2_{\omega}(\R^3)}+\big\|\Lambda^su\big\|_{L^2_{\omega}(\R^3)}.\]
\end{definition}
Such a weighted Hilbert space enjoys the following properties which provided a working framework for weak solutions. Before stating it, we introduce the commutator between an operator $\Lambda^\alpha$ and a function $\phi$ defined by the formula
\[[\Lambda^\alpha, \phi] f=\Lambda^\alpha(\phi f)-f \Lambda^\alpha(\phi).\]
\begin{lemma}\label{lem-Comm}
	Let $\alpha\in(0,1),\, f\in L^2(\R^3)$ and $\phi\in \dot{C}^{\beta}(\R^3)\cap \dot{W}^{1,\infty}(\R^3)$, There exists a constant $C>0$ such that for each $ \beta\in[0,\alpha),$
	\[\big\|\left[\Lambda^\alpha,\,\phi\right]f\big\|_{L^2(\R^3)}\leq C\max\left\{\|\phi\|_{\dot{C}^{\beta}(\R^3)},\,\|\phi\|_{\dot{W}^{1,\infty}(\R^3)}\right\}\|f\|_{L^2(\R^3)}.\]
\end{lemma}
\begin{proof}
	Thanks to the definition of the fractional operator, we  write
	$$\left[\Lambda^\alpha,\,\phi(x)\right]f=c_\alpha
	\int_{\R^3}\frac{\phi(x)-\phi(y)}{|x-y|^{3+\alpha}}f(y)\,\mathrm{d}y.$$
	Moreover, we  have by using $\phi\in \dot{C}^{\beta}(\R^3)\cap \dot{W}^{1,\infty}(\R^3)$ that
	\begin{align*}
	\big|\left[\Lambda^\alpha,\,\phi(x)\right]f\big|\leq&c_\alpha
	\int_{\R^3}\frac{\big|\phi(x)-\phi(y)\big|}{|x-y|^{3+\alpha}}|f|(y)\,\mathrm{d}y\\
	\leq&c_\alpha	\int_{\R^3}\frac{\min\big\{\|\phi\|_{\dot{C}^{\beta}(\R^3)} |x-y|^{\beta},\,\|\phi\|_{\dot{W}^{1,\infty}(\R^3)}|x-y|\big \}}{|x-y|^{3+\alpha}}|f|(y)\,\mathrm{d}y.
	\end{align*}
	Since $\beta\in[0,\alpha[,$	we readily have by the Young inequality that
	\begin{equation*}
	\big\|\big[\Lambda^\alpha,\,\phi\big]f\big\|_{L^2(\R^3)}\leq C\max\left\{\|\phi\|_{\dot{C}^{\beta}(\R^3)},\,\|\phi\|_{\dot{W}^{1,\infty}(\R^3)}\right\}\|f\|_{L^2(\R^3)}.
	\end{equation*}
	Thus we complete the proof of the lemma.
\end{proof}
\begin{lemma}\label{Weighted}
	Let $\omega=\langle x\rangle.$  There hold that
	\begin{itemize}
		\item [(i)]  For $s>0$, the weighted space $H^s_{\omega}(\R^3)$ is a Banach space .
		\item [(ii)] For $s\in (\frac12,2)$, there exist a constant $C>0$ such that
		\begin{equation}\label{key-Q}
		C^{-1}\|u\|_{H^s_{\omega}(\R^3)}\leq\left\|\omega u\right\|_{H^s(\R^3)}\leq C\|u\|_{H^s_{\omega}(\R^3)}.
		\end{equation}
	\end{itemize}
\end{lemma}
\begin{proof}
	Since $\omega=\langle x\rangle$ belongs to $A_2$-class, we can show $\rm (i)$ by the standard argument. So we omit its proof.   For $\rm (ii)$, we firstly consider the case  where $s\in(1/2,1).$ By the triangle inequality, we see that
	\begin{equation}\label{112}
	\left\|\omega u\right\|_{H^s(\R^3)}\leq \|u\|_{H^s_\omega(\R^3)}+\big\|[\Lambda^s,\omega]u\big\|_{L^2(\R^3)}.
	\end{equation}
	By Lemma \ref{lem-Comm},  we readily have that $\beta\in[0,s[$,
	\begin{equation}\label{111}
	\big\|[\Lambda^s,\omega]u\big\|_{L^2(\R^3)}\leq C\max\left\{\|\omega\|_{\dot{C}^{\beta}(\R^3)},\,\|\omega\|_{\dot{W}^{1,\infty}(\R^3)}\right\}\|u\|_{L^2(\R^3)}.
	\end{equation}
	So we need to bound both term $\|\omega\|_{\dot{C}^{\beta}(\R^3)}$ and $\|\omega\|_{\dot{W}^{1,\infty}(\R^3)}$. A simple calculation yields
	\begin{equation*}
\nabla \omega=\frac{x}{2(1+|x|^2)^{\frac34}}.
	\end{equation*}
	It follows that for $p>6,$
	$\|\nabla \omega\|_{L^p(\R^3)}<+\infty.$
	This inequality together with the embedding that $\dot{W}^{1,p}(\R^3)\hookrightarrow \dot{C}^{1-\frac{3}{p}}(\R^3)$ enables us to conclude that for each $s\in(1/2,1],$
	\begin{equation}\label{C-s}
\|\omega\|_{\dot{C}^s(\R^3)}<+\infty.
\end{equation}
	Plugging this inequality in \eqref{111} gives
	\begin{equation*}
	\big\|[\Lambda^s,\omega]u\big\|_{L^2(\R^3)}\leq C\|u\|_{L^2(\R^3)}.
	\end{equation*}
Inserting this estimate into \eqref{112}, we immediately get
	\begin{equation*}
\left\|\omega u\right\|_{H^s(\R^3)}\leq \|u\|_{H^s_{\omega}(\R^3)}+\|u\|_{L^2(\R^3)}\leq C \|u\|_{H^s_{\omega}(\R^3)}.
\end{equation*}
Next we consider the case  where $s\in]1,2[.$  Performing the similar fashion, we can show that
\begin{align*}
\left\|\omega u\right\|_{H^s(\R^3)}\leq & \big\|\omega u\big\|_{L^2(\R^3)}+ \big\|\nabla (\omega u)\big\|_{\dot{H}^{s-1}(\R^3)}\\
\leq&\big\|\omega u\big\|_{L^2(\R^3)}+ \big\|\nabla \omega u\big\|_{\dot{H}^{s-1}(\R^3)} +\big\|\omega \nabla u \big\|_{\dot{H}^{s-1}(\R^3)}.
\end{align*}
On one hand, the Leibinz estimate allows us to infer
\begin{align*}
 \big\|\nabla \omega u\big\|_{\dot{H}^{s-1}(\R^3)}\leq C\|\nabla w\|_{L^\infty(\R^3)}\|u\|_{H^s(\R^3)}+C\|u\|_{H^s(\R^3)}\|\omega\|_{\dot{C}^{s}(\R^3)}.
\end{align*}
The same argument as used in \eqref{C-s} allows  us to conclude that $\|\nabla ^2\omega\|_{L^\infty(\R^3)}<+\infty.$ Moreover, we have by the interpolation theorem that
for each $s\in(1/2,2],$
$\|u\|_{\dot{C}^s(\R^3)}<+\infty.
$ Therefore
\begin{align*}
\big\|\nabla \omega u\big\|_{\dot{H}^{s-1}(\R^3)}\leq C\|u\|_{H^s(\R^3)}.
\end{align*}
On the other hand, by Lemma \ref{lem-Comm} again, we obtain
\begin{align*}
\big\|\omega \nabla u \big\|_{\dot{H}^{s-1}(\R^3)}
\leq C\|u\|_{H^s_{\omega}(\R^3)}+ \left\|[\Lambda^{s-1},\omega]\nabla
u\right\|_{L^2(\R^3)}\leq C\|u\|_{H^s_{\omega}(\R^3)}.
\end{align*}
The left inequality in \eqref{key-Q}   can be processed in the same way.
\end{proof}
	
\begin{lemma}\label{Lfreeterm}
Let $u_{0}(x)=\frac{1}{|x|^{2\alpha-1}}\sigma(\frac{x}{|x|})\ \text{with}\ \sigma\in L^{\infty}(\mathbb{S}^{2})$ and
 $G^{(\alpha)}(x,t)$ the heat kernel of the operator
$\partial_{t}+(-\Delta)^{\alpha}$, then
 $$U_{0}(x)\triangleq(G^{(\alpha)}\ast u_{0})(x,t)\Big|_{t=1}\in C^{\infty}(\R^{3})$$
  such that $|U_{0}|\leq C(1+|x|)^{1-2\alpha}$ for $x\in \R^{3}$ and
\begin{equation*}\begin{aligned}
&|\nabla U_{0}|\leq C(1+|x|)^{-2\alpha},\ \;\quad\;\;  \mbox{if}\ \  u_{0}(x)\in C^{0,1}(\R^{3}\setminus\{0\})\\
&|\nabla^{2} U_{0}|\leq C(1+|x|)^{-1-2\alpha},\ \;\; \mbox{if}\ \  u_{0}(x)\in C^{1,1}(\R^{3}\setminus\{0\}).
\end{aligned}\end{equation*}
\end{lemma}
\begin{proof}
 Note that
\begin{equation}\label{G-Estimate}
|\nabla^{k}G^{(\alpha)}(x,1)|\leq C(1+|x|)^{-3-2\alpha-k},\;\; \forall \;\; k\ge 0,
\end{equation}
for proof, see \cite{MYZ}. From this, it is easy to see that
$$
U_{0}(x)=\int_{\R^{3}}G^{(\alpha)}(x-y,1)u_{0}(y)\,{\rm d} y\in C^{\infty}(\R^{3})
$$
Now we need to derive the decay estimates of $U_{0}$. To this end,
we  decompose $U_{0}$ to be
\begin{align*}
U_{0}(x)&=\int_{\R^{3}}G^{(\alpha)}(x-y,1)\varphi(y)\,{\rm d}y=\Big(\int_{|y|\leq\frac{|x|}{2}}+\int_{\frac{|x|}{2}\leq|y|\leq2|x|}
+\int_{|y|\geq2|x|}\Big)G^{(\alpha)}(x-y,1)\varphi(y)\,{\rm d}y\\&
\triangleq\rm{I+II+III}.
\end{align*}
From  \eqref{G-Estimate}, we have for  $|x|>1$
$$
{\rm I}\leq C(1+|x|)^{-3-2\alpha}\int_{0}^{\frac{|x|}{2}}r^{3-2\alpha}\,{\rm d}r\leq C|x|^{1-4\alpha};
$$
\begin{align*}
&{\rm II}\leq C|x|^{1-2\alpha}\int_{0}^{3|x|}(1+r)^{-1-2\alpha}\,{\rm d}r\leq C|x|^{1-2\alpha};\\&
{\rm III}\leq C\int_{|x|}^{\infty}r^{-4\alpha}\,{\rm d}r\leq C|x|^{1-4\alpha}.
\end{align*}
Thus, we obtain
$$
|U_{0}|(x)\leq C(1+|x|)^{1-2\alpha}\ \ \mbox{for}\ \ x\in\R^{3}.
$$
Now we come back to the estimate of $\nabla U_{0}(x)$.
From $|\nabla\sigma(x)|<\infty$, we have $|\nabla u_{0}(x)|\leq C |x|^{-2\alpha}$ for $|x|>1$.
 As in above, we decompose $\nabla U_{0}$ as
\begin{align*}
\nabla U_{0}(x)=&\int_{\R^{3}}\nabla G^{(\alpha)}(x-y,1)\varphi(y)\,{\rm d}y\\
=&\Big(\int_{|y|\leq\frac{|x|}{2}}+\int_{\frac{|x|}{2}\leq|y|\leq2|x|}
+\int_{|y|\geq2|x|}\Big)\nabla_{x} G^{(\alpha)}(x-y,1)\varphi(y)\,{\rm d}y\\
\triangleq&\rm{I+II+III}.
\end{align*}
It is clear that
$$
|{\rm{I}}|\leq C(1+|x|)^{-4\alpha},\ \ \ \  |{\rm III}|\leq C(1+|x|)^{-4\alpha}.\qquad\qquad\qquad\quad
$$
For $\rm{II}$, we have
\begin{align*}
\rm{II}&=\int_{\frac{|x|}{2}\leq|y|\leq2|x|}\nabla_{x} G^{(\alpha)}(x-y,1)\varphi(y)\,{\rm d}y\\&
=\int_{\frac{|x|}{2}\leq|y|\leq2|x|} G^{(\alpha)}(x-y,1)\nabla\varphi(y)\,{\rm d}y
-\int_{\{|y|=\frac{|x|}{2}\}\cup\{|y|=2|x|\}} G^{(\alpha)}(x-y,1)\varphi(y)\vec{n}\,{\rm d}S_{y}\\&
\leq C(1+|x|)^{-2\alpha},
\end{align*}
where $\vec{n}$ is the unit outward normal vector  of the boundary $\big\{|y|=\frac{|x|}{2}\big\}\cup\big\{|y|=2|x|\big\}$. Collecting the estimates of (I-III),
 one has $$|\nabla U_{0}|(x)\leq (1+|x|)^{-2\alpha}.$$
  Similar calculation as above can lead to
  $$|\nabla^{2} U_{0}|(x)\leq C(1+|x|)^{-1-2\alpha}.$$
 Here we omit the details.
\end{proof}

Lastly, we recall the special structure of the fundamental solution of the non-local Stokes operator
$e^{-(-\Delta)^{\alpha}t}\mathbb{P}$, the  non-local Oseen kernel  denoted by
$\mathcal{O}(x,t)=(\mathcal{O}_{kj}(x,t))_{1\leq j,k\leq 3}$, which is given as
$$
\mathcal{O}_{kj}(x,t)=\delta_{kj}G^{(\alpha)}(x,t)+\partial_{k}\partial_{j}\int_{\R^{3}}\mathbb{E}(x-y)G^{(\alpha)}(x,t)\,{\rm d}y,\; \;  t>0.
$$
Here and in what follows, we denote by $\mathbb{E}=\frac{1}{4\pi}\frac{1}{|x|}$ the fundamental solution of $-\Delta$ in $\R^{3}$. From the self-similar property of $G^{(\alpha)}(x,t)$, one instantly derives
$$
\mathcal{O}_{kj}(x,t)=t^{-\frac{3}{2\alpha}}\mathcal{O}_{kj}\left(x/t^{\frac{1}{2\alpha}}\right).
$$

\begin{lemma}[\cite{Br, Cao}]\label{L4.1}
	For $(x,t)\in \R^{3}\times(0,\infty)$, the non-local Oseen kernel $\mathcal{O}(x,t)$
	\begin{align}\label{Ossen}
	 \mathcal{O}_{kj}(x,t)=\partial_{k}\partial_{j}\mathbb{E}++t^{-\frac{3}{2\alpha}}\Psi_{kj}\left(x/t^{\frac{1}{2\alpha}}\right)
	 =\frac{1}{4\pi}\cdot\frac{3x_{k}x_{j}-\delta_{kj}|x|^{2}}{|x|^{5}}+t^{-\frac{3}{2\alpha}}\Psi_{kj}\left(x/t^{\frac{1}{2\alpha}}\right)
	\end{align}
	and
	$$
	|D^{\ell}_x\partial_{t}^{m}\mathcal{O}_{kj}(x,t)|\leq C t^{-m}(t^{\frac{1}{2\alpha}}+|x|)^{-3-\ell},\;\;\;\forall\; \ell, m\in \mathbb{Z}^{+}.
	$$
	Here the function $\Psi_{kj}(x)$ is  smooth for $x\in \R^{3}\setminus\{0\}$ such that if $1/2<\alpha<1$
	$$
	|D_{x}^{\ell}\Psi_{kj}(x)|\leq C(1+|x|)^{-3-2\alpha-\ell}\ \ \ \mbox{as}\ \ |x|\to\infty,
	$$
	and if $\alpha=1$,
	$$
	|D_{x}^{\ell}\Psi_{kj}(x)|\leq C e^{-c|x|^{2}} \ \ \ \mbox{as}\ \ |x|\to\infty
	$$
	for some constants $C$ and $c$, depending only on $\ell$.
\end{lemma}
For proof, please see \cite{Br,Br1} for the case $\alpha=1$, and \cite{Cao} for the case $1/2<\alpha<1$.  For  completeness, we give an alternate proof for the case $\alpha=1$, which is more direct and and simpler. \vskip 0.1in

\begin{proof}[Proof of Lemma \ref{L4.1}]\;
 Let $G(x,t)=(4\pi t)^{-\frac{3}{2}}e^{-\frac{x^{2}}{4t}}$  be the kernel of heat operator, $\mathbb{P}$ be  the Leray projector,
 then
 $$\mathcal{O}(x,t)=(\mathcal{O}_{kj}(x,t))_{1\leq j,k\leq 3}=\mathbb{P} G(x,t).$$
Thus, from the definition of the Leray projector,
$$
\mathcal{O}_{kj}(x,1)=\delta_{kj}G_{1}(x)+\partial_{k}\partial_{j}\Big(\mathbb{E}(\cdot)\ast G_{1}(\cdot)\Big)(x)
$$
 with $G_{1}(x)\triangleq G(x,1)$.  Set  $\Theta(x)=(\mathbb{E}(\cdot)\ast G_{1}(\cdot))(x)$, it is easy to see $\Theta(x)=\Theta(|x|)\in C^{\infty}(\R^{3})$ satisfying
$$
-\Delta\Theta(x)=G_{1}(|x|),\ \ \ x\in \R^{3}.
$$
This, due to spherically symmetric of $\Theta$, can be rewritten as
$$
\Theta''(r)+\frac{2}{r}\Theta'(r)=-G_{1}(r),\ \ \ r=|x|\ \ \  \Psi'(0)=0.
$$
Integrating this equation from 0 to $r$, we have from the  fact $\int_{\R^{3}}G_{1}(x){\rm d}x=1$ that
\begin{align}
\nonumber
\Theta'(r)&=-r^{-2}\int_{0}^{r} s^{2}G_{1}(s){\rm d}s=-\frac{1}{4\pi}r^{-2}+r^{-2}\int_{r}^{+\infty}s^{2}G_{1}(s){\rm d}s\\&
\triangleq-\frac{1}{4\pi}r^{-2}+r^{-2}\phi_{1}(r),
\label{Integrate}
\end{align}
with $\phi_{1}(r)=\int_{r}^{+\infty}s^{2}G_{1}(s)\,{\rm d}s$.

We directly calculus to obtain
$$
\partial_{k}\Theta(x)=\Theta'(r)\frac{x_{k}}{r}=-\frac{1}{4\pi}x_{k}r^{-3}+r^{-3}\phi_{1}(r)x_{k}
$$
and
\begin{align*}
\partial_{k}\partial_{j}\Theta(x)&=\frac{1}{4\pi}\cdot\frac{3x_{k}x_{j}-\delta_{kj}r^{2}}{r^{5}}+
\frac{\delta_{kj}r^{2}-3x_{k}x_{j}}{r^{5}}\phi_{1}(r)-G_{1}(r)\frac{x_{k}x_{j}}{r^{2}}\\&
\triangleq\partial_{k}\partial_{j}\mathbb{E}+\Psi_{kj}(x)-G_{1}(r)\delta_{kj}.
\end{align*}
Thus,
\begin{align*}
\mathcal{O}_{kj}(x,1)=\partial_{k}\partial_{j}\mathbb{E}+\Psi_{kj}(x)=\frac{1}{4\pi}\cdot\frac{3x_{k}x_{j}-\delta_{kj}|x|^{2}}{|x|^{5}}+\Psi_{kj}(x)
\end{align*}
with
$$
|D_{x}^{\ell}\Psi_{kj}(x)|\leq C e^{-c|x|^{2}} \ \ \ \mbox{as}\ \ |x|\to\infty
$$
for $\ell\in \mathbb{Z}^{+}$.
From the relation $\mathcal{O}_{kj}(x,t)=t^{-\frac{3}{2}}\mathcal{O}_{kj}(x/t^{\frac12},1)$,
we obtain \eqref{Ossen} and finish the  proof of Lemma \ref{L4.1}.
\end{proof}

As a direct consequence of Lemma \ref{L4.1},  we immediately obtain the similar structure to the distribution kernel $F(x,t)$
of $e^{-(-\Delta)^{\alpha}t}\mathbb{P}{\rm div}(\cdot)$:
$$
\Upsilon(x,t)=(\Upsilon_{hkj}(x,t))_{1\leq j,h,k\leq 3}\triangleq(\partial_{h}\mathcal{O}_{kj}(x,t))_{1\leq j,h,k\leq 3}.
$$
It is clear that the $\Upsilon(x,t)$ possesses the self-similar property
$$
\Upsilon(x,t)=t^{-\frac{3+1}{2\alpha}}\Upsilon(x/t^{\frac{1}{2\alpha}},1).
$$
Besides, the $\Upsilon(x,t)$ has following properties:
\begin{lemma}\label{C4.1}
The profile $\Upsilon(x,1)$ satisfies the following pointwise decay estimate	
$$
|\Upsilon(x,1)|\leq C(1+|x|)^{-4} \ \ \ \mbox{for}\ \ x\in\R^{3}.
$$
Moreover, $\Upsilon(x,t)$ has the same structure as the Oseen kernel $\mathcal{O}$
\begin{align*}
\Upsilon_{hkj}(x,t)&=\partial_{h}\partial_{k}\partial_{j}\mathbb{E}+t^{-\frac{3}{2\alpha}}\partial_{h}\Psi_{kj}\left(x/t^{\frac{1}{2\alpha}}\right)\\&
	 =\frac{1}{4\pi}\cdot\partial_{h}\frac{3x_{k}x_{j}-\delta_{kj}|x|^{2}}{|x|^{5}}+t^{-\frac{3}{2\alpha}}\partial_{h}\Psi_{kj}\left(x/t^{\frac{1}{2\alpha}}\right)\\&
	\triangleq \mathcal{F}_{hkj}+t^{-\frac{3}{2\alpha}}\partial_{h}\Psi_{kj}\left(x/t^{\frac{1}{2\alpha}}\right)
	\end{align*}
	and $\mathcal{F}=(\mathcal{F}_{hkj})_{1\leq j,h,k\leq 3}$, usually called as a three-order tensor in $\R^{3}$, possesses the following cancellation properties
	\begin{align}\label{cancel}
	\int_{\mathbb{S}^{2}}\mathcal{F}(x) \;{\rm d}\mathbb{S}=0,\ \ \int_{\mathbb{S}^{2}}x_{i}\mathcal{F}(x) \;{\rm d}\mathbb{S}=0,
	\ i=1,2,3.
	\end{align}
\end{lemma}
\begin{proof}
We only need to prove \eqref{cancel}.  The proof is standard, and we give a proof for completeness. From
\begin{align*}
\mathcal{F}&=(\mathcal{F}_{hkj})=\Big(\frac{1}{4\pi}\cdot\partial_{h}\frac{3x_{k}x_{j}-\delta_{kj}|x|^{2}}{|x|^{5}}\Big)\\&
=\left(\frac{3}{4\pi}\frac{|x|^{2}(\delta_{jh}x_{k}+\delta_{hk}x_{j}+\delta_{kj}x_{h})-5x_{j}x_{h}x_{k}}{|x|^{7}}\right),
\end{align*}
 and  the symmetry of spherical surface, it is easy to see that
$$\int_{\mathbb{S}^{2}}\mathcal{F}(x) \;{\rm d}\mathbb{S}=0.$$
Besides, note that
$$
\int_{\mathbb{S}^{2}}x_{i}^{2}\,{\rm d}\mathbb{S}=\frac{1}{3}\int_{\mathbb{S}^{2}}|x|^{2}\,{\rm d}x=\frac{4\pi}{3}
$$
and with help of Gauss-Green formula, one instantly derives
$$
\int_{\mathbb{S}^{2}}x_{i}^{4}\,{\rm d}\mathbb{S}_{x}=\int_{\mathbb{B}_{1}}\partial_{i}(x_{i}^{3})\,{\rm d}x=\frac{4\pi}{5}
$$
and
$$
\int_{\mathbb{S}^{2}}x^{2}_{i}x^{2}_{j}\,{\rm d}\mathbb{S}_{x}
=\int_{\mathbb{B}_{1}}x^{2}_{i}\partial_{j}x_{j}\,{\rm d}x=\frac{4\pi}{15}.
$$
The above two identities imply  for $i=1,2,3$
$$
\int_{\mathbb{S}^{2}}x_{i}\mathcal{F}_{hkj} \;{\rm d}\mathbb{S}_{x}
=\frac{3}{4\pi}\int_{\mathbb{S}^{2}}x_{i}\big\{(\delta_{jh}x_{k}+\delta_{hk}x_{j}+\delta_{kj}x_{h})-5x_{j}x_{h}x_{k}\big\}\,{\rm d}\mathbb{S}_{x}=0.
$$
This verifies the  second identity of \eqref{cancel}.
\end{proof}

\section{Global regularity of weak solutions}\label{sec-3}
\setcounter{equation}{0}
\setcounter{equation}{0}
\setcounter{equation}{0}

In this section, we apply the bootstrap argument to show the high regularity of a weak solution established in Theorem \ref{thm-1}. To do that, we divide the proof into four steps as following.

Step 1:  With the help of method of vanishing viscosity, we will establish $H_{\omega}^\alpha(\R^3)$-estimate for weak solution $V$ by  choosing the suitable test function in the weak sense.

Step 2:
We will give $B^{2\alpha}_{2,\infty}(\R^3)$-estimate for weak solution $V$  based on the differences characterization of Besov spaces.

Step 3:  By the bootstrap argument, we will further establish high inequality $H^{1+\alpha}(\R^3)$-estimate for weak solution $V$.

Step 4: In view of difference, we will show high regularity for weak solution $V$ in the weighted Hilbert space $H_{\omega}^{1+\alpha}(\R^3)$.

\subsection{$	\big\|V\big\|_{H^\alpha_{\omega}(\R^3)}$-esitmate}
In this step, we are going to show  uniform estimate  for the weak solution $V$ in the weighted space $H^\alpha_{\omega}(\R^3)$. Specifically,
\begin{proposition}\label{prop3.1}
	Let $\alpha\in(5/6,1)$, and $V\in H_{\sigma}^{\alpha}(\R^{3})$ be the weak solution which was established in Theorem~\ref{thm-1}. Then there exists $C>0$ such that
	\begin{equation}\label{add-eq-002}
	\big\|\,V\big\|_{H_{\omega}^\alpha(\R^3)}\leq C\left(U_0\right).
	\end{equation}
\end{proposition}
Now, let's begin to show the  desired estimate in Proposition \ref{prop3.1} by choosing a suitable test function in the weak sense.  Since $V\in H^{\alpha}(\R^3)$ with $\alpha\in]0,1[$,
we are lack of the information of derivative of $V$. Then, it seems difficult to modify $V$ directly to the required  test function. To overcome that difficulty,  we apply a fact, with the method of vanishing viscosity, that the solution established in Theorem \ref{thm-1} is a weak limit of weak solution to the following equations.
 \begin{align}\label{Ev}
 \begin{split}
 -\nu\Delta V_{\nu}+ (-\Delta)^{\alpha}V_{\nu}- \tfrac{2\alpha-1}{2\alpha}V_{\nu}-\tfrac{1}{2\alpha}x\cdot \nabla V_\nu+\nabla P_\nu=&-U_{0}\cdot\nabla U_{0}
 -(U_{0}+V_\nu)\cdot\nabla V_\nu\\
&-V_\nu\cdot\nabla U_{0},
\end{split}
\end{align}
\begin{equation}\label{eq.incomv}
\textnormal{div}\,V_{\nu}=0.
\end{equation}
 In \cite{LXZ}, we have proved that equations \eqref{Ev}-\eqref{eq.incomv} admit at least  one weak solution $V_\nu$ such that
 \begin{equation}\label{eq-energy-v}
 \nu\|V_\nu\|^2_{\dot{H}^1(\R^3)}+\|V_\nu\|^2_{\dot{H}^\alpha(\R^3)}+\|V_\nu\|^2_{L^2(\R^3)}\leq C(U_0),
 \end{equation}
and there exists a pressure $P_\nu\in L^2(\R^3)$  governed by
\begin{equation*}\label{eq.PE}
P_\nu=\sum_{i,j=1}^3\frac{1}{4\pi}\partial^2_{x_i,x_j}\int_{\mathbb{R}^3}\frac{1}{|x-y|}\big(V_\nu^iV_\nu^j+U^i_0 V_\nu^j+V_\nu^iU^j_0+U^i_0 U^j_0\big)\,\mathrm{d}y
\end{equation*}
such that for all vector fields $\varphi\in H^1(\mathbb{R}^3)$ with $\big\||\cdot|^{\frac12}\varphi\big\|_{L^2(\mathbb{R}^3)}<+\infty,$
the couple $(V_\nu,P_\nu)$ satisfies
\begin{align}
&\nu\int_{\mathbb{R}^3}\nabla V_\nu:\nabla\varphi\,\mathrm{d}x+\int_{\mathbb{R}^3}\Lambda^\alpha V_\nu:\Lambda^\alpha\varphi\,\mathrm{d}x-\frac{1}{2\alpha}\int_{\mathbb{R}^3}x\cdot \nabla V_\nu\cdot\varphi\,\mathrm{d}x-\frac{2\alpha-1}{2\alpha}\int_{\mathbb{R}^3} V_\nu\cdot\varphi\,\mathrm{d}x\nonumber\\
=&\int_{\mathbb{R}^3}P_\nu\, \mathrm{div}\,\varphi\,\mathrm{d}x+\int_{\mathbb{R}^3}V_\nu\cdot \nabla \varphi\cdot V_\nu\,\mathrm{d}x-\int_{\mathbb{R}^3}U_0\cdot \nabla  V_\nu\cdot \varphi\,\mathrm{d}x-\int_{\mathbb{R}^3}(V_\nu+U_0)\cdot \nabla  U_0\cdot \varphi\,\mathrm{d}x.
\label{eq.weak}
\end{align}
In the following subsection, we denote $V_\nu$ by $V$   for clarity.

Letting $h_\varepsilon(x):=\sqrt{\frac{1+|x|}{1+\varepsilon|x|^2}}$,
$W_{\varepsilon}(x):=h_\varepsilon^2V(x)$ and $V_\varepsilon(x):=h_\varepsilon V(x)$ with $\varepsilon\in]0,1]$,
we easily find by using the energy estimate \eqref{eq-energy-v} that $W_{\varepsilon}\in H^1(\R^3)$ and satisfies $\big\|\langle\cdot\rangle W_\varepsilon\big\|_{L^2(\mathbb{R}^3)}<+\infty$ for each $\varepsilon\in]0,1]$.
Then we readily have by taking $\varphi(x)=W_{\varepsilon}(x)$ in \eqref{eq.weak} that
\begin{equation}\label{eq.weak-i}
\begin{split}
&\nu\int_{\mathbb{R}^3}\nabla V:\nabla W_{\varepsilon}\,\mathrm{d}x+\int_{\mathbb{R}^3}\Lambda^\alpha V:\Lambda^\alpha W_{\varepsilon}\,\mathrm{d}x-\frac{1}{2\alpha}\int_{\mathbb{R}^3}x\cdot \nabla V\cdot W_{\varepsilon}\,\mathrm{d}x\\
&-\frac{2\alpha-1}{2\alpha}\int_{\mathbb{R}^3} V\cdot W_{\varepsilon}\,\mathrm{d}x-\int_{\mathbb{R}^3}P\, \mathrm{div}\,W_{\varepsilon}\,\mathrm{d}x\\
=&\int_{\mathbb{R}^3}V\cdot \nabla W_{\varepsilon}\cdot V\,\mathrm{d}x-\int_{\mathbb{R}^3}U_0\cdot \nabla  V\cdot W_{\varepsilon}\,\mathrm{d}x-\int_{\mathbb{R}^3}(V+U_0)\cdot \nabla  U_0\cdot W_{\varepsilon}\,\mathrm{d}x.
\end{split}
\end{equation}
By a  simple calculation, the third integrand in the left side of the above equality reduces to
\begin{align*}
-\frac{1}{2\alpha}\int_{\mathbb{R}^3}x\cdot \nabla V\cdot W_{\varepsilon}\,\mathrm{d}x
=&\frac{3}{4\alpha}\int_{\mathbb{R}^3}h^2_\varepsilon|V|^2\,\mathrm{d}x+\frac{1}{4\alpha}\int_{\mathbb{R}^3}|V|^2\,x\cdot \nabla h^2_\varepsilon\,\mathrm{d}x\\
=&\frac{3}{4\alpha}\big\|V_\varepsilon\big\|_{L^2(\R^3)}^2+\frac{1}{4\alpha}\int_{\mathbb{R}^3}|V|^2\,   {|x|}g_\varepsilon^2\,\mathrm{d}x-\frac{1}{2\alpha}\int_{\mathbb{R}^3}|V|^2\,     {\varepsilon|x|^2}g_\varepsilon^2h^2_\varepsilon \,\mathrm{d}x \\
=&\frac{1}{\alpha}\big\|V_\varepsilon\big\|_{L^2(\R^3)}^2-\frac{1}{4\alpha}\int_{\mathbb{R}^3}\frac{|V|^2}{1+\varepsilon|x|^2}\,\mathrm{d}x-\frac{1}{2\alpha}\int_{\mathbb{R}^3}|V_\varepsilon|^2\,    {\varepsilon|x|^2}g_\varepsilon^2\,\mathrm{d}x.
\end{align*}
Here and what in follows, we denote $g_\varepsilon=\frac{1}{\sqrt{1+\varepsilon|x|^2}}.$

Hence, we have
\begin{equation}\label{eq-W-1}
\begin{split}
&\frac{1}{2\alpha}\int_{\mathbb{R}^3}x\cdot \nabla V\cdot W_{\varepsilon}\,\mathrm{d}x-\frac{2\alpha-1}{2\alpha}\int_{\mathbb{R}^3} V\cdot W_{\varepsilon}\,\mathrm{d}x\\
=&\frac{1-\alpha}{\alpha}\big\|V_\varepsilon\big\|_{L^2(\R^3)}^2+\frac{1}{2\alpha}\int_{\mathbb{R}^3}|V_{\varepsilon}|^2\,   g_\varepsilon^2\,\mathrm{d}x-\frac{1}{4\alpha}\int_{\mathbb{R}^3}|V|^2\,   g_\varepsilon^2\,\mathrm{d}x.
\end{split}
\end{equation}
For the first integrand, we see that
\begin{equation}\label{eq-Lapace}
\begin{split}
&\nu\int_{\mathbb{R}^3}\nabla V:\nabla W_{\varepsilon}\,\mathrm{d}x\\
=&\nu\int_{\mathbb{R}^3}\left(\nabla V\cdot\nabla h_\varepsilon\right)\cdot V_{\varepsilon}\,\mathrm{d}x
 +\nu\big\|\nabla V_\varepsilon\big\|^2_{L^2(\R^3)}-\nu\int_{\mathbb{R}^3} \nabla h_\varepsilon\cdot \left(V\cdot\nabla V_{\varepsilon}\right)\,\mathrm{d}x.
 \end{split}
\end{equation}
Next, we tackle with the term involving the fractional operator.  After some simple computation we get
\begin{equation}\label{eq-w-2}
\begin{split}
&\int_{\mathbb{R}^3}\Lambda^\alpha V:\Lambda^\alpha W_{\varepsilon}\,\mathrm{d}x\\
=&\int_{\mathbb{R}^3}\Lambda^\alpha V:\big(h_{\varepsilon}\,\Lambda^\alpha   V_{\varepsilon} \big)\,\mathrm{d}x-\int_{\mathbb{R}^3}\Lambda^\alpha V:\big([h_{\varepsilon},\,\Lambda^\alpha]   V_{\varepsilon} \big)\,\mathrm{d}x\\
=&\big\|V_\varepsilon\big\|^2_{\dot{H}^\alpha(\R^3)}-\int_{\mathbb{R}^3}\big([h_\varepsilon,\,\Lambda^\alpha] V\big): \Lambda^\alpha V_{\varepsilon}  \,\mathrm{d}x-\int_{\mathbb{R}^3}\Lambda^\alpha V:\big([h_{\varepsilon},\,\Lambda^\alpha] V_{\varepsilon} \big)\,\mathrm{d}x.
\end{split}
\end{equation}
Inserting \eqref{eq-W-1}, \eqref{eq-Lapace} and \eqref{eq-w-2} into \eqref{eq.weak-i} leads to
\begin{equation}\label{eq.weak-ii}
\begin{split}
&\nu\big\|V_\varepsilon\big\|^2_{\dot{H}^1(\R^3)}+\big\|V_\varepsilon\big\|^2_{\dot{H}^\alpha(\R^3)}+\frac{1-\alpha}{\alpha}\big\|V_\varepsilon\big\|_{L^2(\R^3)}^2+\frac{1}{2\alpha}\int_{\mathbb{R}^3} \frac{|V_{\varepsilon}|^2}{1+\varepsilon|x|^2}\,\mathrm{d}x\\
=&\frac{1}{4\alpha}\int_{\mathbb{R}^3} \frac{|V|^2}{1+\varepsilon|x|^2}\,\mathrm{d}x+\int_{\mathbb{R}^3}\big([h_\varepsilon,\,\Lambda^\alpha] V\big): \Lambda^\alpha V_{\varepsilon}  \,\mathrm{d}x
+\int_{\mathbb{R}^3}\Lambda^\alpha V:\big([h_{\varepsilon},\,\Lambda^\alpha] h_{\varepsilon}V \big)\,\mathrm{d}x\\
&-\nu\int_{\mathbb{R}^3}\nabla V:\nabla h_\varepsilon V_{\varepsilon}\,\mathrm{d}x+\nu\int_{\mathbb{R}^3} \nabla h_\varepsilon V:\nabla V_{\varepsilon}\,\mathrm{d}x+\int_{\mathbb{R}^3}P\, \mathrm{div}\,W_{\varepsilon}\,\mathrm{d}x\\
&+\int_{\mathbb{R}^3}V\cdot \nabla W_{\varepsilon}\cdot V\,\mathrm{d}x-\int_{\mathbb{R}^3}U_0\cdot \nabla  V\cdot W_{\varepsilon}\,\mathrm{d}x-\int_{\mathbb{R}^3}(V+U_0)\cdot \nabla  U_0\cdot W_{\varepsilon}\,\mathrm{d}x
\triangleq\sum_{i=1}^{10}I_i.
\end{split}
\end{equation}
Now, let's to bound the terms in the right side of equality \eqref{eq.weak-ii}. For   $I_1$,  it is obvious that
\begin{equation}\label{eq-r-1}
I_1=\frac{1}{4\alpha}\int_{\mathbb{R}^3} \frac{|V|^2}{1+\varepsilon|x|^2}\,\mathrm{d}x\leq \frac{1}{4\alpha}\big\|V\big\|^2_{L^2(\R^3)}.
\end{equation}
By the H\"older inequality and Lemma \ref{lem-Comm}, one has
\begin{equation}\label{eq-3.1-1}
\begin{split}
I_2\leq&\big\| [h_\varepsilon,\,\Lambda^\alpha] V\big\|_{L^2(\R^3)}\big\|\Lambda^\alpha V_{\varepsilon}\big\|_{L^2(\R^3)}\\
\leq& C\max\left\{\|h_\varepsilon\|_{\dot{C}^{\frac34}(\R^3)},\,\|h_\varepsilon\|_{\dot{W}^{1,\infty}(\R^3)}\right\}\|V\|_{L^2(\R^3)}\|V_{\varepsilon}\|_{\dot{H}^\alpha(\R^3)}.
\end{split}
\end{equation}
This require to bound both terms  concerning on $h_\varepsilon$ in \eqref{eq-3.1-1}. We compute
	\begin{align*}
	\nabla h_\varepsilon(x)=\frac{x}{2|x|\sqrt{1+|x|}}\frac{1}{\sqrt {1+\varepsilon |x|^2}}-\frac{\varepsilon x\sqrt{1+|x|}}{(1+\varepsilon |x|^2)^{\frac32}}.
	\end{align*}
It follows that for  each $\varepsilon\in]0,1],$
	\begin{align*}
	\big|	\nabla h_\varepsilon(x)\big|\leq\frac{1}{2\sqrt{1+|x|}}+\frac{4\varepsilon^\frac14 }{\sqrt{1+\varepsilon |x|}}.
	\end{align*}
	This inequality enables us to infer that
	\begin{align*}
	\|\nabla h_\varepsilon\|_{L^p(\R^3)}\leq \left(1/2+4\varepsilon^{\frac14-\frac3p}\right)\bigg(\int_{ \mathbb { R } ^ { 3 } }\frac{1}{(1+|x|)^{\frac{p}{2}}}\,\textnormal{d}x\bigg)^{\frac{1}{p}}.
	\end{align*}
	Thus for $\forall p\geq 12$ and $\forall\varepsilon\in]0,1],$ there exists a absolute constant $C_0>0$ such that
	\begin{align*}
		\|\nabla h_\varepsilon\|_{L^p(\R^3)}\leq \left(1/2+4\right)\bigg(\int_{ \mathbb { R } ^ { 3 } }\frac{1}{(1+|x|)^{\frac{p}{2}}}\,\textnormal{d}x\bigg)^{\frac{1}{p}}\leq C_0.
	\end{align*}
	This together with the embedding  relation that $\dot{W}^{1,p}(\R^3)\hookrightarrow C^{1-\frac3p}(\R^3)$ yield that there exists a constant $C>0$ independent of $\varepsilon $ such that for $\forall\varepsilon\in]0,1]$ and $\forall s\in[3/4,1[$,
	\begin{equation}\label{estimate-h}
	\|h_\varepsilon\|_{C^s(\R^3)}+\| h_\varepsilon\|_{\dot{W}^{1,\infty}(\R^3)}\leq C.
	\end{equation}
	Inserting estimate \eqref{estimate-h} into \eqref{eq-3.1-1}, we immediately have
\begin{equation}\label{eq-r-2}
 I_2\leq  C\|V\|_{L^2(\R^3)}\|V_{\varepsilon}\|_{\dot{H}^\alpha(\R^3)}.
\end{equation}
In the same way, $I_3$ can be bounded as follows
\begin{equation}\label{eq-r-3}
\begin{split}
I_3\leq C\big\|V^\varepsilon \big\|_{L^2(\R^3)}\|V\|_{\dot{H}^\alpha(\R^3)}
\leq C\|V\|_{\dot{H}^\alpha(\R^3)}^2+\frac{1-\alpha}{8\alpha}\big\|V_\varepsilon \big\|^2_{L^2(\R^3)}.
\end{split}
\end{equation}
Next, we turn to estimate   $I_4$ and $I_5$. We have by the H\"older inequality and Cauchy-Schwarz inequality that
\begin{equation}\label{r-d-1}
\begin{split}
I_4=-\nu\int_{\mathbb{R}^3}\nabla V:(h_\varepsilon\nabla h_\varepsilon)V\,\mathrm{d}x
\le&\nu\|\nabla V\|_{L^2(\R^3)}\|h_\varepsilon\nabla h_\varepsilon\|_{L^\infty(\R^3)}\|V\|_{L^2(\R^3)}\\
\leq &C\nu\|V\|^2_{L^2(\R^3)}+\nu\|\nabla V\|^2_{L^2(\R^3)}.
\end{split}
\end{equation}
In the last line of \eqref{r-d-1}, we have used the following equality
\begin{equation}\label{eq.h-d-h}
h_\varepsilon\nabla h_\varepsilon=\frac{x}{2|x| (1+\varepsilon |x|^2)}-\frac{\varepsilon x(1+|x|)}{(1+\varepsilon |x|^2)^{2}}.
\end{equation}
Similarly, by the H\"older inequality and estimate \eqref{estimate-h},  we obtain
\begin{equation}\label{r-d-2}
\begin{split}
I_5\le \nu\|\nabla h_\varepsilon\|_{L^\infty(\R^3)} \| V\|_{L^2(\R^3)}\|\nabla V_{\varepsilon}\|_{L^2(\R^3)}
\leq  C\nu\|V\|^2_{L^2(\R^3)}+\frac14\nu\|\nabla V_\varepsilon\|^2_{L^2(\R^3)}.
\end{split}
 \end{equation}
Now we come back to tackle with the term involving the pressure. Thanks to
\begin{equation*}
-\Delta P=\text{div}\,\left(U_{0}\cdot\nabla U_{0}
+(U_{0}+V)\cdot\nabla V+V\cdot\nabla U_{0}\right),
\end{equation*}
one writes
\[P=\frac{\text{div}}{-\Delta}\big((U_{0}+V)\cdot\nabla V+V\cdot\nabla U_{0}\big)+\frac{\text{div}}{-\Delta}\big(U_0\cdot\nabla U_{0}\big):=P_1+P_2. \]
Moreover, we have by the H\"older inequality that
\begin{equation}
\label{eq.I-2}
\|P_1\|_{L^2(\R^3)}\leq C\left(\|V\|_{L^4(\R^3)}^2+\|U_0\|_{L^\infty(\R^3)}\|V\|_{L^2(\R^3)}\right)
\end{equation}
and
\begin{equation}\label{eq-p-2}
\big\|\nabla P_2\big\|_{L^2(\R^3)}\leq C\|U_0\|_{L^\infty(\R^3)}\|\nabla U_0\|_{L^2(\R^3)}.
\end{equation}
Note that
$$
\sup_{x\in\R^3,R>0}\left(\frac{1}{|\mathbb{B}_R(x)|}\int_{B(x,R)}(1+|y|)
\,\mathrm{d}y\right)^{1/2}\left(\frac{1}{|\mathbb{B}_R(x)|}\int_{\mathbb{B}_R(x)}(1+|y|)^{-1}\,\mathrm{d}y\right)^{1/2}<\infty,
$$
for its proof please see \cite[Lemma 1]{Fe},  and  then  $1+|x|$ belongs to $A_2$-class,  therefore by \cite[Theorem 5.4.2]{stein-2}  that
\begin{equation}\label{eq.star}
\int_{\mathbb{R}^3}|\nabla P_2|^2\,(1+|x|)\mathrm{d}x\leq C\int_{\mathbb{R}^3} |U_0\cdot\nabla U_0|^2(x)(1+|x|)\,\mathrm{d}x.
\end{equation}
Integrating by parts, we see that
\begin{equation*}
\int_{\mathbb{R}^3}P_2\,\mathrm{div}\,W_{\varepsilon}\,\mathrm{d}x=-\int_{\mathbb{R}^3}\nabla P_2\,\cdot W_{\varepsilon}\,\mathrm{d}x.
\end{equation*}
Thus we get by using  the H\"older, the Young inequalities and estimates \eqref{eq-p-2}-\eqref{eq.star} that
\begin{equation}\label{est-p-1}
\begin{split}
\int_{\mathbb{R}^3}P_2\,\mathrm{div}\,W_{\varepsilon}\,\mathrm{d}x
\leq&\|h_\varepsilon\nabla P_2\|_{L^2(\R^3)}\big\|V_\varepsilon\big\|_{L^2(\R^3)}\\
\leq&\left(\| \nabla P_2\|_{L^2(\R^3)}+\|\sqrt{1+|\cdot|}\nabla P_2\|_{L^2(\R^3)}\right)\big\|V_\varepsilon\big\|_{L^2(\R^3)}\\
\leq&C\left(\|U_0\|_{L^\infty(\R^3)}\|\nabla U_0\|_{L^2(\R^3)}+\big\|\sqrt{1+|\cdot|} U_0\cdot\nabla U_0\big\|_{L^2(\R^3)}\right)\big\|V_\varepsilon\big\|_{L^2(\R^3)}\\
\leq &C\|U_0\|^2_{L^\infty(\R^3)}\|\nabla U_0\|^2_{L^2(\R^3)}+C\big\|\sqrt{1+|\cdot|} U_0\cdot\nabla U_0\big\|^2_{L^2(\R^3)}+\frac{1-\alpha}{8\alpha}\big\|V_\varepsilon\big\|^2_{L^2(\R^3)}.
\end{split}
\end{equation}
An computation yields
\begin{equation*}
\int_{\mathbb{R}^3}P_1\, \mathrm{div}\,W_{\varepsilon}\,\mathrm{d}x=2\int_{ \mathbb { R } ^ { 3 } }P_1\,h_\varepsilon\nabla h_\varepsilon\cdot V\,\mathrm{d}x.
\end{equation*}
Moreover, we obtain in terms of the H\"older inequality, \eqref{eq.h-d-h} and \eqref{eq.I-2} that
\begin{equation}\label{est-p-2}
\begin{split}
\int_{\mathbb{R}^3}P_1\, \mathrm{div}\,W_{\varepsilon}\,\mathrm{d}x
\leq&2\|h_\varepsilon\nabla h_\varepsilon\|_{L^\infty(\R^3)}\|P_1\|_{L^2(\R^3)}\|V\|_{L^2(\R^3)}\\
\leq&C\left(\|V\|_{L^4(\R^3)}^2+\|U_0\|_{L^\infty(\R^3)}\|V\|_{L^2(\R^3)}\right)\|V\|_{L^2(\R^3)}.
\end{split}
\end{equation}
Adding \eqref{est-p-1}  to   \eqref{est-p-2}, then
\begin{equation}\label{est-p}
\begin{split}
I_6
\leq& C\left(\|V\|_{L^4(\R^3)}^2+\|U_0\|_{L^\infty(\R^3)}\|V\|_{L^2(\R^3)}\right)\|V\|_{L^2(\R^3)}\\
&+C\|U_0\|^2_{L^\infty(\R^3)}\|\nabla U_0\|^2_{L^2(\R^3)}+\frac{1-\alpha}{8\alpha}\big\|V_\varepsilon\big\|^2_{L^2(\R^3)}.
\end{split}
\end{equation}
For $I_7$, we have by the H\"older inequality and
estimate \eqref{estimate-h} that
\begin{equation}\label{eq-con}
\begin{split}
I_7=\int_{\mathbb{R}^3}V\cdot \nabla h_\varepsilon\,  V_\varepsilon\cdot V\,\mathrm{d}x
\leq&\|V\|_{L^4(\R^3)}^2\|\nabla h_\varepsilon\|_{L^\infty(\R^3)}\|V_\varepsilon\|_{L^2(\R^3)}\\
\leq& C\|V\|_{L^4(\R^3)}^4+\frac{1-\alpha}{8\alpha}\big\|V_\varepsilon\big\|^2_{L^2(\R^3)}.
\end{split}
\end{equation}
For $I_8$,  after a simple calculation, we have by the H\"older inequality, the Young inequality and estimate \eqref{estimate-h}  that
\begin{equation}\label{eq-r-l-1}
\begin{split}
I_8= -\int_{\mathbb{R}^3}U_0\cdot \nabla W_{\varepsilon} \cdot V\,\mathrm{d}x
=&-\int_{\mathbb{R}^3}\left(U_0\cdot \nabla h_\varepsilon\right)  \,\left( V_\varepsilon\cdot V\right)\,\mathrm{d}x\\
\leq &\|U_0\|_{L^\infty(\R^3)}\|\nabla h_\varepsilon\|_{L^\infty(\R^3)}\|V\|_{L^2(\R^3)}\|V_\varepsilon\|_{L^2(\R^3)}\\
\leq& C\|U_0\|^2_{L^\infty(\R^3)}\|V\|^2_{L^2(\R^3)}+\frac{1-\alpha}{8\alpha}\big\|V_\varepsilon\big\|^2_{L^2(\R^3)}.
\end{split}
\end{equation}
At last, employing  the H\"older inequality and the Young inequality, we readily obtain
\begin{equation}\label{eq-r-l-2}
\begin{split}
I_9\leq \int_{\mathbb{R}^3}\big|V\big| (1+|x|) \nabla  U_0 \big|V \big|\,\mathrm{d}x
\leq\left(\|\nabla U_0\|_{L^\infty(\R^3)}+\big\||\cdot|\nabla U_0\big\|_{L^\infty(\R^3)}\right)\|V\|^2_{L^2(\R^3)},
\end{split}
\end{equation}
and
\begin{equation}\label{eq-r-known}
\begin{split}
I_{10}\leq \int_{\mathbb{R}^3}\big|\sqrt{1+|x|} \,U_0 \cdot \nabla  U_0 \big|\,\big|V_\varepsilon \big|\,\mathrm{d}x
\leq& \big\|\sqrt{1+|\cdot|} \,U_0\cdot  \nabla  U_0\big\|_{L^2(\R^3)} \|V_\varepsilon\|_{L^2(\R^3)}\\
\leq&C\big\|\sqrt{1+|\cdot|}\, U_0\cdot  \nabla  U_0\big\|^2_{L^2(\R^3)}+\frac{1-\alpha}{16\alpha} \|V_\varepsilon\|^2_{L^2(\R^3)}.
\end{split}
\end{equation}
Plugging \eqref{eq-r-1}, \eqref{eq-r-2}-\eqref{r-d-2}, \eqref{eq-con}-\eqref{eq-r-known} into \eqref{eq.weak-ii}, we eventually obtain
\begin{align}
\nonumber
&\tfrac12\nu\big\|V_\varepsilon\big\|^2_{\dot{H}^1(\R^3)}+\tfrac12\big\|V_\varepsilon\big\|^2_{\dot{H}^\alpha(\R^3)}
+\tfrac{3(1-\alpha)}{8\alpha}\big\|V_\varepsilon\big\|_{L^2(\R^3)}^2+\tfrac{1}{2\alpha}\int_{\mathbb{R}^3}|V_{\varepsilon}|^2\,   \tfrac{1}{1+\varepsilon|x|^2}\,\mathrm{d}x\\
\leq& \tfrac{1}{4\alpha}\big\|V\big\|^2_{L^2(\R^3)}+C\|V\|_{H^\alpha(\R^3)}^2
+C\left(\|V\|_{L^4(\R^3)}^2+\|U_0\|_{L^\infty(\R^3)}\|V\|_{L^2(\R^3)}\right)\|V\|_{L^2}\nonumber\\
&+C\|U_0\|^2_{L^\infty(\R^3)}\|\nabla U_0\|^2_{L^2(\R^3)}+C\|V\|_{L^4(\R^3)}^4+ C\|U_0\|^2_{L^\infty(\R^3)}\|V\|^2_{L^2(\R^3)}\nonumber\\
&+C\big(\|\nabla U_0\|_{L^\infty(\R^3)}+\big\||\cdot|\nabla U_0\big\|_{L^\infty(\R^3)}\big)\|V\|^2_{L^2(\R^3)}+C\big\|\sqrt{\langle\cdot\rangle} U_0\cdot  \nabla  U_0\big\|^2_{L^2(\R^3)}\nonumber\\
&+ C\nu\|V\|^2_{L^2(\R^3)}+\nu\|\nabla V\|^2_{L^2(\R^3)}.
\nonumber
\end{align}
Taking $\varepsilon\to0+$ and $\nu\to 0+$, we eventually  get by using \eqref{eq-energy-v} and Lebesgue' dominated convergence theorem that
\begin{align*}
& \big\|\sqrt{1+|\cdot|} V \big \|^2_{\dot{H}^\alpha(\R^3)}+\frac{3(1-\alpha)}{ 8\alpha}\big \|\sqrt{1+|\cdot|}V\big \|_{L^2(\R^3)}^2+\frac{1}{2\alpha}\int_{\mathbb{R}^3}(1+|x|)|V|^2  \,\mathrm{d}x \leq C\big(U_0\big).
 \end{align*}
This estimate implies  the desired result in Proposition \ref{prop3.1}.

\subsection{$B^{2\alpha}_{2,\infty}(\R^3)$-estimate }\;
 The second step is to improve the regularity of $V$ in terms of the differences characterization of Besov spaces.
\begin{proposition}\label{prop3.2}
	Let  $\alpha\in(5/6,1),$ and $V\in H_{\sigma}^{\alpha}(\R^{3})$ be the weak solution which was established in Theorem~\ref{thm-1}. Then there exists $C>0$ such that
	\begin{equation}\label{add-eq-003}
	\big\| V\big\|_{B_{2,\infty}^{2\alpha}(\R^3)}<C\left(U_0\right).
	\end{equation}
\end{proposition}
\begin{proof}  We are using vanishing viscosity again to get the proposition, which was used in the proof of Proposition \ref{prop3.1}. For the sake of simplicity, we denote $\vartriangle^{h^{\alpha}}_k$ by $\vartriangle^{h}_k$. According to both estimates \eqref{add-eq-002} and \eqref{eq-energy-v}, it is easy to check that $\vartriangle_k^{-h}\big(g^2_\varepsilon\vartriangle^{h}_kV\big)\in  H^1_{\omega}(\R^3)$ with $g_{\varepsilon}(x)=\frac{1}{\sqrt{1+\varepsilon|x|^{2}}}$. Moreover, we immediately have by
	choosing $\varphi=-\vartriangle_k^{-h}\big(g^2_\varepsilon\vartriangle^{h}_kV\big)$ in \eqref{eq.weak-i}  that
\begin{align}\label{eq.weak3.2}
\begin{split}
&-\nu\int_{\mathbb{R}^3}\nabla V:\nabla\vartriangle_k^{-h}\big(g^2_\varepsilon\vartriangle^{h}_kV\big)\,\mathrm{d}x-\int_{\mathbb{R}^3}\Lambda^\alpha V:\Lambda^\alpha\vartriangle_k^{-h}\big(g^2_\varepsilon\vartriangle^{h}_kV\big)\,\mathrm{d}x\\
&-\frac{1}{2\alpha}\int_{\mathbb{R}^3}x\cdot \nabla V\cdot\vartriangle_k^{-h}\big(g^2_\varepsilon\vartriangle^{h}_kV\big)\,\mathrm{d}x+\frac{2\alpha-1}{2\alpha}\int_{\mathbb{R}^3} V\cdot\vartriangle_k^{-h}\big(g^2_\varepsilon\vartriangle^{h}_kV\big)\,\mathrm{d}x\\
=&-\int_{\mathbb{R}^3}P\, \mathrm{div}\,\vartriangle_k^{-h}\big(g^2_\varepsilon\vartriangle^{h}_kV\big)\,\mathrm{d}x-\int_{\mathbb{R}^3}V\cdot \nabla\vartriangle_k^{-h}\big(g^2_\varepsilon\vartriangle^{h}_kV\big)\cdot V\,\mathrm{d}x\\
&-\int_{\mathbb{R}^3}U_0\cdot \nabla  V\cdot\vartriangle_k^{-h}\big(g^2_\varepsilon\vartriangle^{h}_kV\big)\,\mathrm{d}x+\int_{\mathbb{R}^3}(V+U_0)\cdot \nabla  U_0\cdot \vartriangle_k^{-h}\big(g^2_\varepsilon\vartriangle^{h}_kV\big)\,\mathrm{d}x.
\end{split}
\end{align}
The first term in the left side of the above equality can be reduced to
\begin{align*}
&-\nu\int_{\mathbb{R}^3}\nabla V:\nabla \vartriangle_k^{-h}\big(g^2_\varepsilon \vartriangle^h_k V\big)\,\mathrm{d}x\\
=&-\nu\int_{\mathbb{R}^3}\nabla V: \vartriangle_k^{-h}\big(g^2_\varepsilon \vartriangle^h_k \nabla V\big)\,\mathrm{d}x-\nu\int_{\mathbb{R}^3}\nabla V: \vartriangle_k^{-h}\big(\nabla g^2_\varepsilon \vartriangle^h_k V\big)\,\mathrm{d}x\\
=&\nu\big\|g_\varepsilon \vartriangle^h_k \nabla V\big\|^2_{L^2(\mathbb{R}^3)}+\nu\int_{\mathbb{R}^3}\vartriangle_k^{h}\nabla V:\big(\nabla g^2_\varepsilon \vartriangle^h_k V\big)\,\mathrm{d}x.
\end{align*}
For the term including the fractional operator, we easily find that
\begin{align*}
&-\int_{\R^3}\Lambda^\alpha V\cdot\Lambda^\alpha\,\vartriangle_k^{-h}\Big(g^2_\varepsilon\vartriangle^h_kV\Big)\,\mathrm{d}x\\
=&\int_{\R^3}g^2_\varepsilon \Lambda^\alpha \vartriangle_k^{h}V\cdot\, \Lambda^\alpha  \vartriangle^h_kV  \,\mathrm{d}x-\int_{\R^3}g_\varepsilon \Lambda^\alpha \vartriangle_k^{h}V\cdot\, \big([g_\varepsilon,\Lambda^\alpha]  \vartriangle^h_kV  \big)\,\mathrm{d}x\\&-\int_{\R^3} \Lambda^\alpha \vartriangle_k^{h}V\cdot\, \big([g_\varepsilon,\Lambda^\alpha] g_\varepsilon \vartriangle^h_kV \big)\,\mathrm{d}x.
\end{align*}
Thus we have
\begin{align*}
-\int_{\R^3}\Lambda^\alpha V\cdot\Lambda^\alpha\,\vartriangle_k^{-h}\Big(g_\varepsilon^2 \vartriangle^h_kV\Big)\,\mathrm{d}x
=&\big\| g_\varepsilon \Lambda^\alpha \vartriangle^h_kV  \big\|_{L^2(\R^3)}^2-\int_{\R^3}g_\varepsilon \Lambda^\alpha \vartriangle_k^{h}V \big([g_\varepsilon,\Lambda^\alpha]  \vartriangle^h_kV  \big)\,\mathrm{d}x\\&-\int_{\R^3} \Lambda^\alpha \vartriangle_k^{h}V\cdot \big([g_\varepsilon,\Lambda^\alpha] g_\varepsilon \vartriangle^h_kV \big)\,\mathrm{d}x.
\end{align*}
 Now we consider the remaining part in left side of \eqref{eq.weak3.2}. Integrating by parts, we get
\begin{align*}
&\frac{1}{2\alpha}\int_{\mathbb{R}^3}x\cdot \nabla V\cdot \vartriangle_k^{-h}\Big(g_\varepsilon^2\vartriangle^h_kV\Big)\,\mathrm{d}x+\frac{2\alpha-1}{2\alpha}\int_{\mathbb{R}^3} V\cdot \vartriangle_k^{-h}\Big(g_\varepsilon^2\vartriangle^h_kV\Big)\,\mathrm{d}x\\
=&-	\frac{1}{2\alpha}\int_{\mathbb{R}^3}g_{\varepsilon}^2\times(x+h\mathbf{e}_k)\cdot \nabla \vartriangle_k^{h} V\cdot \vartriangle^h_kV\,\mathrm{d}x-	 \frac{1}{2\alpha}\int_{\mathbb{R}^3}g_\varepsilon^2\big(\vartriangle_k^{h}x\big)\cdot \nabla V\cdot \vartriangle^h_kV\,\mathrm{d}x\\
&-\frac{2\alpha-1}{2\alpha}\int_{\mathbb{R}^3}g^2_\varepsilon \vartriangle^h_kV\cdot \vartriangle^h_kV\,\mathrm{d}x\\
=&\frac{5-4\alpha}{4\alpha}\big\|g_\varepsilon \vartriangle^h_k  V\big\|^2_{L^2(\mathbb{R}^3)}+\frac{1}{4\alpha}\int_{\R^3}x\cdot\nabla  g_\varepsilon^2\,\vartriangle_k^{h} V\cdot \vartriangle^h_kV\,\mathrm{d}x
-	\frac{h}{2\alpha}\int_{\mathbb{R}^3}g_\varepsilon^2 \partial_{x_k}\vartriangle_k^{h} V\cdot \vartriangle^h_kV\,\mathrm{d}x\\&-	 \frac{1}{2\alpha}\int_{\mathbb{R}^3}g_\varepsilon^2 \partial_{x_k}V\cdot \vartriangle^h_kV\,\mathrm{d}x.\end{align*}
This equality implies that
\begin{align*}
&\frac{1}{2\alpha}\int_{\mathbb{R}^3}x\cdot \nabla V\cdot \vartriangle_k^{-h}\Big(g_\varepsilon^2\vartriangle^h_kV\Big)\,\mathrm{d}x+\frac{2\alpha-1}{2\alpha}\int_{\mathbb{R}^3} V\cdot \vartriangle_k^{-h}\Big(g_\varepsilon^2\vartriangle^h_kV\Big)\,\mathrm{d}x\\
=&\frac{5-4\alpha}{4\alpha}\big\|g_\varepsilon \vartriangle^h_k  V\big\|^2_{L^2(\mathbb{R}^3)} -\frac{1}{2\alpha}\int_{\R^3} \frac{\varepsilon |x|^2}{(1+\varepsilon|x|^2)^2}\,\vartriangle_k^{h} V\cdot \vartriangle^h_kV\,\mathrm{d}x\\
&+\frac{h}{4\alpha}\int_{\mathbb{R}^3}\partial_{x_k}g_\varepsilon^2 \vartriangle_k^{h} V\cdot \vartriangle^h_kV\,\mathrm{d}x-\frac{1}{2\alpha}\int_{\mathbb{R}^3}g_\varepsilon^2 \partial_{x_k}V\cdot \vartriangle^h_kV\,\mathrm{d}x.
\end{align*}
Inserting the above estimates into \eqref{eq.weak3.2} leds to that
\begin{align}\label{eq-w-h}
\begin{split}&\nu\big\|g_\varepsilon \vartriangle^h_k \nabla V\big\|^2_{L^2(\mathbb{R}^3)}+\big\| g_\varepsilon \Lambda^\alpha \vartriangle^h_kV  \big\|_{L^2(\R^3)}^2+\frac{5-4\alpha}{4\alpha}\big\|g_\varepsilon \vartriangle^h_k  V\big\|^2_{L^2(\mathbb{R}^3)}\\
= &-\nu\int_{\mathbb{R}^3}\vartriangle_k^{h}\nabla V:\big(\nabla g^2_\varepsilon \vartriangle^h_k V\big)\,\mathrm{d}x+\int_{\R^3}g_\varepsilon \Lambda^\alpha \vartriangle_k^{h}V\cdot\, \big([g_\varepsilon,\Lambda^\alpha]  \vartriangle^h_kV  \big)\,\mathrm{d}x\\
&+\int_{\R^3} \Lambda^\alpha \vartriangle_k^{h}V\cdot\, \big([g_\varepsilon,\Lambda^\alpha] g_\varepsilon \vartriangle^h_kV \big)\,\mathrm{d}x+\frac{1}{2\alpha}\int_{\R^3} \frac{\varepsilon |x|^2}{1+\varepsilon|x|^2}\,\vartriangle_k^{h} V\cdot \vartriangle^h_kV\,\mathrm{d}x\\
&-\frac{h}{4\alpha}\int_{\mathbb{R}^3}\partial_{x_k}g_\varepsilon^2 \vartriangle_k^{h} V\cdot \vartriangle^h_kV\,\mathrm{d}x+\frac{1}{2\alpha}\int_{\mathbb{R}^3}g_\varepsilon^2 \partial_{x_k}V\cdot \vartriangle^h_kV\,\mathrm{d}x\\
&-\int_{\R^3}P\,{\rm div}\vartriangle_k^{-h}\big(g^2_\varepsilon \vartriangle^h_kV\big)\,\mathrm{d}x+\int_{\R^3}
V\cdot\nabla V\,\vartriangle_k^{-h}\big(g^2_\varepsilon \vartriangle^h_kV\big)\,\mathrm{d}x\\
&+\int_{\R^3}
U_0\cdot\nabla V\,\vartriangle_k^{-h}\big(g^2_\varepsilon \vartriangle^h_kV\big)\,\mathrm{d}x+\int_{\R^3}
\big(V+U_0\big)\cdot\nabla U_0\,\vartriangle_k^{-h}\big(g^2_\varepsilon \vartriangle^h_kV\big)\,\mathrm{d}x:=\sum_{i=1}^{10}J_i.
\end{split}
\end{align}
Our task is now to estimate for all terms in the right side of equality \eqref{eq-w-h}. Let's begin with  $J_1$. Note that
\begin{equation*}
-\nu\int_{\mathbb{R}^3}\vartriangle_k^{h}\nabla V:\big(\nabla g^2_\varepsilon \vartriangle^h_k V\big)\,\mathrm{d}x
=
\nu\int_{\mathbb{R}^3 }\nabla \vartriangle_k^{h}V:\left( \frac{2\varepsilon x}{(1+\varepsilon|x|^2)^2}\vartriangle^h_k V\right)\,\mathrm{d}x,
\end{equation*}
we have that for each $\varepsilon\in(0,1)$,
\begin{align*}
J_1
\leq C\nu \big\|g_\varepsilon \vartriangle_k^{h}\nabla V\big\|_{L^2(\R^3)}\|\vartriangle^h_kV\|_{L^2(\R^3)}
\leq C\nu \|V\|^2_{\dot{H}^\alpha(\R^3)}+\frac14\nu\big\|g_\varepsilon \vartriangle_k^{h}\nabla V\big\|^2_{L^2(\R^3)}.
\end{align*}
In terms of the H\"older inequality, Cauchy-Schwarz inequality and Lemma \ref{lem-Comm}, we have
\begin{align*}
J_2
\leq &\big\|g_\varepsilon \Lambda^\alpha \vartriangle_k^{h}V\big\|_{L^2(\R^3)}\big\|[g_\varepsilon,\Lambda^\alpha]  \vartriangle^h_kV\big\|_{L^2(\R^3)}\\
\leq &\big\|g_\varepsilon \Lambda^\alpha \vartriangle_k^{h}V\big\|_{L^2(\R^3)}\big\|\vartriangle^h_kV\big\|_{L^2(\R^3)}
\leq C\|V\|^2_{\dot{H}^\alpha(\R^3)}+\frac{1}{16}\big\|g_\varepsilon \Lambda^\alpha \vartriangle_k^{h}V\big\|^2_{L^2(\R^3)}.
\end{align*}
As for $J_3$, we rewrite
\begin{align*}
J_3=
\int_{\R^3} g_\varepsilon\Lambda^\alpha \vartriangle_k^{h}V\cdot\, \big([\Lambda^\alpha, g_\varepsilon] \vartriangle^h_kV \big)\,\mathrm{d}x +\int_{\R^3} \Lambda^\alpha \vartriangle_k^{h}V\cdot\, \big([\Lambda^\alpha, g^2_\varepsilon] \vartriangle^h_kV \big)\,\mathrm{d}x.
\end{align*}
In the similar fashion as used in bounding $J_2$, we see that
\begin{align}\label{add-eq-004}
\begin{split}
\int_{\R^3} g_\varepsilon\Lambda^\alpha \vartriangle_k^{h}V\cdot \big([\Lambda^\alpha, g_\varepsilon] \vartriangle^h_kV \big)\,\mathrm{d}x
 \leq&\big\|g_\varepsilon\Lambda^\alpha \vartriangle_k^{h}V\big\|_{L^2(\R^3)}\big\| [g_\varepsilon,\Lambda^\alpha] (  \vartriangle^h_kV )\big\|_{L^2(\R^3)}\\
\leq&\big\|g_\varepsilon\Lambda^\alpha \vartriangle_k^{h}V\big\|_{L^2(\R^3)}\big\| \vartriangle^h_kV \big\|_{L^2(\R^3)}\\
\leq&C \| V  \|^2_{\dot{H}^\alpha(\R^3)}+\frac{1}{16}\big\|g_\varepsilon\Lambda^\alpha \vartriangle_k^{h}V\big\|^2_{L^2(\R^3)}.
\end{split}
\end{align}
For another part of $J_3$, we have to resort to the following Lemma:
\begin{lemma}\label{lem-g-}
Let  $\alpha\in(5/6,1)$. Then we have
\begin{equation}\label{commut-02}
	\big\|g_\varepsilon^{-\frac12}\big[\Lambda^\alpha,\,g^2_\varepsilon\big]f\big\|_{L^2(\R^3)}\leq C\max\big\{ \varepsilon^{\frac{1}{2}},\varepsilon^{\frac{1}{4}}\big\}\|f\|_{L^2(\R^3)}.
	\end{equation}
\end{lemma}
\begin{proof}[Proof of Lemma \ref{lem-g-} ]
We divide the  commutator into two parts  as follows
	\begin{align*}
\left[\Lambda^\alpha,g^2_\varepsilon\right]f=&
	\int_{\R^3}\frac{g^2_\varepsilon(y)-g^2_\varepsilon(x)}{|x-y|^{3+\alpha}}f(y)\,\mathrm{d}y\\
=&
	 \varepsilon\int_{\mathbb B_1(x)}\frac{ (x-y)\cdot(x+y)}{|x-y|^{3+\alpha}}\frac{1}{(1+\varepsilon|x|^2)(1+\varepsilon|y|^2)}f(y)\,\mathrm{d}y\\
	& +\varepsilon\int_{\mathbb B_1^c(x)}\frac{ (x-y)\cdot(x+y)}{|x-y|^{3+\alpha}}\frac{1}{(1+\varepsilon|x|^2)(1+\varepsilon|y|^2)}f(y)\,
\mathrm{d}y\triangleq{I+II}.
  \end{align*}
By the triangle inequality, we easily find that
\begin{align*}	
| I|&\leq \sqrt{\varepsilon}\int_{\mathbb B_1(x)}\frac{|x-y|}{|x-y|^{3+\alpha}}\frac{\sqrt{\varepsilon}(|x|+|y|)}{(1+\varepsilon|x|^{2})
(1+\varepsilon|y|^{2})}
|f(y)|\,\mathrm{d}y\\&
\leq \sqrt{\varepsilon}g_{\varepsilon}^{\frac{1}{2}}\int_{\mathbb B_1(x)}\frac{|x-y|}{|x-y|^{3+\alpha}}|f(y)|\,\mathrm{d}y,
\end{align*}
and
\begin{align*}
|II|&\leq \varepsilon\int_{\mathbb B_1^c(x)}\frac{|x-y|^{\frac{1}{2}}}{|x-y|^{3+\alpha}}\frac{|x-y|^{\frac{1}{2}}|x+y|}{(1+\varepsilon|x|^2)(1+\varepsilon|y|^2)}|f(y)|\,\mathrm{d}y\\&
\leq C \varepsilon^{\frac{1}{4}}g_{\varepsilon}^{\frac{1}{2}}\int_{\mathbb B_1^c(x)}\frac{|x-y|^{\frac{1}{2}}}{|x-y|^{3+\alpha}}
\frac{ (\sqrt{\varepsilon}|x|)^{\frac{3}{2}}+(\sqrt{\varepsilon}|y|)^{\frac{3}{2}}) }{(1+\varepsilon|x|^2)^{\frac{3}{4}}(1+\varepsilon|y|^2)}f(y)\,\mathrm{d}y\\&
\leq C \varepsilon^{\frac{1}{4}}g_{\varepsilon}^{\frac{1}{2}}\int_{\mathbb B_1^c(x)}\frac{|x-y|^{\frac{1}{2}}}{|x-y|^{3+\alpha}}f(y)\,\mathrm{d}y,
\end{align*}
in the second line, we have used the following estimate
\begin{align*}
\frac{\varepsilon|x-y|^{\frac{1}{2}}|x+y|}{(1+\varepsilon|x|^2)(1+\varepsilon|y|^2)}\leq& \frac{\varepsilon (|x|+|y|)^{\frac{3}{2}}}{(1+\varepsilon|x|^2)(1+\varepsilon|y|^2)}\\
\leq& \frac{2\varepsilon (|x|^{\frac{3}{2}}+|y|^{\frac{3}{2}}) }{(1+\varepsilon|x|^2)(1+\varepsilon|y|^2)}\leq 2\varepsilon^{\frac{1}{4}}g_{\varepsilon}\frac{ (\sqrt{\varepsilon}|x|)^{\frac{3}{2}}+(\sqrt{\varepsilon}|y|)^{\frac{3}{2}}) }{(1+\varepsilon|x|^2)^{\frac{3}{4}}(1+\varepsilon|y|^2)}.
\end{align*}
Combining both estimates concerning $I$ and $II$ yields
$$
\big|\left[\Lambda^\alpha,g^2_\varepsilon\right]f\big|\leq C\max\big\{ \varepsilon^{\frac{1}{2}},\varepsilon^{\frac{1}{4}}\big\}g^{\frac12}_\varepsilon(x)\int_{\R^3}\frac{\min\{ |x-y|^{\frac12}, \, |x-y| \}}{|x-y|^{3+\alpha}}\big|f(y)\big|\,\mathrm{d}y.
$$
Furthermore, by the Young inequality,  we readily have
	\begin{equation*}
	\big\|g_\varepsilon^{-\frac12}\big[\Lambda^\alpha,\,g^2_\varepsilon\big]f\big\|_{L^2(\R^3)}\leq C\max\big\{ \varepsilon^{\frac{1}{2}},\varepsilon^{\frac{1}{4}}\big\}\|f\|_{L^2(\R^3)}.
	\end{equation*}
So we complete the proof of Lemma \ref{lem-g-}.
\end{proof}
\vskip0.2cm

With  Lemma \ref{lem-g-} in hand, we get by the H\"older inequality that
\begin{align*}
\int_{\R^3} \Lambda^\alpha \vartriangle_k^{h}V\cdot \big([\Lambda^\alpha, g^2_\varepsilon] \vartriangle^h_kV \big)\,\mathrm{d}x
\leq &\big\|g^{\frac12}_\varepsilon\Lambda^\alpha \vartriangle_k^{h}V\big\|_{L^2(\R^3)}\left\|g_\varepsilon^{-\frac12}\big[\Lambda^\alpha,\,g^2_\varepsilon\big]\vartriangle^h_kV \right\|_{L^2(\R^3)}\\
\leq&C\max\big\{ \varepsilon^{\frac{1}{2}},\varepsilon^{\frac{1}{4}}\big\}\big\|g^{\frac12}_\varepsilon\Lambda^\alpha \vartriangle_k^{h}V\big\|_{L^2(\R^3)}\big\|\vartriangle^h_kV \big\|_{L^2(\R^3)}.
\end{align*}
This together with the fact that for each
$0<\varepsilon<\min\{1,\,h^{4\alpha}\},$
\begin{equation*}
\max\big\{ \varepsilon^{\frac{1}{2}},\varepsilon^{\frac{1}{4}}\big\}\big\|g^{\frac12}_\varepsilon\Lambda^\alpha \vartriangle_k^{h}V\big\|_{L^2(\R^3)}\leq 2\|\Lambda^\alpha V\big\|_{L^2(\R^3)}
\end{equation*}
enables us to conclude that
\[\int_{\R^3} \Lambda^\alpha \vartriangle_k^{h}V\cdot\, \big([\Lambda^\alpha, g^2_\varepsilon] \vartriangle^h_kV \big)\,\mathrm{d}x\leq C\|V\|^2_{\dot{H}^\alpha(\R^3)}.\]
Combining this inequality  with \eqref{add-eq-004} yields  that
$$J_3\leq C \| V  \|^2_{\dot{H}^\alpha(\R^3)}+\frac{1}{16}\big\|g_\varepsilon\Lambda^\alpha \vartriangle_k^{h}V\big\|^2_{L^2(\R^3)}.$$
Next,  by using the H\"older inequality,  one has
\begin{align*}
J_4\leq\frac{1}{2\alpha}\int_{\R^3} \left|\vartriangle_k^{h} V\cdot \vartriangle^h_kV\right|\,\mathrm{d}x
\leq \frac{1}{2\alpha}\big\|\vartriangle^h_kV\big\|_{L^2(\R^3)}^2\leq C\|V\|^2_{\dot{H}^\alpha(\R^3)},
\end{align*}
and
\begin{align*}
J_5\leq &\frac{1}{2\alpha}\|g_\varepsilon\|_{L^\infty(\R^3)}\|\nabla g_\varepsilon\|_{L^\infty(\R^3)}\big\|\vartriangle^h_kV\big\|_{L^2(\R^3)}^2
\leq C\|V\|^2_{\dot{H}^\alpha(\R^3)}.
\end{align*}
For $J_6$, we rewrite it as
\begin{align*}
J_6
=&-\frac{1}{2\alpha}\int_{\mathbb{R}^3}\partial_{x_k}(g_\varepsilon V)\cdot \big(g_\varepsilon \vartriangle^h_kV\big)\,\mathrm{d}x+\frac{1}{2\alpha}\int_{\mathbb{R}^3}(\partial_{x_k}g_\varepsilon V)\cdot \big(g_\varepsilon \vartriangle^h_kV\big)\,\mathrm{d}x:=J_{6_1}+J_{6_2}.
\end{align*}
By the H\"older inequality, we obtain
\begin{align}\label{add-eq-005}
\begin{split}
J_{6_2}
\leq\frac{1}{2\alpha}\|g_\varepsilon\|_{L^\infty(\R^3)}\|\nabla g_\varepsilon\|_{L^\infty(\R^3)}\|V\|_{L^2(\R^3)}\big\|\vartriangle^h_kV\big\|_{L^2(\R^3)}
\leq C\|V\|_{L^2(\R^3)}\|V\|_{\dot{H}^\alpha(\R^3)}.
\end{split}
\end{align}
On the other hand, we have by the interpolation inequality that
\begin{align*}
J_{6_1}\leq\big\|\partial_{x_k}(g_\varepsilon V)\big\|_{\dot{H}^{-\alpha}}\big\|g_\varepsilon \vartriangle^h_kV\big\|_{\dot{H}^\alpha(\R^3)}
\leq \left(\|V\|_{L^2(\R^3)}+\|g_\varepsilon V\|_{\dot{H}^\alpha(\R^3)}\right)\|g_\varepsilon  \vartriangle^h_kV\|_{\dot{H}^\alpha(\R^3)}.
\end{align*}
Since
\begin{align*}
\|g_\varepsilon V\|_{\dot{H}^\alpha(\R^3)}\leq\|g_\varepsilon \Lambda^\alpha V\|_{L^2(\R^3)}+\big\|[\Lambda^\alpha,g_\varepsilon]V\big\|_{L^2(\R^3)}
\leq\|g_\varepsilon \Lambda^\alpha V\|_{L^2(\R^3)}+C\|V\|_{L^2(\R^3)},
\end{align*}
and
\begin{align*}
\|g_\varepsilon \vartriangle^h_kV\|_{\dot{H}^\alpha(\R^3)}\leq\big\|g_\varepsilon \Lambda^\alpha \vartriangle^h_k V\big\|_{L^2(\R^3)}+C\|\vartriangle^h_kV\|_{L^2(\R^3)},
\end{align*}
we have by the Young inequality that
\begin{align}\label{add-eq-006}
\begin{split}
J_{6_1}
\leq&C\|V\|^2_{H^\alpha(\R^3)}+C\|V\|_{H^\alpha(\R^3)}\big\|g_\varepsilon \Lambda^\alpha \vartriangle^h_k V\big\|_{L^2(\R^3)}\\
\leq&C\|V\|^2_{H^\alpha(\R^3)}+\frac{1}{16}\big\|g_\varepsilon \Lambda^\alpha \vartriangle^h_k V\big\|^2_{L^2(\R^3)}.
\end{split}
\end{align}
Collecting estimates \eqref{add-eq-005}-\eqref{add-eq-006}, we  obtain
$$J_6\leq C\|V\|^2_{H^\alpha(\R^3)}+\frac{1}{16}\big\|g_\varepsilon \Lambda^\alpha \vartriangle^h_k V\big\|^2_{L^2(\R^3)}. $$

Next,  we start to tackle with the term involving the pressure. We rewrite $J_7$ as follows
\begin{align*}
J_7=\int_{\R^3}\vartriangle_k^{h} P_1\,\big(\nabla g^2_\varepsilon\cdot \vartriangle^h_kV\big)\,\mathrm{d}x
+\int_{\R^3}\vartriangle_k^{h} P_2\,\big(\nabla g^2_\varepsilon\cdot \vartriangle^h_kV\big)\,\mathrm{d}x:=J_{7_1}+J_{7_2},
\end{align*}
where $P_1$ satisfies
$$-\Delta P_1=\text{div}\,\text{div}\,\left(V\otimes V+U_0\otimes V\right)$$
and $P_2$ satisfies
$$-\Delta P_2=\text{div}\,\left(V\cdot\nabla U_0+U_0\cdot\nabla U_0\right).$$
Since $\nabla g^2_\varepsilon=- \frac{2\varepsilon x }{(1+\varepsilon|x|^2)^2},$
we obtain by using the H\"older inequality and the Young inequality that
\begin{align}\label{add-eq-007}
\begin{split}
J_{7_2}\leq&\big\|\vartriangle_k^{h} P_2\big\|_{L^2(\R^3)}\big\|g_\varepsilon \vartriangle^h_kV\big\|_{L^2(\R^3)}\\
\leq &C\|U_0\|^2_{L^\infty(\R^3)}\|\nabla U_0\|^2_{L^2(\R^{3})}+C\|V\|_{L^4(\R^3)}^2\|\nabla U_0\|_{L^4(\R^3)}^2+\frac{1-\alpha}{16\alpha}\big\|g_\varepsilon \vartriangle^h_kV\big\|^2_{L^2(\R^3)}.
\end{split}
\end{align}
Note that
\begin{align*}
J_{7_1}=
\int_{\R^3}  \vartriangle_k^{h} P_1\,\Big(\nabla g^2_\varepsilon \vartriangle^h_kV\Big)\,\mathrm{d}x
=-
\int_{\R^3}   P_1\,\vartriangle_k^{-h}\Big(\nabla g^2_\varepsilon \vartriangle^h_kV\Big)\,\mathrm{d}x,
\end{align*}
we have by the H\"older inequality that
\begin{align*}
J_{7_1}\leq&
\|P_1\|_{L^2(\R^3)}\Big\|\vartriangle_k^{-h}\Big(\nabla g^2_\varepsilon\vartriangle^h_kV\Big)\Big\|_{L^2(\R^3)}\\
\leq&C\left(\|V\|_{L^4(\R^3)}^2+\|U_0\|_{L^4(\R^3)}^2\right)\Big\|\Lambda^\alpha\Big(\nabla g^2_\varepsilon \vartriangle^h_kV\Big)\Big\|_{L^2(\R^3)}\\
\leq&C\left(\|V\|_{L^4(\R^3)}^2+\|U_0\|_{L^4(\R^3)}^2\right)\Big(\Big\|\nabla g^2_\varepsilon \Lambda^\alpha \vartriangle^h_kV\Big\|_{L^2(\R^3)}+\Big\|\big[ \Lambda^\alpha, \nabla g^2_\varepsilon\big] \vartriangle^h_kV\Big\|_{L^2(\R^3)}\Big)\\
\leq&C\left(\|V\|_{L^4(\R^3)}^2+\|U_0\|_{L^4(\R^3)}^2\right)\left(\big\|g_\varepsilon\Lambda^\alpha \vartriangle^h_kV\big\|_{L^2(\R^3)}+\big\|\Lambda^\alpha V\big\|_{L^2(\R^3)}\right).
\end{align*}
Here we have used the following estimate deduced by Lemma \ref{lem-Comm} that
\begin{equation}\label{c-2}
\big\|\big[ \Lambda^\alpha, \nabla g^2_\varepsilon\big] \vartriangle^h_kV\big\|_{L^2(\R^3)}\leq \|\nabla g^2_{\varepsilon}\|_{W^{1,\infty}(\R^3)}\big\|\vartriangle^h_kV\big\|_{L^2(\R^3)}.
\end{equation}
Moreover, by the Young inequality, we immediately have
\begin{align}\label{add-eq-008}
\begin{split}
J_{7_1}\leq&C\big(\|V\|_{L^4(\R^3)}^2+\|U_0\|_{L^4(\R^3)}^2\big)\big\|\Lambda^\alpha V\big\|_{L^2(\R^3)}
+C\big(\|V\|_{L^4(\R^3)}^2+\|U_0\|_{L^4(\R^3)}^2\big)^2\\
&+\frac{1}{16}\big\|g_\varepsilon\Lambda^\alpha \vartriangle^h_kV\big\|^2_{L^2(\R^3)}.
\end{split}
\end{align}
This inequality together with \eqref{add-eq-007} yields that
 \begin{align*}
 J_7\le &C\big(\|V\|_{L^4(\R^3)}^2+\|U_0\|_{L^4(\R^3)}^2\big)\big\|\Lambda^\alpha V\big\|_{L^2(\R^3)}
+C\big(\|V\|_{L^4(\R^3)}^2+\|U_0\|_{L^4(\R^3)}^2\big)^2\\
 &+C\|U_0\|^2_{L^\infty(\R^3)}\|\nabla U_0\|^2_{L^2(\R^{3})}+C\|V\|_{L^4(\R^3)}^2\|\nabla U_0\|_{L^4(\R^3)}^2\\
 &+\frac{1-\alpha}{16\alpha}\big\|g_\varepsilon \vartriangle^h_kV\big\|^2_{L^2(\R^3)}+\frac{1}{16}\big\|g_\varepsilon\Lambda^\alpha \vartriangle^h_kV\big\|^2_{L^2(\R^3)}.\end{align*}
Next, we rewrite $I_8$ as follows
\begin{align*}
J_8=-\int_{\R^3}
V\cdot\nabla \vartriangle_k^{h}V\,\big(g^2_\varepsilon \vartriangle^h_kV\big)\,\mathrm{d}x-\int_{\R^3}
\vartriangle_k^{h}V(x+he_k)\cdot\nabla V\,\big(g^2_\varepsilon \vartriangle^h_kV\big)\,\mathrm{d}x:=J_{8_1}+J_{8_2}.
\end{align*}
Thanks to the incompressible condition $\text{div}\, V=0,$ we have
\begin{align*}
J_{8_1}=&-\int_{\R^3}
V\cdot\nabla \big(g_\varepsilon \vartriangle^h_kV\big)\,\big(g_\varepsilon \vartriangle^h_kV\big)\,\mathrm{d}x+\int_{\R^3}
V\cdot\nabla g_\varepsilon\, \vartriangle^h_kV\,\big(g_\varepsilon \vartriangle^h_kV\big)\,\mathrm{d}x\\
=&-\int_{\R^3}
V\cdot\frac{\varepsilon x}{(1+\varepsilon|x|^2)^{\frac32}} \, \vartriangle^h_kV\,\big(g_\varepsilon \vartriangle^h_kV\big)\,\mathrm{d}x.
\end{align*}
By the H\"older inequality and the interpolation inequality, we get
\begin{align*}
J_{8_1}\leq &\sqrt{\varepsilon}\|V\|_{L^3(\R^3)}\big\|g_\varepsilon \vartriangle^h_kV\big\|^2_{L^3(\R^3)}\\
\leq&\sqrt{\varepsilon}\|V\|_{L^3(\R^3)}\big\|g_\varepsilon \vartriangle^h_kV\big\|^{2-\frac{1}{\alpha}}_{L^2(\R^3)}\big\|\Lambda^\alpha(g_\varepsilon \vartriangle^h_kV)\big\|^{\frac1\alpha}_{L^2(\R^3)}\\
\leq&\sqrt{\varepsilon}\|V\|_{L^3(\R^3)}\big\|g_\varepsilon \vartriangle^h_kV\big\|^{2-\frac{1}{\alpha}}_{L^2(\R^3)}\Big(\big\|g_\varepsilon\Lambda^\alpha
\vartriangle^h_kV\big\|^{\frac1\alpha}_{L^2(\R^3)}+\big\|[\Lambda^\alpha,g_\varepsilon] \vartriangle^h_kV\big\|^{\frac1\alpha}_{L^2(\R^3)}\Big).
\end{align*}
Moreover, we have by \eqref{c-2} that
\begin{align}\label{add-eq-009}
\begin{split}
J_{8_1}\leq& \|V\|^{\frac{2\alpha}{2\alpha-1}}_{L^3(\R^3)}\big\|g_\varepsilon \vartriangle^h_kV\big\|^{2}_{L^2(\R^3)}+\|V\|_{L^3(\R^3)}\big\|g_\varepsilon \vartriangle^h_kV\big\|^{2}_{L^2(\R^3)}+\frac{1}{16}\big\|g_\varepsilon\Lambda^\alpha \vartriangle^h_kV\big\|^{2}_{L^2(\R^3)}\\
\leq& \|V\|^{\frac{2\alpha}{2\alpha-1}}_{L^3(\R^3)}\big\| V\big\|^{2}_{\dot{H}^\alpha(\R^3)}+\|V\|_{L^3(\R^3)}\big\|V\big\|^{2}_{\dot{H}^\alpha(\R^3)}+\frac{1}{16}\big\|g_\varepsilon\Lambda^\alpha \vartriangle^h_kV\big\|^{2}_{L^2(\R^3)}.
\end{split}
\end{align}
On the other hand,  we see that
\[J_{8_2}=-\int_{\R^3}
\text{div}\,\big(g_\varepsilon \vartriangle_k^{h}V(x+he_k)V\big)\big(g_\varepsilon \vartriangle^h_kV\big)\,\mathrm{d}x+\int_{\R^3}
\vartriangle_k^{h}V(x+he_k)\cdot\nabla g_\varepsilon V\big(g_\varepsilon
\vartriangle^h_kV\big)\,\mathrm{d}x.\]
Furthermore,  we obtain by the H\"older inequality that
\begin{align*}
J_{8_2}
\leq&\big \|\text{div}\,\big(g_\varepsilon \vartriangle_k^{h}V(\cdot+he_k)V\big)\big\|_{\dot H^{-\alpha}(\R^3)}\big\|g_\varepsilon \vartriangle^h_kV\big\|_{\dot H^\alpha(\R^3)}\\
&+\|\nabla g_\varepsilon\|_{L^\infty(\R^3)}\big\|\vartriangle_k^{h}V(\cdot+he_k)\big\|_{L^2(\R^3)}\|V\|_{L^4(\R^3)}
\|g_\varepsilon
\vartriangle^h_kV\|_{L^4(\R^3)}\\
\leq&\big \|g_\varepsilon \vartriangle_k^{h}V(\cdot+he_k)V\big\|_{\dot H^{1-\alpha}(\R^3)}\big\|g_\varepsilon \Lambda^\alpha \vartriangle^h_kV\big\|_{L^2(\R^3)}\\
&+\big \|g_\varepsilon \vartriangle_k^{h}V(\cdot+he_k)V\big\|_{\dot H^{1-\alpha}(\R^3)}\big\|[g_\varepsilon, \Lambda^\alpha ]\vartriangle^h_kV\big\|_{L^2(\R^3)}\\
&+\|\nabla g_\varepsilon\|_{L^\infty(\R^3)}\big\|\vartriangle_k^{h}V(\cdot+he_k)\big\|_{L^2(\R^3)}\|V\|_{L^4(\R^3)}
\|g_\varepsilon \vartriangle^h_kV\|^{\frac{4\alpha-3}{4\alpha}}_{L^2(\R^3)}\|g_\varepsilon \vartriangle^h_kV\|^{\frac{3}{4\alpha}}_{\dot{H}^\alpha(\R^3)},\end{align*}
it follows that
\begin{align*}
J_{8_2}\leq&\big \|g_\varepsilon \vartriangle_k^{h}V(\cdot+he_k)V\big\|_{\dot H^{1-\alpha}(\R^3)}\big(\big\|g_\varepsilon \Lambda^\alpha \vartriangle^h_kV\big\|_{L^2(\R^3)}
+\big\| \Lambda^\alpha V\big\|_{L^2(\R^3)}\big)\\
&+\|\nabla g_\varepsilon\|_{L^\infty(\R^3)}\big\|V\big\|_{\dot{H}^\alpha(\R^3)}\|V\|_{L^4(\R^3)}
\|g_\varepsilon \vartriangle^h_kV\|^{\frac{4\alpha-3}{4\alpha}}_{L^2(\R^3)}\|g_\varepsilon \Lambda^\alpha \vartriangle^h_k V\|^{\frac{3}{4\alpha}}_{L^2(\R^3)}
\\
&+\|\nabla g_\varepsilon\|_{L^\infty(\R^3)}\big\|V\big\|_{\dot{H}^\alpha(\R^3)}\|V\|_{L^4(\R^3)}
\|g_\varepsilon \vartriangle^h_kV\|^{\frac{4\alpha-3}{4\alpha}}_{L^2(\R^3)}\|[g_\varepsilon,\Lambda] \vartriangle^h_kV\|^{\frac{3}{4\alpha}}_{L^2(\R^3)}.
\end{align*}
Thanks to the Bony para-product decomposition, one has
\begin{align*}
&\big \|g_\varepsilon \vartriangle_k^{h}V(\cdot+he_k)V\big\|_{\dot H^{1-\alpha}(\R^3)}\\
\leq&\big \|g_\varepsilon \vartriangle_k^{h}V(\cdot+he_k)\big\|_{L^{\frac{3}{2\alpha-1}}(\R^3)}\big\|V\big\|_{\dot W^{1-\alpha,\frac{6}{5-4\alpha}}(\R^3)}\\&+\|g_\varepsilon \vartriangle_k^{h}V(\cdot+he_k)\big\|_{\dot W^{1-\alpha,\frac{3}{\alpha}}(\R^3)}\big\|V\big\|_{L^\frac{6}{3-2\alpha}(\R^3)}\\
\leq& \big \|g_\varepsilon \vartriangle_k^{h}V(\cdot+he_k)\big\|^{\frac{6\alpha-5}{2\alpha}}_{L^{2}(\R^3)}\big \|g_\varepsilon \vartriangle_k^{h}V(\cdot+he_k)\big\|^{\frac{5-4\alpha}{2\alpha}}_{\dot{H}^{\alpha}(\R^3)}\big\|V\big\|_{\dot H^{\alpha}(\R^3)}\\&+\big \|g_\varepsilon \vartriangle_k^{h}V(\cdot+he_k)\big\|^{\frac{6\alpha-5}{2\alpha}}_{L^{2}(\R^3)}\big \|g_\varepsilon \vartriangle_k^{h}V(\cdot+he_k)\big\|^{\frac{5-4\alpha}{2\alpha}}_{\dot{H}^{\alpha}(\R^3)}\big\|V\big\|_{\dot{H}^\alpha(\R^3)}\\
\leq&\big \| V\big\|^{\frac{8\alpha-5}{2\alpha}}_{\dot{H}^{\alpha}(\R^3)}\big \|g_\varepsilon \vartriangle_k^{h}V(\cdot+he_k)\big\|^{\frac{5-4\alpha}{2\alpha}}_{\dot{H}^{\alpha}(\R^3)}.
\end{align*}
By the triangle inequality,   we observe  that  for each $h\in [-1,1],$
\begin{align*}
&\big \|g_\varepsilon \vartriangle_k^{h}V(\cdot+he_k)\big\|_{\dot{H}^{\alpha}(\R^3)}\\
\leq&\big \|[\Lambda,g_\varepsilon] \vartriangle_k^{h}V(\cdot+he_k)\big\|^{\frac{5-4\alpha}{2\alpha}}_{\dot{H}^{\alpha}(\R^3)}\\
&+\left(\int_{\R^3}\frac{1+\varepsilon|x+he_k|^2}{1+\varepsilon|x|^2}\frac{1}{1+\varepsilon|x+he_k|^2}\left| \Lambda^\alpha \vartriangle_k^{h}V(\cdot+he_k) \right|^2\,\mathrm{d}x\right)^{\frac{1}{2}}\\
\leq&\big \|V\big\|_{\dot{H}^{\alpha}(\R^3)}+2\big\| g_\varepsilon \Lambda^\alpha \vartriangle_k^hV\big\|_{L^2(\R^3)}.
\end{align*}
This implies
\begin{equation*}
\big \|g_\varepsilon \vartriangle_k^{h}V(\cdot+he_k)V\big\|_{\dot H^{1-\alpha}(\R^3)}
\leq C\big \|V\big\|^2_{\dot{H}^{\alpha}(\R^3)}+C \|V\big\|^{\frac{8\alpha-5}{2\alpha}}_{\dot{H}^{\alpha}(\R^3)}\big\| g_\varepsilon \Lambda^\alpha \vartriangle_k^hV\big\|^{\frac{5-4\alpha}{2\alpha}}_{L^2(\R^3)}.
\end{equation*}
Therefore we have by the Young inequality that
\begin{align*}
J_{8_2}
\leq&C\|V\|^3_{H^\alpha(\R^3)}+C \|V\big\|^{\frac{8\alpha-5}{2\alpha}}_{\dot{H}^{\alpha}(\R^3)}\big\| g_\varepsilon \Lambda^\alpha \vartriangle_k^hV\big\|^{\frac{5-4\alpha}{2\alpha}}_{L^2(\R^3)}\\&
+\big\|V\big\|_{\dot{H}^\alpha(\R^3)}\|V\|_{L^4(\R^3)}
\|g_\varepsilon \vartriangle^h_kV\|^{\frac{4\alpha-3}{4\alpha}}_{L^2(\R^3)}\|g_\varepsilon \Lambda^\alpha \vartriangle^h_kV\|^{\frac{3}{4\alpha}}_{L^2(\R^3)}\\
\leq&C\|V\|^3_{H^\alpha(\R^3)}+C \|V\big\|^{\frac{16\alpha-10}{6\alpha-5}}_{\dot{H}^{\alpha}(\R^3)}
+\big\|V\big\|^{\frac{3(4\alpha-1)}{8\alpha-3}}_{\dot{H}^\alpha(\R^3)} +\frac{1}{16}\big\| g_\varepsilon \Lambda^\alpha \vartriangle_k^hV\big\|^{2}_{L^2(\R^3)}.
\end{align*}
Combining the above inequalities  with estimate \eqref{add-eq-009}, we get
 \begin{align*}
J_8\le&\|V\|^{\frac{2\alpha}{2\alpha-1}}_{L^3(\R^3)}\big\| V\big\|^{2}_{\dot{H}^\alpha(\R^3)}+\|V\|_{L^3(\R^3)}\big\|V\big\|^{2}_{\dot{H}^\alpha(\R^3)}+\frac{1}{8}\big\|g_\varepsilon\Lambda^\alpha \vartriangle^h_kV\big\|^{2}_{L^2(\R^3)}\\
 &+C\|V\|^3_{H^\alpha(\R^3)}+C \|V\big\|^{\frac{16\alpha-10}{6\alpha-5}}_{\dot{H}^{\alpha}(\R^3)}
+\big\|V\big\|^{\frac{3(4\alpha-1)}{8\alpha-3}}_{\dot{H}^\alpha(\R^3)}.
\end{align*}
As for $I_9$, we decompose  it into two parts as follows
\begin{align*}
J_9=-\int_{\R^3}
U_0\cdot\nabla \vartriangle_k^{h}V\,\big(g^2_\varepsilon \vartriangle^h_kV\big)\,\mathrm{d}x-\int_{\R^3}
\vartriangle_k^{h}U_0(x+he_k)\cdot\nabla V\,\big(g^2_\varepsilon \vartriangle^h_kV\big)\,\mathrm{d}x:=J_{9_1}+J_{9_2}.
\end{align*}
Since $\text{div}\,V=0$, we have
\begin{align*}
J_{9_1}=&-\int_{\R^3}
U_0\cdot\nabla(g_\varepsilon \vartriangle_k^{h}V)\,\big(g_\varepsilon \vartriangle^h_kV\big)\,\mathrm{d}x+\int_{\R^3}
U_0\cdot\nabla g_\varepsilon \,\vartriangle_k^{h}V\,\big(g_\varepsilon \vartriangle^h_kV\big)\,\mathrm{d}x\\
=&-2\int_{\R^3}
U_0\cdot\frac{\varepsilon x}{1+\varepsilon|x|^2} \,\vartriangle_k^{h}V\,\big(g_\varepsilon \vartriangle^h_kV\big)\,\mathrm{d}x.
\end{align*}
By the H\"older inequality and Lemma \ref{lem-equi}, we find that
\begin{align}\label{add-eq-0010}
J_{9_1}&\leq\|U_0\|_{L^\infty(\R^3)}\|g_\varepsilon\|_{L^\infty(\R^3)}\big\|\vartriangle_k^{h}V\big\|^2_{L^2(\R^3)}
\leq\|U_0\|_{L^\infty(\R^3)}\big\|V\big\|^2_{\dot{H}^{\alpha}(\R^3)}.
\end{align}
Similarly, we can show that
\begin{align*}
J_{9_2}
\leq&\big\|\vartriangle_k^{h}U_0(\cdot+he_k)\cdot\nabla V\big\|_{H^{-\alpha}(\R^3)}\big\|g^2_\varepsilon \vartriangle^h_kV\big\|_{\dot{H}^\alpha(\R^3)}\\
\leq&\big\|\vartriangle_k^{h}U_0(\cdot+he_k)\big\|_{L^\infty(\R^3)}\|V\|_{\dot{H}^{1-\alpha}(\R^3)}\big\|g_\varepsilon^2\Lambda^\alpha \vartriangle^h_kV\big\|_{L^2(\R^3)}\\
&+\big\|\vartriangle_k^{h}U_0(\cdot+he_k)\big\|_{L^\infty(\R^3)}\|V\|_{\dot{H}^{1-\alpha}(\R^3)}\big\|[g_\varepsilon^2,\Lambda^\alpha] \vartriangle^h_kV\big\|_{L^2(\R^3)}\\
\leq&\|U_0\|_{\dot{W}^{\alpha,\infty}(\R^3)}\|V\|_{{H}^\alpha(\R^3)}\big\|g_\varepsilon\Lambda^\alpha D^h_kV\big\|_{L^2(\R^3)}+\|U_0\|_{\dot{W}^{\alpha,\infty}(\R^3)}\|V\|^2_{{H}^\alpha(\R^3)}.
\end{align*}
Thus we have by the Young inequality that
\begin{align*}
J_{9_2}\leq C\left(\|U_0\|_{\dot{W}^{\alpha,\infty}(\R^3)} +\|U_0\|^2_{\dot{W}^{\alpha,\infty}(\R^3)}\right)\|V\|^2_{{H}^\alpha(\R^3)}+\frac{1}{16}\big\|g_\varepsilon\Lambda^\alpha D^h_kV\big\|^2_{L^2(\R^3)}.
\end{align*}
Combining this inequality with \eqref{add-eq-0010} yields
$$J_9\le C\left(\|U_0\|_{\dot{W}^{\alpha,\infty}(\R^3)} +\|U_0\|^2_{\dot{W}^{\alpha,\infty}(\R^3)}+\|g_\varepsilon\|_{L^\infty(\R^3)}\right)\|V\|^2_{{H}^\alpha(\R^3)}
+\frac{1}{16}\big\|g_\varepsilon\Lambda^\alpha D^h_kV\big\|^2_{L^2(\R^3)}.
$$
For $I_{10}$, one writes
\[J_{10}=\int_{\R^3}
V\cdot\nabla U_0\,\vartriangle_k^{-h}\big(g^2_\varepsilon \vartriangle^h_kV\big)\,\mathrm{d}x+\int_{\R^3}
U_0\cdot\nabla U_0\,\vartriangle_k^{-h}\big(g^2_\varepsilon \vartriangle^h_kV\big)\,\mathrm{d}x:=J_{{10}_1}+J_{{10}_2}.\]
By the H\"older inequality and the Leibniz estimate, we finally obtain that
\begin{align*}
J_{{10}_1}
=&\int_{\R^3}
\vartriangle_k^{h}(V\cdot\nabla U_0)\,\big(g^2_\varepsilon \vartriangle^h_kV\big)\,\mathrm{d}x\\
\leq&\|g_\varepsilon\|^2_{L^\infty(\R^3)}\big\|\vartriangle_k^{h}(V\cdot\nabla U_0)\big\|_{L^2(\R^3)}\big\| \vartriangle^h_kV\big\|_{L^2(\R^3)}\\
\leq&C\left(\|\nabla U_0\|_{L^\infty(\R^3)}\|V\|_{\dot{H}^\alpha(\R^3)}+\|U_0\|_{\dot{B}^{1+\alpha}_{\infty,\infty}(\R^3)}\|V\|_{L^2(\R^3)}\right)\big\|  V\big\|_{\dot{H}^\alpha(\R^3)},
\end{align*}
and
\begin{align*}
J_{{10}_2}
=&\int_{\R^3}\vartriangle_k^{h}
\big(U_0\cdot\nabla U_0\big)\, g^2_\varepsilon \vartriangle^h_kV \,\mathrm{d}x\\
\leq&\|g_\varepsilon\|^2_{L^\infty(\R^3)}\big\|\vartriangle_k^{h}(U_0\cdot\nabla U_0)\big\|_{L^2(\R^3)}\big\| \vartriangle^h_kV\big\|_{L^2(\R^3)}\\
\leq&C\left(\|\nabla U_0\|_{L^\infty(\R^3)}\|U_0\|_{\dot{H}^\alpha(\R^3)}+\|U_0\|_{L^\infty(\R^3)}\|U_0\|_{\dot{H}^{1+\alpha}(\R^3)}\right)\big\|  V\big\|_{\dot{H}^{\alpha}(\R^3)}.
\end{align*}
Collecting  both estimates and then inserting the resulting estimate into \eqref{eq-w-h}, we eventually obtain that
\begin{align*}
\nu\big\|g_\varepsilon \vartriangle^h_k \nabla V\big\|^2_{L^2(\mathbb{R}^3)}+\big\| g_\varepsilon \Lambda^\alpha \vartriangle^h_kV  \big\|_{L^2(\R^3)}^2+\frac{5-4\alpha}{4\alpha}\big\|g_\varepsilon \vartriangle^h_k  V\big\|^2_{L^2(\mathbb{R}^3)}\leq C\big(U_0\big).
\end{align*}
Taking $\nu\to 0$ and $\varepsilon\to0$, we immediately have by the Lebesgue dominated convergence theorem that
\begin{align*}
\big\|  \Lambda^\alpha \vartriangle^h_kV  \big\|_{L^2(\R^3)}^2+\frac{5-4\alpha}{4\alpha}\big\| \vartriangle^h_k  V\big\|^2_{L^2(\mathbb{R}^3)}\leq C\big(U_0\big).
\end{align*}
Taking supremum with respective to $h$, we finally obtain by Lemma \ref{lem-equi} that
\begin{equation*}
\|V\|_{B^{2\alpha}_{2,\infty}(\R^3) }\leq C\big(U_0\big).
\end{equation*}
This estimate implies the desired estimate in Proposition \ref{prop3.2}.
\end{proof}

\subsection{$	\big\|V\big\|_{H^{1+\alpha}(\R^3)}$-estimate}

 The third  step is to establish the regularity of $V$ in the Sobolev space with higher order  by tha so-called bootstrap argument.
\begin{proposition}\label{prop3.3}
	Let  $\alpha\in(5/6,1)$, and $V\in H_{\sigma}^{\alpha}(\R^{3})$ be the weak solution which was established in Theorem~\ref{thm-1}. Then there exists $C>0$ such that
	\begin{equation}
	\big\|V\big\|_{H^{1+\alpha}(\R^3)}<C\left(U_0\right).
	\end{equation}
\end{proposition}
\begin{proof}
	Thanks to Proposition \ref{prop3.2}, we know that $\nabla v\in L^2(\R^3)$  and then we can write \eqref{eq.weak-ESTIMATE} as  follows
\begin{align}\label{eq.weak-IIII}
\begin{split}
&\int_{\mathbb{R}^3}\Lambda^\alpha V:\Lambda^\alpha\varphi\,\mathrm{d}x-\frac{1}{2\alpha}\int_{\mathbb{R}^3}x\cdot \nabla V\cdot\varphi\,\mathrm{d}x-\frac{2\alpha-1}{2\alpha}\int_{\mathbb{R}^3} V\cdot\varphi\,\mathrm{d}x\\
=&\int_{\mathbb{R}^3}P\, \mathrm{div}\,\varphi\,\mathrm{d}x+\int_{\mathbb{R}^3}V\cdot \nabla \varphi\cdot V\,\mathrm{d}x-\int_{\mathbb{R}^3}U_0\cdot \nabla  V\cdot \varphi\,\mathrm{d}x-\int_{\mathbb{R}^3}(V+U_0)\cdot \nabla  U_0\cdot \varphi\,\mathrm{d}x
\end{split}
\end{align}
for each $\varphi\in H^1_{\omega}(\R^3)$. Since $ \|V\|_{B^{2\alpha}_{2,\infty}(\R^3) }\leq C\big(U_0\big)$,
it is easy to check that $\vartriangle_k^{-h}\big(g^2_\varepsilon \vartriangle_k^{h}V\big)\in H^1_{\omega}(\R^3)$.
Here and in Proposition \ref{key}, we denote by $\vartriangle_k^{-h}V$ the difference quotient
$$\vartriangle_k^{-h}V=\frac{V(x+he_k)-V(x)}{h}, \quad\; h\in \mathbb R,\;\; h\not=0.$$
	Choosing the test function $\varphi=-\vartriangle_k^{-h}\big(g^2_\varepsilon \vartriangle_k^{h}V\big)$ in equality \eqref{eq.weak-IIII},  we immediately have
\begin{align}	\label{add-eq-0012}
	\begin{split}
	&-\int_{\mathbb{R}^3}\Lambda^\alpha V:\Lambda^\alpha\vartriangle_k^{-h}\big(g^2_\varepsilon \vartriangle_k^{h}V\big)\,\mathrm{d}x+\frac{1}{2\alpha}\int_{\mathbb{R}^3}x\cdot \nabla V\cdot\vartriangle_k^{-h}\big(g^2_\varepsilon \vartriangle_k^{h}V\big)\,\mathrm{d}x\\
		&+\frac{2\alpha-1}{2\alpha}\int_{\mathbb{R}^3} V\cdot\vartriangle_k^{-h}\big(g^2_\varepsilon \vartriangle_k^{h}V\big)\,\mathrm{d}x\\
		=&-\int_{\mathbb{R}^3}P\, \mathrm{div}\,\vartriangle_k^{-h}\big(g^2_\varepsilon \vartriangle_k^{h}V\big)\,\mathrm{d}x-\int_{\mathbb{R}^3}V\cdot \nabla\vartriangle_k^{-h}\big(g^2_\varepsilon \vartriangle_k^{h}V\big)\cdot V\,\mathrm{d}x\\
&+\int_{\mathbb{R}^3}U_0\cdot \nabla  V\cdot\vartriangle_k^{-h}\big(g^2_\varepsilon \vartriangle_k^{h}V\big)\,\mathrm{d}x+\int_{\mathbb{R}^3}(V+U_0)\cdot \nabla  U_0\cdot\vartriangle_k^{-h}\big( g^2_\varepsilon \vartriangle_k^{h}V\big)\,\mathrm{d}x.
\end{split}
\end{align}
First of all, we compute
 the term including the fractional operator to obtain
\begin{align*}
&-\int_{\R^3}\Lambda^\alpha V\cdot\Lambda^\alpha\,\vartriangle_k^{-h}\Big(g^2_\varepsilon \vartriangle_k^{h}V\Big)\,\mathrm{d}x\\
=&\int_{\R^3}g_\varepsilon \Lambda^\alpha\vartriangle_k^{h}V\cdot\, \Lambda^\alpha \big(g_\varepsilon \vartriangle_k^{h}V \big)\,\mathrm{d}x-\int_{\R^3} \Lambda^\alpha\vartriangle_k^{h}V\cdot\, \big([g_\varepsilon,\Lambda^\alpha] g_\varepsilon \vartriangle_k^{h}V \big)\,\mathrm{d}x\\
=&\big\| g_\varepsilon \Lambda^\alpha \vartriangle_k^{h}V  \big\|_{L^2(\R^3)}^2-\int_{\R^3}g_\varepsilon \Lambda^\alpha\vartriangle_k^{h}V\cdot\, \big([g_\varepsilon,\Lambda^\alpha]  \vartriangle_k^{h}V  \big)\,\mathrm{d}x\\&-\int_{\R^3} \Lambda^\alpha\vartriangle_k^{h}V\cdot\, \big([g_\varepsilon,\Lambda^\alpha] g_\varepsilon \vartriangle_k^{h}V \big)\,\mathrm{d}x.
\end{align*}
Integrating by parts, we rewrite the second and third term of right side of \label{add-eq-0012} to be
	\begin{align*}
&\frac{1}{2\alpha}\int_{\mathbb{R}^3}x\cdot \nabla V\cdot\vartriangle_k^{-h}\Big(g_\varepsilon^2\vartriangle_k^{h}V\Big)\,\mathrm{d}x+\frac{2\alpha-1}{2\alpha}\int_{\mathbb{R}^3} V\cdot\vartriangle_k^{-h}\Big(g_\varepsilon^2\vartriangle_k^{h}V\Big)\,\mathrm{d}x\\
=&-	\frac{1}{2\alpha}\int_{\mathbb{R}^3}g_{\varepsilon}^2(x+h\mathbf{e}_k)\cdot \nabla\vartriangle_k^{h} V\cdot \vartriangle_k^{h}V\,\mathrm{d}x-	 \frac{1}{2\alpha}\int_{\mathbb{R}^3}g_\varepsilon^2\big(\vartriangle_k^{h}x\big)\cdot \nabla V\cdot \vartriangle_k^{h}V\,\mathrm{d}x\\
&-\frac{2\alpha-1}{2\alpha}\int_{\mathbb{R}^3}g^2_\varepsilon \vartriangle_k^{h}V\cdot \vartriangle_k^{h}V\,\mathrm{d}x\\
=&\frac{5-4\alpha}{4\alpha}\big\|g_\varepsilon \vartriangle_k^{h}  V\big\|^2_{L^2(\mathbb{R}^3)}+\frac{1}{4\alpha}\int_{\R^3}x\cdot\nabla  g_\varepsilon^2\,\vartriangle_k^{h} V\cdot \vartriangle_k^{h}V\,\mathrm{d}x
-	\frac{h}{2\alpha}\int_{\mathbb{R}^3}g_\varepsilon^2 \partial_{x_k}\vartriangle_k^{h} V\cdot \vartriangle_k^{h}V\,\mathrm{d}x\\&-	 \frac{1}{2\alpha}\int_{\mathbb{R}^3}g_\varepsilon^2 \partial_{x_k}V\cdot \vartriangle_k^{h}V\,\mathrm{d}x.\end{align*}
From this, it follows that
\begin{align*}
&\frac{1}{2\alpha}\int_{\mathbb{R}^3}x\cdot \nabla V\cdot\vartriangle_k^{-h}\Big(g_\varepsilon^2\vartriangle_k^{h}V\Big)\,\mathrm{d}x+\frac{2\alpha-1}{2\alpha}\int_{\mathbb{R}^3} V\cdot\vartriangle_k^{-h}\Big(g_\varepsilon^2\vartriangle_k^{h}V\Big)\,\mathrm{d}x\\
=&\frac{5-4\alpha}{4\alpha}\big\|g_\varepsilon \vartriangle_k^{h}  V\big\|^2_{L^2(\mathbb{R}^3)} -\frac{1}{2\alpha}\int_{\R^3} \frac{\varepsilon |x|^2}{(1+\varepsilon|x|^2)^2}\,\vartriangle_k^{h} V\cdot \vartriangle_k^{h}V\,\mathrm{d}x\\
&+\frac{h}{4\alpha}\int_{\mathbb{R}^3}\partial_{x_k}g_\varepsilon^2\vartriangle_k^{h} V\cdot \vartriangle_k^{h}V\,\mathrm{d}x+\frac{1}{2\alpha}\int_{\mathbb{R}^3}g_\varepsilon^2 \partial_{x_k}V\cdot \vartriangle_k^{h}V\,\mathrm{d}x.
	\end{align*}
Summing up the above inductions, we have from \eqref{eq.weak-IIII} that
	 \begin{align*}
	 & \big\| g_\varepsilon \Lambda^\alpha \vartriangle_k^{h}V  \big\|_{L^2(\R^3)}^2+\frac{5-4\alpha}{4\alpha}\big\|g_\varepsilon \vartriangle_k^{h}  V\big\|^2_{L^2(\mathbb{R}^3)}\\
	\leq & \int_{\R^3}g_\varepsilon \Lambda^\alpha\vartriangle_k^{h}V\cdot\, \big([g_\varepsilon,\Lambda^\alpha]  \vartriangle_k^{h}V  \big)\,\mathrm{d}x+\int_{\R^3} \Lambda^\alpha\vartriangle_k^{h}V\cdot\, \big([g_\varepsilon,\Lambda^\alpha] g_\varepsilon \vartriangle_k^{h}V \big)\,\mathrm{d}x\\&+\frac{1}{2\alpha}\int_{\R^3} \frac{\varepsilon |x|^2}{1+\varepsilon|x|^2}\,\vartriangle_k^{h} V\cdot \vartriangle_k^{h}V\,\mathrm{d}x-\frac{h}{4\alpha}\int_{\mathbb{R}^3}\partial_{x_k}g_\varepsilon^2\vartriangle_k^{h} V\cdot \vartriangle_k^{h}V\,\mathrm{d}x\\
	&+\frac{1}{2\alpha}\int_{\mathbb{R}^3}g_\varepsilon^2 \partial_{x_k}V\cdot \vartriangle_k^{h}V\,\mathrm{d}x+\int_{\R^3}\nabla P\,\vartriangle_k^{-h}\big(g^2_\varepsilon \vartriangle_k^{h}V\big)\,\mathrm{d}x\\&+\int_{\R^3}
	V\cdot\nabla V\,\vartriangle_k^{-h}\big(g^2_\varepsilon \vartriangle_k^{h}V\big)\,\mathrm{d}x+\int_{\R^3}
	U_0\cdot\nabla V\,\vartriangle_k^{-h}\big(g^2_\varepsilon \vartriangle_k^{h}V\big)\,\mathrm{d}x\\&+\int_{\R^3}
	V\cdot\nabla U_0\,\vartriangle_k^{-h}\big(g^2_\varepsilon \vartriangle_k^{h}V\big)\,\mathrm{d}x+\int_{\R^3}
	U_0\cdot\nabla U_0\,\vartriangle_k^{-h}\big(g^2_\varepsilon \vartriangle_k^{h}V\big)\,\mathrm{d}x
\triangleq\sum_{i=1}^{10}K_i.
	 \end{align*}
   By the H\"older inequality, the Cauchy-Schwarz inequality and  Lemma \ref{lem-Comm}, we readily have
	 \begin{align*}
	K_1\
	 \leq \big\|g_\varepsilon \Lambda^\alpha\vartriangle_k^{h}V\big\|_{L^2(\R^3)}\big\|\vartriangle_k^{h}V\big\|_{L^2(\R^3)} \leq C\|V\|^2_{\dot{H}^1(\R^3)}+\frac{1}{16}\big\|g_\varepsilon \Lambda^\alpha\vartriangle_k^{h}V\big\|^2_{L^2(\R^3)}.
	  \end{align*}
	  We see that
	  \begin{align*}
	 K_2
	 =&-\int_{\R^3} g_\varepsilon\Lambda^\alpha\vartriangle_k^{h}V\cdot\, \big([g_\varepsilon,\Lambda^\alpha]  \vartriangle_k^{h}V \big)\,\mathrm{d}x -\int_{\R^3} \Lambda^\alpha\vartriangle_k^{h}V\cdot\, \big([g^2_\varepsilon,\Lambda^\alpha]   \vartriangle_k^{h}V \big)\,\mathrm{d}x\\
	 :=&K_{2_1}+K_{2_2}.
	  \end{align*}
	By the H\"older inequality, Cauchy-Schwarz inequality and \eqref{c-2}, we see that
\begin{align}
K_{2_1}
\leq C\|V\|^2_{\dot{H}^1(\R^3)}+\frac{1}{16}\big\|g^{\frac12}_\varepsilon \Lambda^\alpha\vartriangle_k^{h}V\big\|^2_{L^2(\R^3)}.\label{add-eq-0013}
\end{align}
On the other hand, by the H\"older inequality and Lemma \ref{lem-g-}, one has
	  \begin{align*}
K_{2_2}\leq&\big\|g^{\frac12}_\varepsilon\Lambda^\alpha\vartriangle_k^{h}V\big\|_{L^2(\R^3)}\big\|g_\varepsilon^{-\frac12}[g^2_\varepsilon,\Lambda^\alpha]   \vartriangle_k^{h}V  \big\|_{L^2(\R^3)}\\
	  \leq&C\max\big\{ \varepsilon^{\frac{1}{2}},\varepsilon^{\frac{1}{4}}\big\}\big\|g^\frac12_\varepsilon\Lambda^\alpha\vartriangle_k^{h}V\big\|_{L^2(\R^3)}\big\|  \vartriangle_k^{h}V \big\|_{L^2(\R^3)}\\
	  \leq&C\|V\|^2_{\dot{H}^1(\R^3)}+\frac{1}{16}\varepsilon^{\frac12}\big\|g^{\frac12}_\varepsilon \Lambda^\alpha\vartriangle_k^{h}V\big\|^2_{L^2(\R^3)}.
	  \end{align*}
This inequality together with \eqref{add-eq-0013} yields
$$K_2\le C\|V\|^2_{\dot{H}^1(\R^3)}+\frac{1}{8}\big\|g^{\frac12}_\varepsilon \Lambda^\alpha\vartriangle_k^{h}V\big\|^2_{L^2(\R^3)}.$$
For $K_3,\,K_4$ and $K_5$, it is obvious that
$$
K_3+K_4+K_5\le C\|V\|^2_{H^{1}(\R^{3})}.
$$
For $K_6$, we see by that
\begin{align*}
\int_{\R^3}\nabla P\,\vartriangle_k^{-h}\big(g^2_\varepsilon \vartriangle_k^{h}V\big)\,\mathrm{d}x=&\int_{\R^3}\vartriangle_k^{h}P\,\big(\nabla g^2_\varepsilon\cdot \vartriangle_k^{h}V\big)\,\mathrm{d}x.
\end{align*}
Hence, we obtain by  the Cauchy-Schwartz inequality
\begin{align*}
K_6
\leq&2\big\|\vartriangle_k^{h}P\big\|_{L^2(\R^3)} \|\nabla g^2_\varepsilon\|_{L^\infty(\R^3)}\big\|\vartriangle_k^{h}V\big\|_{L^2(\R^3)}\\
\leq &C\big(\|V\|_{L^\infty(\R^3)}\|\nabla V\|_{L^2(\R^3)}+\|U_0\|_{L^\infty(\R^3)}\|\nabla V\|_{L^2(\R^3)}\\
&+\|V\|_{L^\infty(\R^3)}\|\nabla U_0\|_{L^2(\R^3)}+\|U_0\|_{L^\infty(\R^3)}\|\nabla U_0\|_{L^2(\R^3)}\big)\|V\|_{\dot{H}^1(\R^3)},
\end{align*}
where we have used the fact that
$\nabla g^2_\varepsilon= -\frac{2\varepsilon x }{(1+\varepsilon|x|^2)^2}.$

Since $\text{div}\,V=0$,  we observe that
\begin{align*}
K_7
=&-\int_{\R^3}
\vartriangle_k^{h}V(x+h\mathbf{e}_k)\cdot\nabla V\big(g^2_\varepsilon \vartriangle_k^{h}V\big)\,\mathrm{d}x+\int_{\R^3}
V \cdot \nabla g_\varepsilon \,\vartriangle_k^{h} V\big(g_\varepsilon \vartriangle_k^{h}V\big)\,\mathrm{d}x:=K_{7_1}+K_{7_2}.
\end{align*}
By the H\"older inequality, we get
\begin{align}
K_{7_2}\leq&\|V\|_{L^\infty(\R^3)}\|\nabla g_\varepsilon\|_{L^\infty(\R^3)}\big\|\vartriangle_k^{h} V\big\|_{L^2(\R^3)}\big\|g_\varepsilon \vartriangle_k^{h}V\big\|_{L^2(\R^3)}\nonumber\\
\leq&C \|V\|^2_{L^\infty(\R^3)}\|V\|^2_{\dot{H}^1(\R^3)}+\frac{1-\alpha}{16\alpha}\big\|g_\varepsilon \vartriangle_k^{h}V\big\|_{L^2(\R^3)}^2.\label{add-eq-0014}
\end{align}
Since $V\in B^{2\alpha}_{2,\infty}(\R^{3})$ with $\alpha>5/6 $ and $ B^{2\alpha}_{2,\infty}(\R^{3})\hookrightarrow\dot{H}^1(\R^3)$,
there  exists a real number $R_0>0$ such that  for each $R>R_0,$
$$\|\nabla V\|_{L^3(\mathbb{B}^c_R(0))}\ll1.$$
Hence, we have by the H\"older inequality that
\begin{align*}
K_{7_1}=&-\int_{\R^3}\frac{g_\varepsilon(x)}{g_\varepsilon(x+h\mathbf{e}_k)}
\left(g_\varepsilon (x+h\mathbf{e}_k)\vartriangle_k^{h}V(x+h\mathbf{e}_k)\right)\cdot\nabla V\,\big(g_\varepsilon \vartriangle_k^{h}V\big)\,\mathrm{d}x\\
\leq&CR\|\nabla V\|_{L^3(\mathbb{B}_R(0))}\big\|\vartriangle_k^{h}V\big\|^2_{L^3(\R^3)}+C\|\nabla V\|_{L^3(\mathbb{B}^c_R(0))}\big\|g_\varepsilon \vartriangle_k^{h}V\big\|^2_{L^3(\R^3)}\\
\leq&CR\|\nabla V\|^3_{L^3(\R^3)}+C\|\nabla V\|_{L^3\big( \mathbb{B}^c_R(0)\big)}\Big(\big\|g_\varepsilon \vartriangle_k^{h}V\big\|^2_{L^2(\R^3)}+\big\|g_\varepsilon \Lambda^\alpha \vartriangle_k^{h}V\big\|^2_{L^2(\R^3)}\Big)\\
\leq&CR\|\nabla V\|^3_{L^3(\R^3)}+\frac1{16}\Big(\big\|g_\varepsilon \vartriangle_k^{h}V\big\|^2_{L^2(\R^3)}+\big\|g_\varepsilon \Lambda^\alpha \vartriangle_k^{h}V\big\|^2_{L^2(\R^3)}\Big).
\end{align*}
Collecting this inequality with \eqref{add-eq-0014}, we obtain
$$
K_7\le C \|V\|^2_{L^\infty(\R^3)}\|V\|^2_{\dot{H}^1(\R^3)}+CR\|\nabla V\|^3_{L^3(\R^3)}+\frac1{8}\Big(\big\|g_\varepsilon \vartriangle_k^{h}V\big\|^2_{L^2(\R^3)}+\big\|g_\varepsilon \Lambda^\alpha \vartriangle_k^{h}V\big\|^2_{L^2(\R^3)}\Big).
$$
We decompose $K_8$  by the incompressible condition $\text{div}\,V=0$ that
\begin{align*}
K_8
=&-\int_{\R^3}
\vartriangle_k^{h}U_0(x+h\mathbf{e}_k)\cdot\nabla V\,\big(g^2_\varepsilon \vartriangle_k^{h}V\big)\,\mathrm{d}x+\int_{\R^3}
U_0 \cdot \nabla g_\varepsilon \,\vartriangle_k^{h} V\,\big(g_\varepsilon \vartriangle_k^{h}V\big)\,\mathrm{d}x\\
:=&K_{8_1}+K_{8_2}.
\end{align*}
By the H\"older inequality, we obtain that
\begin{align*}
K_{8_2}\leq&\|U_0\|_{L^\infty(\R^3)}\big\|\nabla g_\varepsilon\big\|_{L^\infty(\R^3)}\|\vartriangle_k^{h}V\|_{L^2(\R^3)}\big\|g_\varepsilon \vartriangle_k^{h}V\big\|_{L^2(\R^3)}\\
\leq &C\|U_0\|_{L^\infty(\R^3)}\|\nabla V\|_{L^2(\R^3)}\big\|g_\varepsilon \vartriangle_k^{h}V\big\|_{L^2(\R^3)}\\
\leq &C\|U_0\|^2_{L^\infty(\R^3)}\|\nabla V\|^2_{L^2(\R^3)}+\frac{1-\alpha}{16\alpha}\big\|g_\varepsilon \vartriangle_k^{h}V\big\|^2_{L^2(\R^3)}.
\end{align*}
and
\begin{align*}
K_{8_1}
=&-\int_{\R^3}\frac{g_\varepsilon(x)}{g_\varepsilon(x+h\mathbf{e}_k)}
\big(g_\varepsilon(x+h\mathbf{e}_k)\vartriangle_k^{h}U_0(x+h\mathbf{e}_k)\big)\cdot\nabla V\,\big(g_\varepsilon \vartriangle_k^{h}V\big)\,\mathrm{d}x\\
\leq&\big\|\nabla V\big\|_{L^3(\R^3)}\big\|g_\varepsilon \vartriangle_k^{h}U_0\big\|_{L^3(\R^3)}\big\|g_\varepsilon \vartriangle_k^{h}V\big\|_{L^3(\R^3)}\\
\leq&\big\|\nabla V\big\|_{L^3(\R^3)}\big\|g_\varepsilon \vartriangle_k^{h}U_0\big\|_{L^3(\R^3)}\big\|g_\varepsilon \vartriangle_k^{h}V\big\|^{\frac{2\alpha-1}{2\alpha}}_{L^2(\R^3)}\big\|g_\varepsilon \vartriangle_k^{h}V\big\|^{\frac{1}{2\alpha}}_{\dot{H}^\alpha(\R^3)}.
\end{align*}
Collecting estimates concerning $K_{8_1}$ and $K_{8_2}$, we readily have
\begin{align*}
K_8\le &C\|U_0\|^2_{L^\infty(\R^3)}\|\nabla V\|^2_{L^2(\R^3)}+\big\|\nabla V\big\|^{2}_{L^3(\R^3)}\big\|g_{\varepsilon}\vartriangle_k^{h}U_0\big\|^{2}_{L^3(\R^3)}
\\&+\frac{1-\alpha}{8\alpha}\big\| \vartriangle_k^{h}V\big\|^2_{L^2(\R^3)}+\frac{1}{16}\big\|\Lambda^\alpha \vartriangle_k^{h}V  \big\|_{L^2(\R^3)}^2.
\end{align*}
Note that
\begin{align*}
K_9=-\int_{\R^3}
\vartriangle_k^{h}V(x+h\mathbf{e}_k)\cdot\nabla U_0\,\big(g^2_\varepsilon \vartriangle_k^{h}V\big)\,\mathrm{d}x-\int_{\R^3}
V(x+h\mathbf{e}_k)\cdot\nabla\vartriangle_k^{h}U_0\,\big(g^2_\varepsilon \vartriangle_k^{h}V\big)\,\mathrm{d}x,
\end{align*}
one easily obtain by the H\"older inequality that
\begin{align*}K_9
\leq&\big\|\vartriangle_k^{h}V\big\|^2_{L^2(\R^3)}\big\|g_\varepsilon^2\nabla U_0\big\|_{L^\infty(\R^3)}\|\vartriangle_k^{h}V\|_{L^2(\R^3)}\\
&+\|g_\varepsilon V\|_{L^2(\R^3)}\big\|\nabla\vartriangle_k^{h}U_0\big\|_{L^2(\R^3)}\big\|g_\varepsilon \vartriangle_k^{h}V\big\|_{L^2(\R^3)}\\
\leq&\big\|g_\varepsilon^2\nabla U_0\big\|_{L^\infty(\R^3)}\|\nabla V\|^{3}_{L^2(\R^3)}+\big\| V\big\|_{L^2(\R^3)}\big\|\nabla^2U_0\big\|_{L^\infty(\R^3)}\big\|g_\varepsilon \vartriangle_k^{h}V\big\|_{L^2(\R^3)}.
\end{align*}

At last,  we rewrite $K_{10}$ with
\begin{align*}
K_{10}
=&-\int_{\R^3}
\vartriangle_k^{h}U_0(x+h\mathbf{e}_k)\cdot\nabla U_0\,\big(g^2_\varepsilon \vartriangle_k^{h}V\big)\,\mathrm{d}x
-\int_{\R^3}
U_0\cdot\nabla\vartriangle_k^{h} U_0\,\big(g^2_\varepsilon \vartriangle_k^{h}V\big)\,\mathrm{d}x.
\end{align*}
Moreover, by the H\"older inequality, we get
\begin{align*}
K_{10}
\leq\|\nabla U_0\|_{L^2(\R^3)}\big\|g_\varepsilon^2\nabla U_0\big\|_{L^\infty(\R^3)}\|\nabla V\|_{L^2(\R^3)}+\big\|g^2_\varepsilon U_0\big\|_{L^\infty(\R^3)}\big\|\nabla^2U_0\big\|_{L^2(\R^3)}\|\nabla V\|_{L^2(\R^3)}.
\end{align*}
Collecting estimates for $K_1$-$K_{10}$, we finally obtain
\begin{align*}
 \big\| g_\varepsilon \Lambda^\alpha \vartriangle_k^{h}V  \big\|_{L^2(\R^3)}^2+ \big\|g_\varepsilon \vartriangle_k^{h}  V\big\|^2_{L^2(\mathbb{R}^3)}\leq C\big(U_0\big).
\end{align*}
Taking $\varepsilon\to 0$ and $h\to0$, we readily have
\begin{align*}
\big\|V\big\|_{H^{1+\alpha}(\R^3)}\leq C\big(U_0\big).
\end{align*}
This completes the proof of Proposition \ref{prop3.3}.
\end{proof}
\subsection{ $\big\|V\big\|_{H_{\omega}^{1+\alpha}(\R^3)}$-esitmate}

\vskip0.1cm
Based on estimates in Step 1-Step 3, the fourth step is devoted to
developing the high regularity  for $ V$  in the weighted Hilbert space by choosing the suitable weak solution.

\begin{proposition}\label{key}
	Let  $\alpha\in(5/6,1)$, and $V\in H_{\sigma}^{\alpha}(\R^{3})$ be the weak solution which was established in Theorem~\ref{thm-1}. Then $V\in H_{\omega}^{1+\alpha}(\R^3)$ satisfies
	\begin{equation*}
	\big\|(1+|\cdot|^{2})^{\frac{1}{4}}\,V\big\|_{H^{1+\alpha}(\R^3)}<C\left(U_0\right).
	\end{equation*}
\end{proposition}
\begin{proof}
According to Proposition \ref{prop3.3}, we easily find that $\vartriangle_k^{-h}\big(h^2_\varepsilon \vartriangle_k^{h}V\big)$ belongs to $H^1_{\omega}(\R^3)$.
By taking $\varphi=-\vartriangle_k^{-h}\big(h^2_\varepsilon \vartriangle_k^{h}V\big)$
in \eqref{eq.weak-IIII},  we obtain
\begin{align}\label{add-eq-0015}
\begin{split}
&-\int_{\mathbb{R}^3}\Lambda^\alpha V:\Lambda^\alpha\vartriangle_k^{-h}\big(h^2_\varepsilon \vartriangle_k^{h}V\big)\,\mathrm{d}x+\frac{1}{2\alpha}\int_{\mathbb{R}^3}x\cdot \nabla V\cdot\vartriangle_k^{-h}\big(h^2_\varepsilon \vartriangle_k^{h}V\big)\,\mathrm{d}x\\
&+\frac{2\alpha-1}{2\alpha}\int_{\mathbb{R}^3} V\cdot\vartriangle_k^{-h}\big(h^2_\varepsilon \vartriangle_k^{h}V\big)\,\mathrm{d}x\\
=&-\int_{\mathbb{R}^3}P\, \mathrm{div}\,\vartriangle_k^{-h}\big(h^2_\varepsilon \vartriangle_k^{h}V\big)\,\mathrm{d}x-\int_{\mathbb{R}^3}V\cdot \nabla\vartriangle_k^{-h}\big(h^2_\varepsilon \vartriangle_k^{h}V\big)\cdot V\,\mathrm{d}x\\
&+\int_{\mathbb{R}^3}U_0\cdot \nabla  V\cdot\vartriangle_k^{-h}\big(h^2_\varepsilon \vartriangle_k^{h}V\big)\,\mathrm{d}x+\int_{\mathbb{R}^3}(V+U_0)\cdot \nabla  U_0\cdot\vartriangle_k^{-h}\big(h^2_\varepsilon \vartriangle_k^{h}V\big)\,\mathrm{d}x.
\end{split}
\end{align}
We compute the term including the fractional operator as follows
\begin{align*}
&-\int_{\R^3}\Lambda^\alpha V\cdot\Lambda^\alpha\,\vartriangle_k^{-h}\Big(h^2_\varepsilon \vartriangle_k^{h}V\Big)\,\mathrm{d}x\\
=&\int_{\R^3}h_\varepsilon \Lambda^\alpha\vartriangle_k^{h}V\cdot\, \Lambda^\alpha \big(h_\varepsilon \vartriangle_k^{h}V \big)\,\mathrm{d}x-\int_{\R^3} \Lambda^\alpha\vartriangle_k^{h}V\cdot\, \big([h_\varepsilon,\Lambda^\alpha] h_\varepsilon \vartriangle_k^{h}V \big)\,\mathrm{d}x\\
=&\big\| h_\varepsilon \Lambda^\alpha \vartriangle_k^{h}V  \big\|_{L^2(\R^3)}^2-\int_{\R^3}h_\varepsilon \Lambda^\alpha\vartriangle_k^{h}V\cdot\, \big([h_\varepsilon,\Lambda^\alpha]  \vartriangle_k^{h}V  \big)\,\mathrm{d}x\\&-\int_{\R^3} \Lambda^\alpha\vartriangle_k^{h}V\cdot\, \big([h_\varepsilon,\Lambda^\alpha] h_\varepsilon \vartriangle_k^{h}V \big)\,\mathrm{d}x.
\end{align*}
The reminding terms in left side of \eqref{add-eq-0015} become
\begin{align*}
&\frac{1}{2\alpha}\int_{\mathbb{R}^3}x\cdot \nabla V\cdot\vartriangle_k^{-h}\Big(h_\varepsilon^2\vartriangle_k^{h}V\Big)\,\mathrm{d}x+\frac{2\alpha-1}{2\alpha}\int_{\mathbb{R}^3} V\cdot\vartriangle_k^{-h}\Big(\big(h_\varepsilon^2\vartriangle_k^{h}V\big)\vartriangle_k^{h}V\Big)\,\mathrm{d}x\\
=&-	\frac{1}{2\alpha}\int_{\mathbb{R}^3}\vartriangle_k^{h}\big(x\cdot \nabla V\big)\cdot (h^2_\varepsilon \vartriangle_k^{h}V)\,\mathrm{d}x-\frac{2\alpha-1}{2\alpha}\int_{\mathbb{R}^3}h^2_\varepsilon \vartriangle_k^{h}V\cdot \vartriangle_k^{h}V\,\mathrm{d}x\\
=&-	\frac{1}{2\alpha}\int_{\mathbb{R}^3}h_{\varepsilon}^2(x+h\mathbf{e}_k)\cdot \nabla\vartriangle_k^{h} V\cdot \vartriangle_k^{h}V\,\mathrm{d}x-	 \frac{1}{2\alpha}\int_{\mathbb{R}^3}h_\varepsilon^2\big(\vartriangle_k^{h}x\big)\cdot \nabla V\cdot \vartriangle_k^{h}V\,\mathrm{d}x\\
&-\frac{2\alpha-1}{2\alpha}\int_{\mathbb{R}^3}h^2_\varepsilon \vartriangle_k^{h}V\cdot \vartriangle_k^{h}V\,\mathrm{d}x\\
=&\frac{5-4\alpha}{4\alpha}\big\|h_\varepsilon \vartriangle_k^{h}  V\big\|^2_{L^2(\mathbb{R}^3)}+\frac{1}{4\alpha}\int_{\R^3}x\cdot\nabla  g_\varepsilon^2\,\vartriangle_k^{h} V\cdot \vartriangle_k^{h}V\,\mathrm{d}x
-	\frac{h}{2\alpha}\int_{\mathbb{R}^3}h_\varepsilon^2 \partial_{x_k}\vartriangle_k^{h} V\cdot \vartriangle_k^{h}V\,\mathrm{d}x\\&-	 \frac{1}{2\alpha}\int_{\mathbb{R}^3}h_\varepsilon^2 \partial_{x_k}V\cdot \vartriangle_k^{h}V\,\mathrm{d}x.
\end{align*}
This equality implies that
\begin{align*}&\frac{1}{2\alpha}\int_{\mathbb{R}^3}x\cdot \nabla V\cdot\vartriangle_k^{-h}\Big(h_\varepsilon^2\vartriangle_k^{h}V\Big)\,\mathrm{d}x+\frac{2\alpha-1}{2\alpha}\int_{\mathbb{R}^3} V\cdot\vartriangle_k^{-h}\Big(\big(h_\varepsilon^2\vartriangle_k^{h}V\big)\vartriangle_k^{h}V\Big)\,\mathrm{d}x\\
=&\frac{5-4\alpha}{4\alpha}\big\|h_\varepsilon \vartriangle_k^{h}  V\big\|^2_{L^2(\mathbb{R}^3)}+\frac{1}{4\alpha}\int_{\R^3} \frac{|x|}{1+\varepsilon|x|^2}\,\vartriangle_k^{h} V\cdot \vartriangle_k^{h}V\,\mathrm{d}x
\\&-\frac{1}{2\alpha}\int_{\R^3} \frac{\varepsilon |x|^2}{1+\varepsilon|x|^2}\,\vartriangle_k^{h} V\cdot \vartriangle_k^{h}V\,\mathrm{d}x-\frac{1}{2\alpha}\int_{\R^3} \frac{\varepsilon |x|^2|x|}{1+\varepsilon|x|^2}\,\vartriangle_k^{h} V\cdot \vartriangle_k^{h}V\,\mathrm{d}x\\
&+\frac{h}{4\alpha}\int_{\mathbb{R}^3}\partial_{x_k}h_\varepsilon^2\vartriangle_k^{h} V\cdot \vartriangle_k^{h}V\,\mathrm{d}x+\frac{1}{2\alpha}\int_{\mathbb{R}^3}h_\varepsilon^2 \partial_{x_k}V\cdot \vartriangle_k^{h}V\,\mathrm{d}x\\
\geq &\frac{1-\alpha}{\alpha}\big\|h_\varepsilon \vartriangle_k^{h}  V\big\|^2_{L^2(\mathbb{R}^3)}
-\frac{1}{2\alpha}\int_{\R^3} \frac{\varepsilon |x|^2}{1+\varepsilon|x|^2}\,\vartriangle_k^{h} V\cdot \vartriangle_k^{h}V\,\mathrm{d}x\\
&+\frac{h}{4\alpha}\int_{\mathbb{R}^3}\partial_{x_k}h_\varepsilon^2\vartriangle_k^{h} V\cdot \vartriangle_k^{h}V\,\mathrm{d}x+\frac{1}{2\alpha}\int_{\mathbb{R}^3}h_\varepsilon^2 \partial_{x_k}V\cdot \vartriangle_k^{h}V\,\mathrm{d}x.
\end{align*}
So, we have from \eqref{add-eq-0015} that
\begin{align*}
& \big\| h_\varepsilon \Lambda^\alpha \vartriangle_k^{h}V  \big\|_{L^2(\R^3)}^2+\frac{1-\alpha}{\alpha}\big\|h_\varepsilon \vartriangle_k^{h}  V\big\|^2_{L^2(\mathbb{R}^3)}\\
\leq & \int_{\R^3}h_\varepsilon \Lambda^\alpha\vartriangle_k^{h}V\cdot\, \big([h_\varepsilon,\Lambda^\alpha]  \vartriangle_k^{h}V  \big)\,\mathrm{d}x-\int_{\R^3} \Lambda^\alpha\vartriangle_k^{h}V\cdot\, \big([h_\varepsilon,\Lambda^\alpha] h_\varepsilon \vartriangle_k^{h}V \big)\,\mathrm{d}x\\&+\frac{1}{2\alpha}\int_{\R^3} \frac{\varepsilon |x|^2}{1+\varepsilon|x|^2}\,\vartriangle_k^{h} V\cdot \vartriangle_k^{h}V\,\mathrm{d}x+\frac{h}{4\alpha}\int_{\mathbb{R}^3}\partial_{x_k}h_\varepsilon^2\vartriangle_k^{h} V\cdot \vartriangle_k^{h}V\,\mathrm{d}x\\
&-\frac{1}{2\alpha}\int_{\mathbb{R}^3}h_\varepsilon^2 \partial_{x_k}V\cdot \vartriangle_k^{h}V\,\mathrm{d}x+\int_{\R^3}\nabla P\,\vartriangle_k^{-h}\big(h^2_\varepsilon \vartriangle_k^{h}V\big)\,\mathrm{d}x\\
&+\int_{\R^3}
V\cdot\nabla V\,\vartriangle_k^{-h}\big(h^2_\varepsilon \vartriangle_k^{h}V\big)\,\mathrm{d}x+\int_{\R^3}
U_0\cdot\nabla V\,\vartriangle_k^{-h}\big(h^2_\varepsilon \vartriangle_k^{h}V\big)\,\mathrm{d}x\\
&+\int_{\R^3}
V\cdot\nabla U_0\,\vartriangle_k^{-h}\big(h^2_\varepsilon \vartriangle_k^{h}V\big)\,\mathrm{d}x+\int_{\R^3}
U_0\cdot\nabla U_0\,\vartriangle_k^{-h}\big(h^2_\varepsilon \vartriangle_k^{h}V\big)\,\mathrm{d}x
\triangleq\sum_{i=1}^{10}L_i.
\end{align*}
By Lemma \ref{lem-Comm} and the H\"older inequality, we have
\begin{align*}
L_1
\leq \big\|h_\varepsilon \Lambda^\alpha\vartriangle_k^{h}V\big\|_{L^2(\R^3)}\big\|[h_\varepsilon,\Lambda^\alpha]  \vartriangle_k^{h}V\big\|_{L^2(\R^3)}
\leq &\big\|h_\varepsilon \Lambda^\alpha\vartriangle_k^{h}V\big\|_{L^2(\R^3)}\big\|\vartriangle_k^{h}V\big\|_{L^2(\R^3)}\\
\leq&C\|V\|_{\dot{H}^1(\R^3)}+\frac{1}{16}\big\|h_\varepsilon \Lambda^\alpha\vartriangle_k^{h}V\big\|^2_{L^2(\R^3)}.
\end{align*}
By the H\"older inequality, Lemma \ref{lem-Comm} and the Young inequality, we see that
\begin{align*}
L_2
\leq \big\|\Lambda^\alpha\vartriangle_k^{h}V\big\|_{L^2(\mathbb{R}^3)}\big\|[h_\varepsilon,\Lambda^\alpha] h_\varepsilon \vartriangle_k^{h}V\big\|_{L^(\R^3)}
\leq&C\|V\|_{\dot{H}^{1+\alpha}(\R^3) }\big\| h_\varepsilon \vartriangle_k^{h}V\big\|_{L^2(\R^3)}\\
\leq&C\|V\|^2_{\dot{H}^{1+\alpha}(\R^3) }+\frac{1-\alpha}{16\alpha}\big\| h_\varepsilon \vartriangle_k^{h}V\big\|^2_{L^2(\R^3)}.
\end{align*}
It is clear from $\nabla h^{2}_{\varepsilon}\in L^{\infty}(\R^{3})$ that
\begin{align*}
L_3+L_{4}\leq C\|V\|^{2}_{H^{1}(\R^{3})}
\end{align*}
Also, we have
$$
L_{5}\leq \frac{1}{2\alpha}
\big\|h_\varepsilon \vartriangle_k^{h}  V\big\|^2_{L^2(\mathbb{R}^3)}.
$$
Note that
\begin{align*}
\int_{\R^3}\nabla P\,\vartriangle_k^{-h}\big(h^2_\varepsilon \vartriangle_k^{h}V\big)\,\mathrm{d}x=&\int_{\R^3}\vartriangle_k^{h}P\,\big(\nabla h^2_\varepsilon\cdot \vartriangle_k^{h}V\big)\,\mathrm{d}x,
\end{align*}
we can deduce  by  using the equality $\nabla h^2_\varepsilon=\frac{x}{|x|(1+\varepsilon|x|^2)}-\frac{2\varepsilon x(1+|x|)}{(1+\varepsilon|x|^2)^2}$ that
\begin{align*}
L_6
\leq&2\big\|\vartriangle_k^{h}P\big\|_{L^2(\R^3)} \|\nabla h^2_\varepsilon\|_{L^\infty(\R^3)}\big\|\vartriangle_k^{h}V\big\|_{L^2(\R^3)}
\leq C  \big\|P\|_{\dot{H}^1(\R^3)}\|V\|_{\dot{H}^1(\R^3)}\\
\leq &C\big(\|V\|_{L^\infty(\R^3)}\|\nabla V\|_{L^2(\R^3)}+\|U_0\|_{L^\infty(\R^3)}\|\nabla V\|_{L^2(\R^3)}\\
&+\|V\|_{L^\infty(\R^3)}\|\nabla U_0\|_{L^2(\R^3)}+\|U_0\|_{L^\infty(\R^3)}\|\nabla U_0\|_{L^2(\R^3)}\big)\|V\|_{\dot{H}^1(\R^3)}.
\end{align*}
Since $\text{div}\,V=0,$ we observe that
\begin{align*}
L_7
=&-\int_{\R^3}
\vartriangle_k^{h}V(x+h\mathbf{e}_k)\cdot\nabla V\,\big(h^2_\varepsilon \vartriangle_k^{h}V\big)\,\mathrm{d}x
-\int_{\R^3}
V \cdot \nabla\vartriangle_k^{h} V\,\big(h^2_\varepsilon \vartriangle_k^{h}V\big)\,\mathrm{d}x\\
=&-\int_{\R^3}
\vartriangle_k^{h}V(x+h\mathbf{e}_k)\cdot\nabla V\,\big(h^2_\varepsilon \vartriangle_k^{h}V\big)\,\mathrm{d}x
-\int_{\R^3}
V \cdot \nabla h_\varepsilon \,\vartriangle_k^{h} V\,\big(h_\varepsilon \vartriangle_k^{h}V\big)\,\mathrm{d}x\\
=&L_{7_1}+L_{7_2}.
\end{align*}
By the H\"older inequality and the interpolation inequality, we get
\begin{align*}
L_{7_2}\leq&\|V\|_{L^\infty(\R^3)}\|\nabla h_\varepsilon\|_{L^\infty(\R^3)}\big\|\vartriangle_k^{h} V\big\|_{L^2(\R^3)}\big\|h_\varepsilon \vartriangle_k^{h}V\big\|_{L^2(\R^3)}\\
\leq&C \|V\|^2_{L^\infty(\R^3)}\|V\|^2_{\dot{H}^1(\R^3)}+\frac{1-\alpha}{16\alpha}\big\|h_\varepsilon \vartriangle_k^{h}V\big\|_{L^2(\R^3)}^2
\end{align*}
and
\begin{align*}
L_{7_2}=&-\int_{\R^3}\frac{h_\varepsilon(x)}{h_\varepsilon(x+h\mathbf{e}_k)}
\left(h_\varepsilon (x+h\mathbf{e}_k)\vartriangle_k^{h}V(x+h\mathbf{e}_k)\right)\cdot\nabla V\,\big(h_\varepsilon \vartriangle_k^{h}V\big)\,\mathrm{d}x\\
\leq&CR\|\nabla V\|_{L^3(\mathbb{B}_R(0))}\big\|\vartriangle_k^{h}V\big\|^2_{L^3(\R^3)}+C\|\nabla V\|_{L^3(\mathbb{B}^c_R(0))}\big\|h_\varepsilon \vartriangle_k^{h}V\big\|^2_{L^3(\R^3)}\\
\leq&CR\|\nabla V\|^3_{L^3(\R^3)}+C\|\nabla V\|_{L^3( \mathbb{B}^c_R(0))}\Big(\big\|h_\varepsilon \vartriangle_k^{h}V\big\|^2_{L^2(\R^3)}+\big\|h_\varepsilon \Lambda^\alpha \vartriangle_k^{h}V\big\|^2_{L^2(\R^3)}\Big).
\end{align*}
In the same way as leading to the estimate of $K_7$, we choose $R$ sufficient large so that
\begin{align*}
L_{7}\le& CR\| V\|^3_{\dot{W}^{1,3}(\R^3)}+\|V\|^2_{L^\infty(\R^3)}\|V\|^2_{\dot{H}^1(\R^3)}+\frac{1-\alpha}{8\alpha}\big\|h_\varepsilon \vartriangle_k^{h}V\big\|_{L^2(\R^3)}^2
+\frac{1}{16}\big\|h_\varepsilon \Lambda^\alpha\vartriangle_k^{h}V\big\|^2_{L^2(\R^3)}.
\end{align*}
From the incompressible condition $\text{div}\,V=0$, we have
\begin{align*}
L_8
=&-\int_{\R^3}
\vartriangle_k^{h}U_0(x+h\mathbf{e}_k)\cdot\nabla V\,\big(h^2_\varepsilon \vartriangle_k^{h}V\big)\,\mathrm{d}x-\int_{\R^3}
U_0 \cdot \nabla h_\varepsilon \,\vartriangle_k^{h} V\,\big(h_\varepsilon \vartriangle_k^{h}V\big)\,\mathrm{d}x\\
:=&L_{8_1}+L_{8_2}.
\end{align*}
By the H\"older inequality, we obtain that
\begin{align*}
L_{8_2}\leq&\|U_0\|_{L^\infty(\R^3)}\big\|\nabla h_\varepsilon\big\|_{L^\infty(\R^3)}\|\vartriangle_k^{h}V\|_{L^2(\R^3)}\big\|h_\varepsilon \vartriangle_k^{h}V\big\|_{L^2(\R^3)}\\
\leq &C\|U_0\|_{L^\infty(\R^3)}\|\nabla U_0\|_{L^2(\R^3)}\big\|h_\varepsilon \vartriangle_k^{h}V\big\|_{L^2(\R^3)}\\
\leq &C\|U_0\|^2_{L^\infty(\R^3)}\|\nabla U_0\|^2_{L^2(\R^3)}+\frac{1-\alpha}{16\alpha}\big\|h_\varepsilon \vartriangle_k^{h}V\big\|^2_{L^2(\R^3)}
\end{align*}
and
\begin{align*}
L_{8_1}=&-\int_{\R^3}\frac{h_\varepsilon(x)}{h_\varepsilon(x+h\mathbf{e}_k)}
\big(h_\varepsilon(x+h\mathbf{e}_k)\vartriangle_k^{h}U_0(x+h\mathbf{e}_k)\big)\cdot\nabla V\,\big(h_\varepsilon \vartriangle_k^{h}V\big)\,\mathrm{d}x\\
\leq&\big\|\nabla V\big\|_{L^3(\R^3)}\big\|h_\varepsilon \vartriangle_k^{h}U_0\big\|_{L^3(\R^3)}\big\|h_\varepsilon \vartriangle_k^{h}V\big\|_{L^3(\R^3)}\\
\leq&\big\|\nabla V\big\|_{L^3(\R^3)}\big\|h_\varepsilon \vartriangle_k^{h}U_0\big\|_{L^3(\R^3)}\big\|h_\varepsilon \vartriangle_k^{h}V\big\|^{\frac{2\alpha-1}{2\alpha}}_{L^2(\R^3)}\big\|h_\varepsilon \vartriangle_k^{h}V\big\|^{\frac{1}{2\alpha}}_{\dot{H}^\alpha(\R^3)}.
\end{align*}
These estimates help us to get
\begin{align*}
L_8\le &C\|U_0\|^2_{L^\infty(\R^3)}\|\nabla U_0\|^2_{L^2(\R^3)}+\big\|\nabla V\big\|^{2}_{L^3(\R^3)}\big\|h_\varepsilon \vartriangle_k^{h}U_0\big\|^{2}_{L^3(\R^3)}
\\&+\frac{1-\alpha}{8\alpha}\big\|h_\varepsilon \vartriangle_k^{h}V\big\|^2_{L^2(\R^3)}+\frac{1}{16}\big\| h_\varepsilon \Lambda^\alpha \vartriangle_k^{h}V  \big\|_{L^2(\R^3)}^2.
\end{align*}
Since
\begin{align*}
L_9=&-\int_{\R^3}
\vartriangle_k^{h}V(x+h\mathbf{e}_k)\cdot\nabla U_0\,\big(h^2_\varepsilon \vartriangle_k^{h}V\big)\,\mathrm{d}x-\int_{\R^3}
V(x+h\mathbf{e}_k)\cdot\nabla\vartriangle_k^{h}U_0\,\big(h^2_\varepsilon \vartriangle_k^{h}V\big)\,\mathrm{d}x,
\end{align*}
we have by the H\"older inequality that
\begin{align*}
L_9\leq&\big\|\vartriangle_k^{h}V\big\|^2_{L^2(\R^3)}\big\|h_\varepsilon^2\nabla U_0\big\|_{L^\infty(\R^3)}\|\vartriangle_k^{h}V\|_{L^2(\R^3)}\\&+\|h_\varepsilon V\|_{L^2(\R^3)}\big\|\nabla\vartriangle_k^{h}U_0\big\|_{L^2(\R^3)}\big\|h_\varepsilon \vartriangle_k^{h}V\big\|_{L^2(\R^3)}\\
\leq&\big\|V\big\|^2_{L^2_{\langle x\rangle }(\R^3)}\big\|h_\varepsilon^2\nabla U_0\big\|_{L^\infty(\R^3)}\|\nabla V\|_{L^2(\R^3)}+\big\| V\big\|_{L^2_{\langle x\rangle}(\R^3)}\big\|\nabla^2U_0\big\|_{L^2(\R^3)}\big\|h_\varepsilon \vartriangle_k^{h}V\big\|_{L^2(\R^3)}.
\end{align*}
For $L_{10}$, we see that
\begin{align*}
L_{10}=&-\int_{\R^3}
\vartriangle_k^{h}U_0(x+h\mathbf{e}_k)\cdot\nabla U_0\,\big(h^2_\varepsilon \vartriangle_k^{h}V\big)\,\mathrm{d}x
-\int_{\R^3}
U_0\cdot\nabla\vartriangle_k^{h} U_0\,\big(h^2_\varepsilon \vartriangle_k^{h}V\big)\,\mathrm{d}x.
\end{align*}
By the H\"older inequality, we obtain
\begin{align*}
L_{10}
\leq&\|\nabla U_0\|_{L^2(\R^3)}\big\|h_\varepsilon^2\nabla U_0\big\|_{L^\infty(\R^3)}\|\nabla V\|_{L^2(\R^3)}+\big\|h^2_\varepsilon U_0\big\|_{L^\infty(\R^3)}\big\|\nabla^2U_0\big\|_{L^2(\R^3)}\|\nabla V\|_{L^2(\R^3)}.
\end{align*}
Collecting all estimates for  $L_1$-$L_{10}$, we eventually obtain that
\begin{align*}
\big\| h_\varepsilon \Lambda^\alpha \vartriangle_k^{h}V  \big\|_{L^2(\R^3)}^2+\frac{1-\alpha}{\alpha}\big\|h_\varepsilon \vartriangle_k^{h}  V\big\|^2_{L^2(\mathbb{R}^3)}\leq C\left(U_0\right).
\end{align*}
Taking $h\to0$ and $\varepsilon\to0$, we immediately have
\begin{equation*}
\big\| \sqrt{1+|\cdot|} \Lambda^\alpha \nabla V  \big\|_{L^2(\R^3)}^2+\frac{1-\alpha}{\alpha}\big\| \sqrt{1+|\cdot|} \nabla  V\big\|^2_{L^2(\mathbb{R}^3)}\leq C\left(U_0\right).
\end{equation*}
Since $\sqrt{1+|x|} \sim (1+|x|^{2})^{\frac{1}{4}}, $  the above inequality allows us to conclude  that
\begin{equation}\label{E3.42}
\big\|(1+|\cdot|^{2})^{\frac{1}{4}}\Lambda^\alpha \nabla V  \big\|_{L^2(\R^3)}^2+\big\|(1+|\cdot|^{2})^{\frac{1}{4}}\nabla  V\big\|^2_{L^2(\mathbb{R}^3)}\leq C\left(U_0\right)
\end{equation}
Finally, by Lemma \ref{Weighted} we have
\begin{equation*}
	\big\|(1+|\cdot|^{2})^{\frac{1}{4}}\,V\big\|_{H^{1+\alpha}(\R^3)}<C\left(U_0\right).
	\end{equation*}
	We complete the proof of Proposition \ref{key}.
\end{proof}
\begin{remark}
Let us point out that Lemma \ref{lem-Comm} allows us to  conclude that for each $\beta\in[0,\alpha[$, $$\int_{\R^{3}}\langle x\rangle^{\frac{\beta}{2}}|V|^2(x)\,\mathrm{d}x+\int_{\R^{3}}\langle x\rangle^{\frac{\beta}{2}}|\Lambda^{1+\alpha}V|^2(x)\,\mathrm{d}x<+\infty$$
	by modifying the test function in \eqref{eq.weak} and \eqref{eq.weak-IIII}. For the sake of simplicity, we choose $\beta=\frac12$ in Proposition \ref{prop3.1} and  Proposition \ref{key}.
\end{remark}

\setcounter{equation}{0}
 \setcounter{equation}{0}
\section{Decay estimates for $V$ near infinity}

In this section, we are going to show the decay estimate  for solution $V$ under the hypotheses of Theorem \ref{T1.3}, which satisfies
\begin{align*}
\left\{
\begin{aligned}
&(-\Delta)^{\alpha}V- \frac{2\alpha-1}{2\alpha}V(x)-\frac{1}{2\alpha}x\cdot \nabla V+\nabla P=f  \\
&\textnormal{div}\,V=0,
\end{aligned}\ \right.,
\end{align*}
where
\begin{equation*} f=-U_{0}\cdot\nabla U_{0}-(U_{0}+V)\cdot\nabla V-V\cdot\nabla U_{0}.
 \end{equation*}
Based on the  regularity in the weighted Sobolev spaces  developed in Section \ref{sec-3}, we know that
	$$\big\|(1+|\cdot|^{2})^{\frac{1}{4}}\,V\big\|_{H^{1+\alpha}(\R^3)}<C\left(U_0\right).$$ Moreover, by the H\"older inequality and the embedding relation that $H^{1+\alpha}(\R^3)\subset C_{\textnormal b}(\R^3)$ for $\alpha>5/6$, we see that $f \in L^{2}(\R^{3})\cap L^{\infty}(\R^{3})$ and
 $$
 V(x)\leq C(1+|x|)^{-\frac{1}{2}}.
 $$
 But, this decay rate  is weaker than the one for the force term $U_0\cdot\nabla U_0$, whose order is $2-4\alpha$. To fill this gap,
 we need to explore  the structure of self-similar solution $V(x)=u(x,1)-e^{\Delta}u_0$.  This  structure induces us to carry it out by  studying
 the following   nonlocal Stokes system with linear singularity force
  \begin{align}\label{Elinear}
 \left\{
 \begin{aligned}
 &\partial_{t}w+(-\Delta)^{\alpha} w+\nabla p=t^{\frac{1}{\alpha}-2}\mbox{div}_{x}f(x/t^{\frac{1}{2\alpha}}) \qquad\; \mbox{in}\; \; \R^{3}\times (0,+\infty),\\
 &\mbox{div}\, w=0\ \quad\qquad\qquad\qquad\qquad\qquad\qquad\qquad \,\,\,\, \ \mbox{in}\ \ \R^{3}.
 \end{aligned} \right.
 \end{align}

\begin{proposition}\label{L4.2}
Let $\alpha>\frac{1}{2}, f\in L^{q}(\R^{3})$ with $1<q<\frac{3}{2\alpha-1}$.
Then problem \eqref{Elinear} admits a solution $w\in L_{t}^{\infty}\big(0,T;\,L_{x}^{q}(\R^{3})\big)$ for any $T<\infty$ which has the form
$$
w(x,t)=\int_{0}^{t}e^{-(-\Delta)^{\alpha}(t-s)}\mathbb{P}{\rm div}_{x}s^{\frac{1}{\alpha}-2}f(\cdot/s^{\frac{1}{2\alpha}})\,{\rm d}s.
$$
If $v(x,t)\in L_{t}^{\infty}\big(0,T;L_{x}^{q_{1}}(\R^{3})\big)$ with any $q_{1}\geq 1$ is another solution of \eqref{Elinear} such that
$$\lim_{t\to0}\|v(\cdot,t)\|_{L^{q_{1}}(\R^{3})}=0,$$
 then $v\equiv w$ in $\R^{3}\times(0,+\infty)$.
\end{proposition}
\begin{proof}
First, we claim $w(x,t)\in L_{t}^{\infty}L_{x}^{q}(\R^{3}\times(0,T))$ for any $T<\infty$ and
$$\lim_{t\to0}\|w(x,t)\|_{L^{q}_{x}(\R^{3})}=0.$$
Indeed,
\begin{align*}
\|w(\cdot,t)\|_{L^{q}_{x}(\R^{3})}&\leq C\int_{0}^{t}(t-s)^{-\frac{2}{\alpha}}\|\Upsilon(\cdot/(t-s)^{\frac{1}{2\alpha}})\|_{L^{1}(\R^{3})}s^{\frac{1}{\alpha}-2}\|f(\cdot/s^{\frac{1}{2\alpha}})\|_{L^{q}(\R^{3})}\,{\rm d}s\\&
\leq C\int_{0}^{t}(t-s)^{-\frac{1}{2\alpha}}s^{\frac{1}{\alpha}-2+\frac{3}{2\alpha q}}\,{\rm d}s\\&
\leq C t^{\frac{1}{\alpha}+\frac{3}{2\alpha q}-1}\to0 \ \ \ \mbox{as}\ \  t\to0.
\end{align*}
This inequality shows that the claim is true.

Next, it is necessary to prove $w(x,t)$ is a solution of \eqref{Elinear} in the sense of distribution, {\rm i.e.},
\begin{align}\label{Elinear-1}
\langle\partial_{t}w,\theta\varphi\rangle=\langle-(-\Delta)^{\alpha} w,\theta\varphi\rangle+\langle\mathbb{P}{\rm div}_{x}s^{\frac{1}{\alpha}-2}f(\cdot/s^{\frac{1}{2\alpha}}),\theta\varphi\rangle,
\end{align}
where $\theta(t)\in C_{0}^{\infty}\big([0,T)\big)$ and $\varphi(x)\in C_{0}^{\infty}(\R^{3})$.

 We adopt the argument in \cite{Du}. Let
$w(\cdot,t)=\mathbb{P}{\rm div}_{x}t^{\frac{1}{\alpha}-2}f(\cdot/t^{\frac{1}{2\alpha}})$,  one easily verifies  that
 \begin{align*}
 |\langle w(\cdot,t),\varphi\rangle|=\left|\int_{\R^{3}}t^{\frac{1}{\alpha}-2}\mathbb{P}f(x/t^{\frac{1}{2\alpha}})\nabla \varphi\,{\rm d}x\right|
 \leq t^{\frac{1}{\alpha}-2+\frac{3}{2\alpha q}} \|f\|_{L^{q}(\R^{3})}\|\nabla\varphi\|_{L^{\frac{q}{q-1}}(\R^{3})}.
\end{align*}
This inequality implies that $w\in L_{{\rm loc}}^{\infty}(0,T; W^{-1,q}(\R^{3}))$. Besides, a simple calculation shows that
\begin{align*}
\Big\langle\partial_{t}\int_{0}^{t}e^{-(-\Delta)^{\alpha}(t-s)}w(s)\,{\rm d}s,\theta\varphi\Big\rangle
=&-\int_{0}^{T}\int_{0}^{t}\langle e^{-(-\Delta)^{\alpha}(t-s)}w(s),\partial_{t}\theta\varphi\rangle \,{\rm d}s{\rm d}t\\
=&-\int_{0}^{T}\int_{0}^{t}\langle w(s), \partial_{t}\theta e^{-(-\Delta)^{\alpha}(t-s)}\varphi\rangle \,{\rm d}s{\rm d}t\\
=&-\int_{0}^{T}\int_{0}^{t}\big\langle w(s), \partial_{t}(\theta e^{-(-\Delta)^{\alpha}(t-s)}\varphi)\\
&\qquad\qquad\quad -\theta(-\Delta)^{\alpha}e^{-(-\Delta)^{\alpha}(t-s)}\varphi\big\rangle \,{\rm d}s{\rm d}t\\
\triangleq&\;{\rm I+II.}
\end{align*}
We directly calculus to show
\begin{align*}
{\rm I}&=-\int_{0}^{T}\int_{0}^{t}\big\langle w(s), \partial_{t}(\theta e^{-(-\Delta)^{\alpha}(t-s)}\varphi)\big\rangle \,{\rm d}s{\rm d}t\\&
=-\lim_{\epsilon\to0}\int_{0}^{T}\int_{s+\epsilon}^{T}\big\langle w(s), \partial_{t}(\theta e^{-(-\Delta)^{\alpha}(t-s)}\varphi)\big\rangle \,{\rm d}t{\rm d}s\\&
=\lim_{\epsilon\to0}\int_{0}^{T}\langle w(s),\theta(s+\epsilon)e^{-(-\Delta)^{\alpha}\epsilon}\varphi\big\rangle\,{\rm d}s\\&
=\int_{0}^{T}\langle w(s),\theta(s)\varphi\rangle\,{\rm d}s=\langle w, \theta\varphi\rangle,
\end{align*}
and
\begin{align*}
{\rm II}&=-\int_{0}^{T}\int_{0}^{t}\big\langle w(s),\theta(-\Delta)^{\alpha}e^{-(-\Delta)^{\alpha}(t-s)}\varphi\big\rangle \,{\rm d}s{\rm d}t\\&
=-\int_{0}^{T}\big\langle\int_{0}^{t}e^{-(-\Delta)^{\alpha}(t-s)}w(s)\,{\rm d}s, (-\Delta)^{\alpha}\varphi\big\rangle\theta\,{\rm d}t\\&
=-\big\langle(-\Delta)^{\alpha}\int_{0}^{t}e^{-(-\Delta)^{\alpha}(t-s)}w(s)\,{\rm d}s, \varphi\theta\big\rangle.
\end{align*}
From the above discussion, we obtain \eqref{Elinear-1}.

Finally, we prove the uniqueness.   Let $\theta =w-v$, then
\begin{align}\label{Elinear-2}
\theta \in L_{t}^{\infty}(L_{x}^{q}+L_{x}^{q_{1}})(\R^{3}\times(0,T))
\end{align}
for any $T<\infty$, fulfills
 \begin{align*}
\left\{
\begin{aligned}
&\partial_{t}\theta+(-\Delta)^{\alpha} \theta+\nabla p=0,\;\; \ \ \mbox{in}\ \ \R^{3}\times (0,T), \\
 &\mbox{div}\, \theta=0\\
\end{aligned} \right.
\end{align*}
for some distribution $p$, and $\lim_{t\to0}\|\theta(\cdot,t)\|_{(L_{x}^{q_{1}}+L_{x}^{q})(\R^{3})}=0$. Let $\eta_{\epsilon}(x)$ be the standard mollifier, then
\begin{align*}
\left\{
\begin{aligned}
&\partial_{t}\eta_{\epsilon}\ast \theta+(-\Delta)^{\alpha}\eta_{\epsilon}\ast \theta+\nabla \eta_{\epsilon}\ast p=0, \;\; \mbox{in}\ \ \R^{3}\times (0,T), \\
 &\mbox{div}\, \eta_{\epsilon}\ast \theta=0\\
\end{aligned} \right.
\end{align*}
with any $T<\infty$. Let  $\theta_{\epsilon}(x,t)=\eta_{\epsilon}\ast w(x,t)$, then $$\theta_{\epsilon}(x,t)\in C^{\infty}(\R^{3})\;\;\text{and}\;\;\lim_{|x|\to\infty} \theta_{\epsilon}(x,t)=0$$
for all $t\geq0$. Denote by
$\tilde{\theta}_{\epsilon}=\nabla\times \theta_{\epsilon}$, then $\tilde{\theta}_{\epsilon}$ is solution of
\begin{align*}
\left\{\begin{aligned}
&\partial_{t}\tilde{\theta}_{\epsilon}+(-\Delta)^{\alpha}\tilde{\theta}_{\epsilon}=0, \;\ \ \mbox{in}\ \ \R^{3}\times (0,T),\\
 &\tilde{\theta}_{\epsilon}(x,0)=0\\
\end{aligned} \right.
\end{align*}
with any $T<\infty$. From the fact $\lim_{|x|\to\infty}\tilde{\theta}_{\epsilon}=0$ and maximum principle of the fractional heat operator, we conclude $\tilde{\theta}_{\epsilon}\equiv0$ for $(x,t)\in\R^{3}\times(0,+\infty)$. From the  identity
$$
\nabla\times(\nabla\times \theta)=\nabla({\rm div} \theta)-\Delta \theta, \ \ \mbox{for all}\ \theta\in C^{\infty}(\R^{3}),
$$
we immediately derive $\Delta \theta_{\epsilon}(x,t)=0$ in $\R^{3}\times(0,+\infty)$. The Liouville theorem of harmonic functions, together  with $\|\theta_{\epsilon}(\cdot,t)\|_{L^{\infty}(\R^{3})}< +\infty$ for
any $t\in (0,+\infty)$, implies $\theta_{\epsilon}=0$ in $\R^{3}\times(0,+\infty)$. Since for a.e. $t\in (0,+\infty)$
$$
\|\theta_{\epsilon}(\cdot,t)-\theta\|_{L^{1}_{{\rm loc}}(\R^{3})}\to0\ \ \ \mbox{as}\ \ \epsilon\to0
$$
we deduce that $\theta(x,t)\equiv0$ for $(x,t)\in \R^{3}\times(0,+\infty)$.
\end{proof}
Borrowing this proposition and $V\in H^{1+\alpha}_{\langle x\rangle}(\R^3)$, we can write $V(x)$ in the following form
\begin{equation}\label{eq.fundamental}
V(x)=v(x,1)=\int_{0}^{1}e^{-(-\Delta)^{\alpha}(1-s)}\mathbb{P}{\rm div}_{x}s^{\frac{1}{\alpha}-2}f(\cdot/s^{\frac{1}{2\alpha}})\,{\rm d}s,
\end{equation}
where \[f(x)=-\Big(V\otimes V+U_{0}\otimes V+V\otimes U_{0}+U_{0}\otimes U_{0}\Big)(x)\in L^{2}(\R^{3})\cap L^{\infty}(\R^{3}),\]
satisfying
$$
|f(x)|\leq C(1+|x|)^{-1}\ \  \mbox{for}\ \quad  x\in\R^{3}.
$$
With expressions of solution \eqref{eq.fundamental} in hand, we will improve the order of decay estimate  step by step in terms of the decay properties of Oseen kernel.

 \begin{proposition}\label{L4.3}
Let $(V,P)$ be the solution of problem \eqref{E} satisfying
$V\in H^\alpha(\mathbb{R}^3)$
with $5/6<\alpha\leq1$. Then we have for $x\in \R^{3}$
$$
|V(x)|\leq C(1+|x|)^{2-4\alpha}.
$$
\end{proposition}
\begin{proof}  Step 1. We claim $|V(x)\leq C(1+|x|)^{-1}.$

By a direct calculation,
 we get  from Proposition \ref{L4.2} that
\begin{align*}
|V(x)|=&\left|\int_{0}^{1}e^{-(-\Delta)^{\alpha}(1-s)}\mathbb{P}{\rm div}_{x}s^{\frac{1}{\alpha}-2}f(\cdot/s^{\frac{1}{2\alpha}})\,{\rm d}s\right|\\
=&\left|\int_{0}^{1}\int_{\R^{3}}\Upsilon(1-s,x-y)s^{\frac{1}{\alpha}-2}f(y/s^{\frac{1}{2\alpha}})\,{\rm d}y\,{\rm d}s\right|\\
\leq& \int_{0}^{1}\int_{\R^{3}}\Big\{(1-s)^{\frac{1}{2\alpha}}+|x-y|\Big\}^{-4}(s^{\frac{1}{2\alpha}}+|y|)^{-1}s^{\frac{3}{2\alpha}-2}\,{\rm d}y{\rm d}s\\
=&\int_{0}^{1}\int_{|y|\geq 2|x|}\cdots\,{\rm d}y{\rm d}s+\int_{0}^{1}\int_{|x|/2<|y|<2|x|}\cdots\,{\rm d}y{\rm d}s+\int_{0}^{1}\int_{|y|\leq|x|/2}\cdots\,{\rm d}y{\rm d}s\\
=&{\rm I+II+III}.
\end{align*}
In the following, we always suppose $|x|\gg 1$.
For I, we have
$$
{\rm I}\leq C |x|^{-1}\int_{0}^{1}\int_{|x|}^{\infty}r^{-2}s^{\frac{3}{2\alpha}-2}\,{\rm d} r{\rm d}s\leq C |x|^{-2}.
$$
Here we have used the fact $\alpha>5/6$. For II,
\begin{align*}
{\rm II}&\leq C |x|^{-1}\int_{0}^{1}\int_{|y|\leq 3|x|}((1-s)^{\frac{1}{2\alpha}}+|y|)^{-4}s^{\frac{3}{2\alpha}-2}\,{\rm d}y{\rm d}s\\&
= C |x|^{-1}\int_{0}^{1/2}\int_{|y|\leq 3|x|}\cdots\,{\rm d}y{\rm d}s+ C |x|^{-1}\int_{1/2}^{1}\int_{|y|\leq 3|x|}\cdots\,{\rm d}y{\rm d}s\\&
={\rm II_{1}+II_{2}}
\end{align*}
Obviously,
$$
 {\rm II_{1}}\leq C|x|^{-1}\int_{0}^{1/2}\int_{|y|\leq 3|x|}(1+|y|)^{-4}s^{\frac{3}{2\alpha}-2}\,{\rm d}y{\rm d}s\leq  C|x|^{-1},
$$
and
\begin{align*}
{\rm II_{2}}&\leq C|x|^{-1}\int_{0}^{1/2}\int_{|y|\leq 3|x|}(s^{\frac{1}{2\alpha}}+|y|)^{-4}\,{\rm d}y{\rm d}s\\&
\leq C|x|^{-1}\int_{0}^{1/2}\int_{0}^{3|x|}(s^{\frac{1}{2\alpha}}+r)^{-2}\,{\rm d}r{\rm d}s\leq C|x|^{-1}.\qquad\quad
\end{align*}
Now we consider III,
\begin{align*}
{\rm III}\leq C|x|^{-4}\int_{0}^{1}\int_{0}^{|x|}rs^{\frac{3}{2\alpha}-2}\,{\rm d}r{\rm d}s\leq C|x|^{-2}.
\end{align*}
Collecting the estimates of I-III, we obtain $V(x)\in L^{\infty}(\R^{3})$.
This helps  us to get $V(x)\leq C(1+|x|)^{-1}$.

Step 2. Due to the fact $2\alpha-1\leq1$ and $|U_{0}(x)|\leq C(1+|x|)^{1-2\alpha}$, we derive
$$
|f(x)|=\Big|\Big(V\otimes V+U_{0}\otimes V+V\otimes U_{0}+U_{0}\otimes U_{0}\Big)(x)\Big|\leq C (1+|x|)^{2-4\alpha}
$$
and furthermore,
$$
|s^{\frac{1}{\alpha}-2}f(y/s^{\frac{1}{2\alpha}})|\leq C(s^{\frac{1}{2\alpha}}+|x|)^{2-4\alpha}.
$$
This inequality helps us to obtain
\begin{align*}
|V(x)|&\leq C\int_{0}^{1}\int_{\R^{3}}\Big\{(1-s)^{\frac{1}{2\alpha}}+|x-y|\Big\}^{-4}(s^{\frac{1}{2\alpha}}+|x|)^{2-4\alpha}\,{\rm d}y\,{\rm d}s\\&
=\int_{0}^{1}\int_{|y|\geq 2|x|}\cdots\,{\rm d}y{\rm d}s+\int_{0}^{1}\int_{|x|/2<|y|<2|x|}\cdots\,{\rm d}y{\rm d}s+\int_{0}^{1}\int_{|y|\leq|x|/2}\cdots\,{\rm d}y{\rm d}s\\&
={\rm J_{1}+J_{2}+J_{3}}.
\end{align*}
As the similar way as in Step 1, we obtain
$$|{\rm J_{1}}|+|{J_{3}}|\leq C(1+|x|)^{1-4\alpha},\;\;\text{and}\;\; \;|{\rm J_{2}}|\leq C(1+|x|)^{2-4\alpha}.$$
Therefore, we finally obtain
$$
|V(x)|\leq C(1+|x|)^{2-4\alpha},
$$
and then we finish the proof of Proposition \ref{L4.3}.
\end{proof}

In next two Propositions, we will prove that $V$ achieves its optimal decay for $\alpha=1$. To do this, we have to prove $D^{2}V$ possesses some decay condition at infinity which plays a key role in our proof.

\begin{proposition}\label{prop-4.2}
Let $\alpha=1$, then for small $0<\delta\ll1$
$$
|\nabla V(x)|\leq C(1+|x|)^{-3}\log (2+|x|),\ \ |\nabla^{2} V(x)|\leq C(1+|x|)^{-3+\delta}.
$$
\end{proposition}
\begin{proof}
From \cite{LXZ}, we know
\begin{equation}\label{E4.6}
|x|\nabla V\in H^{2}(\R^{3}),\ \ |x|\nabla P\in H^{1}(\R^{3}),\ \ \mbox{and}\  |\nabla V|\leq C(1+|x|)^{-2}.
\end{equation}
 This, together $|V(x)|\leq C(1+|x|)^{-2}$, implies
$$
|\hat{f}(x)|\leq C(1+|x|)^{-3},
$$
where
\begin{equation}\label{E4.7}
\hat{f}(x)=-(V\otimes\nabla V+\nabla U_{0}\otimes V+\nabla V\otimes U_{0}+\nabla U_{0}\otimes U_{0}).
\end{equation}

On the other hand,  from Proposition  \ref{L4.2} one derives for $|x|>1$
\begin{align*}
|\nabla V|&=\left|\int_{0}^{1}\int_{\R^{3}}\nabla\mathcal{O}(x-y,1-s)s^{-\frac{3}{2}}\hat{f}\left(\frac{y}{\sqrt{s}}\right)\,{\rm d}y\,{\rm d}s\right|\\&
\leq C\left|\int_{0}^{1}\int_{\R^{3}}(|x-y|+\sqrt{1-s})^{-4}(|y|+\sqrt{s})^{-3}\,{\rm d}y\,{\rm d}s\right|.
\end{align*}
Using the same argument as in Proposition \ref{L4.3}, we have $|\nabla V|\leq C(1+|x|)^{-3}\log(2+|x|)$.

Secondly, we prove
 $$
 |\nabla^{2} V(x)|\leq C(1+|x|)^{-3+\delta}.
 $$
For this purpose, we  decompose our proof into three steps.

Step 1. $|\nabla^{2} V(x)|\leq C(1+|x|)^{-1}$. Now let
$$
E_{k}=|x|\partial_{kk}V,\quad\; P_{k}=|x|\partial_{kk}P.
$$
From \eqref{E4.6}, we obtain that
$$
E_{k}\in H^{1}(\R^{3}), \quad\;P_{k}\in L^{2}(\R^{3}).
$$
A simple calculation gives
\begin{align*}
-|x|\Delta\partial_{kk}V&=\frac{1}{2}x\cdot\nabla E_{k}+\nabla(|x|\partial_{k}V_{k})-\frac{x}{|x|}\cdot\partial_{k}V-\nabla P_{k}+\frac{x}{|x|}\partial_{kk}P\\&
-E_{k}\nabla V-V\nabla E_{k}+V\frac{x}{|x|}\cdot \partial_{kk}V.
\end{align*}
 Besides, by the Hardy inequality, $\frac{\partial_{kk}V}{|x|}\in L^{2}(\R^{3})$. Thus, $E_{k}$ fulfills the following equation in weak sense
\begin{align*}
-\Delta E_{k}-\frac{x}{2}\cdot\nabla E_{k}+\nabla P_{k}=&\nabla(|x|\partial_{k}V)-\frac{x}{|x|}\cdot\partial_{k}V+\frac{x}{|x|}\partial_{kk}P
-E_{k}\nabla V\\&-V\nabla E_{k}+V\frac{x}{|x|}\cdot \partial_{kk}V-2\frac{x}{|x|}\cdot\nabla\partial_{k}V-2\frac{\partial_{kk}V}{|x|}\\
\triangleq&\nabla(|x|\partial_{k}V)+\tilde{f}(x),
\end{align*}
i.e., for any $\varphi\in H^{1}(\R^{3})$ with $\||\cdot|\varphi\|_{L^{2}(\R^{3})}<\infty$
\begin{align}\label{E-W}
\int_{\R^{3}}\nabla E_{k}\nabla\varphi\,{\rm d}x-\frac{1}{2}\int_{\R^{3}}x\cdot\nabla E_{k}\varphi\,{\rm d}x-
\int_{\R^{3}}P_{k}{\rm div} \varphi \,{\rm d}x
=\int_{\R^{3}}(\nabla(|x|\partial_{k}V)\varphi+\tilde{f})\varphi\,{\rm d}x.
\end{align}
Here, we have used the fact
$$
\nabla(|x|\partial_{k}V),\ \  \tilde{f},\ \  |\cdot|\tilde{f}\in L^{2}(\R^{3}).
$$
Now, we need to prove $x\cdot\nabla E_{k}\in L^{2}(\R^{3})$. To this end, choosing
$$
\varphi(x)=\frac{|x|^{2}}{(1+\epsilon|x|^{2})^{\frac{3}{2}}}E_{k}(x)\triangleq h^{2}_{\epsilon}(x)E_{k}(x)
$$
in \eqref{E-W}, and letting $W_{\epsilon}(x)=h_{\epsilon}(x)E_{k}$, we have
\begin{align*}
&\int_{\R^{3}}\nabla E_{k}\nabla(h_{\epsilon}W_{\epsilon})\,{\rm d}x-\frac{1}{2}\int_{\R^{3}}x\cdot(h_{\epsilon}W_{\epsilon})\cdot\nabla E_{k}\,{\rm d}x\\
=&\int_{\R^{3}}P_{k}{\rm div}(h_{\epsilon}W_{\epsilon})\,{\rm d}x+\int_{\R^{3}}|x|\partial_{k}V_{k}\cdot{\rm div}(h_{\epsilon}W_{\epsilon})\,{\rm d}x
+\int_{\R^{3}}\hat{f}h_{\epsilon}W_{\epsilon}\,{\rm d}x.
\end{align*}
A simple calculation gives
\begin{align*}
\|\nabla W_{\epsilon}\|^{2}_{L^{2}(\R^{3})}+\frac{5}{4}\| W_{\epsilon}\|^{2}_{L^{2}(\R^{3})}
=&\int_{\R^{3}}P_{k}{\rm div}(h_{\epsilon}W_{\epsilon})\,{\rm d}x+\int_{\R^{3}}\hat{f}h_{\epsilon}W_{\epsilon}\,{\rm d}x
\\&-\int_{\R^{3}}\nabla E_{k}(\nabla h_{\epsilon}\otimes W_{\epsilon})\,{\rm d}x
+\int_{\R^{3}}\nabla W_{\epsilon}(\nabla h_{\epsilon}\otimes E_{k})\,{\rm d}x\\&+
\int_{\R^{3}}|x|\partial_{k}V_{k}{\rm div} (\nabla h_{\epsilon}W_{\epsilon})\,{\rm d}x.
\end{align*}
By a routine  calculation, we have
\begin{align*}
&\Big|\int_{\R^{3}}P_{k}{\rm div}(h_{\epsilon}W_{\epsilon})\,{\rm d}x\Big|\leq C \|P_{k}\|^{2}_{L^{2}(\R^{3})}+\frac{1}{4}\| W_{\epsilon}\|^{2}_{L^{2}(\R^{3})},\\
&\Big|\int_{\R^{3}}\nabla E_{k}(\nabla h_{\epsilon}\otimes W_{\epsilon})\,{\rm d}x\Big|\leq C \|\nabla E_{k}\|^{2}_{L^{2}(\R^{3})}+\frac{1}{4}\| W_{\epsilon}\|^{2}_{L^{2}(\R^{3})},\\
&\Big|\int_{\R^{3}}\nabla W_{\epsilon}(\nabla h_{\epsilon}\otimes V)\,{\rm d}x\Big|\leq C \| E_{k}\|^{2}_{L^{2}(\R^{3})}+\frac{1}{4}\|\nabla W_{\epsilon}\|^{2}_{L^{2}(\R^{3})},
\end{align*}
and
$$
\Big|\int_{\R^{3}}\tilde{f}h_{\epsilon}W_{\epsilon}\,{\rm d}x\Big|\leq C \|\tilde{f}h_{\epsilon}\|^{2}_{L^{2}(\R^{3})}+\frac{1}{4}\| W_{\epsilon}\|^{2}_{L^{2}(\R^{3})}.
$$
Thus, we have
\begin{align*}
\|\nabla W_{\epsilon}\|^{2}_{L^{2}(\R^{3})}+\| W_{\epsilon}\|^{2}_{L^{2}(\R^{3})}\leq C\left(\|P_{k}\|^{2}_{L^{2}(\R^{3})}+\| E_{k}\|^{2}_{H^{1}(\R^{3})}
+\||\cdot|\tilde{f}\|^{2}_{L^{2}(\R^{3})}\right).
\end{align*}
From this inequality  we immediately derives as $\epsilon\to0$
$$
 W_{\epsilon}\to \overline{W}\ \ \mbox{in}\ \ L^{1}_{{\rm loc}}(\R^{3}),\ \ \mbox{and}\ \ W_{\epsilon}\rightharpoonup \overline{W}\ \ \mbox{in}\ \ H^{1}(\R^{3}).
 $$
Besides,
 $$
 W_{\epsilon}\rightarrow W\triangleq |x|E_{k}\ \  \mbox{a.e}\  x\ \ \mbox{as}\ \  \epsilon\to0,
$$
thus, $W=\overline{W}$ by the uniqueness.
This further implies that  $x\cdot\nabla E_{k}\in L^{2}(\R^{3})$ by weak lower continuous. Thus,
$$
-\Delta E_{k}+\nabla P_{k}=x\cdot\nabla E_{k}+\nabla(|x|\partial_{k}V)+\tilde{f}(x)\in L^{2}(\R^{3}).
$$
This, together the classical elliptic regularity estimates, shows $E_{k}\in H^{2}(\R^{3})\hookrightarrow L^{\infty}(\R^{3})$,
and in turn implies $|\nabla^{2}V|\leq C(1+|x|)^{-1}$.

Step 2.  Our goal is prove the following Claim:
$$|\nabla^{2}V| \leq C(1+|x|)^{-2}.$$
Now we rewrite the equation \eqref{E} as
\begin{align*}
-\Delta V-\frac{1}{2}(x\cdot \nabla V+V)+\nabla P=\hat{f}(x)
\end{align*}
with   $\hat{f}(x)$ defined  in \eqref{E4.7}.
Then $V_{k}\triangleq \partial_{k}V \ (k=1,2,3)$ fulfills
\begin{align*}
-\Delta V_{k}-\frac{1}{2}(x\cdot \nabla V_{k}+V_{k})+\nabla P=\partial_{k}\hat{f}(x)+V_{k}(x)\triangleq g(x).
\end{align*}
From step 1, one see $|\nabla^{2}V|\leq C(1+|x|)^{-1}$, which, together $|U_{0}|\leq C(1+|x|)^{-1}$, implies
$$
|g(x)|=|\partial_{k}\hat{f}(x)+V_{k}(x)|\leq C|\nabla V_{k}||U_{0}|\leq C(1+|x|)^{-2}.
$$
Again using Proposition \ref{L4.2}
\begin{align*}
|\nabla V_{k}|&=\Big|\int_{0}^{1}\int_{\R^{3}}\nabla\mathcal{O}(x-y,1-s)s^{-\frac{3}{2}}g\Big(\frac{y}{\sqrt{s}}\Big)\,{\rm d}y{\rm d}s\Big|\\&
\leq C\Big|\int_{0}^{1}\int_{\R^{3}}(|x-y|+\sqrt{1-s})^{-4}s^{-\frac{1}{2}}(|y|+\sqrt{s})^{-2}\,{\rm d}y{\rm d}s\Big|
\end{align*}
Using the argument as in Proposition \ref{L4.3}, we obtain
$$
|\nabla V_{k}|\leq C(1+|x|)^{-2},\ \  \mbox{i.e.} \ \ |\nabla^{2}V| \leq C(1+|x|)^{-2}.
$$

Step 3. We  finally prove the decay estimate
 $$|\nabla^{2}V| \leq C(1+|x|)^{-3+\delta}, \qquad\; \text{for}\;\;0<\delta\ll 1.$$
By the step 2, we have $|\nabla^{2}V| \leq C(1+|x|)^{-2}$. This, together $|\nabla^{2}U_{0}|\leq C(1+|x|)^{-3}$, implies
 $$
 |g(x)|\leq C(1+|x|)^{-3+\delta}, \qquad\; \text{for small} \;\delta>0,
 $$
  and then for $|x|>1$, one, by the same calculus as in Proposition \ref{L4.3}, instantly  deduces
\begin{align*}
|\nabla V_{k}|&=\left|\int_{0}^{1}\int_{\R^{3}}\nabla\mathcal{O}(x-y,1-s)s^{-\frac{3}{2}}g\left(\frac{y}{\sqrt{s}}\right)\,{\rm d}y{\rm d}s\right|\\&
\leq C\left|\int_{0}^{1}\int_{\R^{3}}(|x-y|+\sqrt{1-s})^{-3+\delta}(|y|+\sqrt{s})^{-4}\,{\rm d}y{\rm d}s\right|\\&
\leq C|x|^{-3+\delta},\quad\; \text{for}\;\; \;0<\delta\ll 1.
\end{align*}
This completes the proof of step 3.
\end{proof}
 Once the decay estimate of $D^{2}V$ is proved, we use  the special structure of the Oseen tensor to remove the  logarithmic  loss and obtain the optimal decay estimate of $V$. The idea of using the special structure of the Oseen tensor is already in the work of Brandolese et.al. \cite{Br, Br1}. However, our arguments heavily rely on the decay estimates of  second derivatives of solution. This is a major difference between our approach and the one used in \cite{Br,Br1}.

\begin{proposition}\label{prop-4.3}
Let $\alpha=1$, then
$$
|V(x)|\leq C(1+|x|)^{-3}.
$$
\end{proposition}

\begin{proof}
Due to the fact  $|U_{0}(x)|\leq C(1+|x|)^{-1}$, we derive
$$
|f(x)|=\Big|\Big(V\otimes V+U_{0}\otimes V+V\otimes U_{0}+U_{0}\otimes U_{0}\Big)(x)\Big|\leq C (1+|x|)^{-2}.
$$
This inequality implies that
$$
\bar{f}(x,s)\triangleq |s^{-1}f(x/s^{\frac{1}{2}})|\leq C(s^{\frac{1}{2}}+|x|)^{-2}.
$$
Now we decompose $V(x)$ as following:
\begin{align*}
V(x)=&\int_{0}^{1}\int_{|y-x|\geq2|x|}\Upsilon(x-y,1-s)\bar{f}(y,s)\,{\rm d}y{\rm d}s\\&
+\int_{0}^{1}\int_{\frac{|x|}{2}\leq|y-x|<2|x|}\Upsilon(x-y,1-s)\bar{f}(y,s){\rm d}y{\rm d}s\\&
+\int_{0}^{1}\int_{|y-x|\leq\frac{|x|}{2}}\Upsilon(x-y,1-s)\bar{f}(y,s)\,{\rm d}y{\rm d}s\\
\triangleq& {\rm I+II+III}.
\end{align*}
In the following, we always suppose $|x|>1$.  The first and the second term are easy to handle. Indeed,
note that
$$
|\bar{f}(y,s)|\leq C |y|^{-2}\leq C|x|^{-2}, \ \ \ y\in \{y: |y-x|\geq 2|x|\}
$$
and
$$
\quad |\Upsilon(x-y,1-s)|\leq C|x-y|^{-4}\leq C|x|^{-4},\ \ y\in \Big\{y:\frac{|x|}{2}\leq|y-x|<2|x|\Big\}.
$$
 From the above inequalities, it is easy to see that
\begin{align*}
{\rm I}\leq C|x|^{-2}\int_{0}^{1}\int_{2|x|}^{\infty}r^{-2}{\rm d}r\leq C|x|^{-3},
\end{align*}
and
\begin{align*}
\quad \quad \quad \quad \quad\quad {\rm II}\leq C|x|^{-4}\int_{0}^{1}\int_{\frac{|x|}{2}\leq|y-x|<{2|x|}}|y|^{-2}\,{\rm d}y\,{\rm d}s\leq C|x|^{-3}.
\end{align*}
 In order to prevent the logarithmic  loss,  we need to carefully deal with the last term ${\rm III}$. For this purpose,
we split it into
\begin{align*}
{\rm III}=&\int_{0}^{1}\int_{|y|<\frac{|x|}{2}}\Upsilon(y,1-s)\Big\{\bar{f}(x-y,s)-\tilde{f}(x,s)+y\cdot\nabla\bar{f}(x,s)\Big\}\,{\rm d}y{\rm d}s\\&
+\int_{0}^{1}\int_{|y|<\frac{|x|}{2}}\Upsilon(y,1-s)\bar{f}(x,s)\,{\rm d}y{\rm d}s-\int_{0}^{1}\int_{|y|<\frac{|x|}{2}}\Upsilon(y,1-s)y\cdot\nabla\bar{f}(x,s)\,{\rm d}y{\rm d}s\\
\triangleq& {\rm III_{1}+III_{2}+III_{3}}.
\end{align*}
From Gauss-Green formula,
$$
\int_{\mathbb{B}_{R}}\partial_{h}\mathcal{O}(y,1-s)\,{\rm d}y=\int_{\partial\mathbb{B}_{R}}\mathcal{O}(y,1-s)\frac{y_{h}}{R}\,{\rm d}\mathbb{S}=O(R^{-1}), \ \ \mbox{for}\ \ s\in [0,1)
$$
Thus, for $s\in[0,1)$
$$
\int_{{\R^{3}}}\Upsilon(y,1-s)\,{\rm d}y=\lim_{R\to\infty}\int_{\mathbb{B}_{R}}\partial_{h}\mathcal{O}(y,1-s)\,{\rm d}y=0,
$$
 from which, one instantly  obtains
\begin{align*}
|{\rm III_{2}}|&=\Big|{\lim_{\epsilon\to0^{+}}}\int_{0}^{1-\epsilon}\int_{|y|\leq\frac{|x|}{2}}\Upsilon(y,1-s)\bar{f}(x,s)\,{\rm d}y{\rm d}s\Big|\\&=
\Big|{\lim_{\epsilon\to0^{+}}}\int_{0}^{1-\epsilon}\int_{|y|\geq\frac{|x|}{2}}\Upsilon(y,1-s)\bar{f}(x,s)\,{\rm d}y{\rm d}s\Big|\\
&\leq C|x|^{-2}\int_{0}^{1}\int_{\frac{|x|}{2}}^{\infty}|y|^{-4}\,{\rm d}y{\rm d}s\leq C|x|^{-3}.
\end{align*}
Now, we estimate ${\rm III_{1}}$.  It is easy to see that there exists  $0<\theta<1$ such that
\begin{align*}
{\rm III_{1}}\leq\int_{0}^{1}\int_{|y|\leq \frac{|x|}{2}}|y|^{2}|\nabla^{2}\bar{f}(\theta(x-y)+(1-\theta)x,s)||\Upsilon(y,1-s)|\,{\rm d}y{\rm d}s.
\end{align*}
 On the other hand, we have from Proposition \ref{prop-4.2}
$$
|\nabla^{2}f(x)|\leq C(|V||\nabla^{2}V|+|U_{0}||\nabla^{2}U_{0}|)\leq C(1+|x|)^{-4},
$$
 this implies
$$
|\nabla^{2}\bar{f}(\theta(x-y)+(1-\theta)x,s)|\leq C(1+|x|)^{-4}, \ \ \ \mbox{for}\ \ y\in \big\{y:|y|\leq |x|/2\big\}.
$$
Thus,
$$
{\rm III_{1}}\leq C|x|^{-4}\int_{0}^{1}\int_{|y|\leq \frac{|x|}{2}}|y|^{-2}\,{\rm d}y{\rm d}s\leq C|x|^{-3}.
$$

To obtain the optimal estimate of ${\rm III_{3}}$, we have to  borrow   the special structure
of the Oseen kernel.
Now we decompose the ${\rm III_{3}}$ as
\begin{align*}
{\rm III_{3}}=&-\int_{0}^{1}\int_{|y|\leq (1-s)^{\frac{1}{2}}}y\cdot\nabla\bar{f}(x,s)\Upsilon(y,1-s)\,{\rm d}y{\rm d}s\\&
-\int_{0}^{1}\int_{(1-s)^{\frac{1}{2}}\leq|y|\leq\frac{|x|}{2}}y\cdot\nabla\bar{f}(x,s)\mathcal{F}(y)\,{\rm d}y{\rm d}s\\&
-\int_{0}^{1}\int_{(1-s)^{\frac{1}{2}}\leq|y|\leq\frac{|x|}{2}}y\cdot\nabla\bar{f}(x,s)(1-s)^{-\frac{3}{2}}\nabla_{y}\Psi(y/(1-s)^{\frac{1}{2}})\,{\rm d}y{\rm d}s\\&
\triangleq {\rm III_{31}+III_{32}+III_{33}}.
\end{align*}
By the cancelation condition \eqref{cancel}, we immediately have
\begin{align*}
{\rm III_{32}}=-\lim_{\epsilon\to0^{+}}\int_{0}^{1-\epsilon}\int_{(1-s)^{\frac{1}{2}}\leq|y|\leq\frac{|x|}{2}}y\cdot\nabla\bar{f}(x,s)\mathcal{F}(y)\,{\rm d}y{\rm d}s=0.
\end{align*}
On the other hand, from
$$
|\Upsilon(y,1-s)|\leq C |1-s|^{-\frac{3}{2}}
$$
and
\begin{align*}
|\nabla\bar{f}(x,s)|&=|s^{-1}\nabla_{x}f(x/s^{\frac{1}{2}})|=|s^{-\frac{3}{2}}\nabla f(x/s^{\frac{1}{2}})|
\\&\leq Cs^{-3/2}|U_{0}(x/s^{1/2})||\nabla U_{0}(x/s^{1/2})|\\&
\leq C(s^{1/2}+|x|)^{-3},
\end{align*}
 one instantly derives
$$
|{\rm III_{31}}|\leq C|x|^{-3}\int_{0}^{1}\int_{|y|\leq (1-s)^{\frac{1}{2}}}(1-s)^{-1}\,{\rm d}y{\rm d}s\leq C|x|^{-3}.
$$
 Finally, by  Lemma \ref{L4.1}, we see that
$$
|\nabla\Psi_{kj}(x)|\leq C e^{-|x|^{2}}\leq C(1+|x|)^{-m}\ \ \ \mbox{for any integer}\ m.
$$
This implies that
$$
|{\rm III_{33}}|\leq C|x|^{-3}\int_{0}^{1}\int_{0}^{\infty}(1+|z|)^{-4}\,{\rm d}z{\rm d}s\leq C|x|^{-3}
$$
with $z=\frac{y}{s^{1/2}}$.
Thus, combining the above estimates, we obtain
$$
{\rm III}\leq  C|x|^{-3}.
$$
This, together estimates of I, II and the fact $V\in L^{\infty}(\R^{3})$, implies
$$
|V(x)|\leq C(1+|x|)^{-3}.
$$
\end{proof}
At last,  collecting  Proposition \ref{L4.3} and Proposition \ref{prop-4.3} yields to Theorem \ref{T1.3}.

\section*{Acknowledgments} We  thank the anonymous referee and the associated editor  for their
invaluable comments   which helped to improve the paper.
This work is supported in part by the  National Key research and development program of China  (No.2020YFAO712903) and  NNSF of China
 under grant  No.11871087, No. 11971148 and  No.11831004.


\begin{thebibliography}{10}
	
	
	\bibitem{BCD11}
	H. Bahouri, J-Y. Chemin and R. Danchin,
	{ Fourier Analysis and Nonlinear Partial Differential Equations},
	A Series of Comprehensive Studies in Mathematics, Grundlehren der Mathematischen Wissenschaften, 343, Springer-Verlag Berlin Heidelberg, 2011.
	
 \bibitem{Br}
 L. Brandolese and  F. Vigneron, {\em New asymptotic profiles of nonstationary solutions of the Navier-Stokes system,} J. Math. Pures Appl.  88 (2007), 64-86.

  \bibitem{Br1}
L. Brandolese, and F. Vigneron, {\em Fine properties of self-similar solutions of the Navier-Stokes equations,} Arch. Ration. Mech. Anal. 192 (2009), 375-401.

\bibitem{Cannon-Meyer-Planchon}
 M. Cannone, Y. Meyer and F. Planchon, {\it Solutions
autosimilaires des {\'e}quations de Navier-Stokes}. {S{\'e}minaire
``{\'E}quations aux D{\'e}riv{\'e}es Partielles" de l'{\'E}cole
polytechnique}, Expos{\'e} VIII, 1993-1994.

\bibitem{C2} M. Cannone and  Y. Planchon, {\em Self-similar solutions for the Navier-Stokes equations in $\R^3$,} Comm. Part.
Diff. Equ. 21(1996), 179-193.

\bibitem{Cao}
 X. Cao, Q. Chen and  B. Lai, {\em Properties of the linear non-local Stokes operator and its application} Nonlinearity 32 (2019), 2633-2666.

 \bibitem{Colombo}
 M. Colombo, C. De Lellis and L. De Rosa, {\em Ill-Posedness of Leray Solutions for the Hypo-dissipative Navier-Stokes Equations,} Commun. Math. Phys. 362 (2018), 659-688.

 \bibitem{Colombo-1}
 M. Colombo, C. De Lellis and A. Massaccesi, {\em The generalized Caffarelli-Kohn-Nirenberg Theorem for the hyperdissipative Navier-Stokes system,} Comm. Pure Appl. Math. 73 (2020),  609-663.

\bibitem{Cons}
P. Constantin and  G. Iyer,  {\em Stochastic   Lagrangian representation of the three-dimensional incompressible
Navier-Stokes equations,} Comm. Pure Appl. Math. LXI (2008), 330-345.

 \bibitem{Cordoba} A. C\'ordoba and  D. C\"ordoba,  {\em A maximum principle applied to quasi-geostrophic equations,} Commun. Math. Phys. 249, 511-528 (2004)



\bibitem{De Lellis}
C. De Lellis and L. Szekelyhidi, Jr.,  {\em Dissipative continuous Euler flows,} Invent. Math. 193
(2013), 377-407.


\bibitem{Du}
S. Dubois, {\em What is a solution to the Navier-Stokes equations?} C. R. Acad. Sci. Paris, Ser. I 335 (2002) 27-32.


\bibitem{Fe}
P. Fern\'{a}ndez-Dalgo and P.  Lemari\'{e}-Rieusset, {\em Weak solutions for Navier-Stokes equations with initial data in weighted $L^2$ spaces,}
Arch. Ration. Mech. Anal. 237 (2020),  347-382.

\bibitem{Fr}
U. Frisch, M. Lesieur, and A. Brissaud, {\em Markovian random coupling model for
turbulence,} J. Fluid Mech. 65 (1974), 145-152.


\bibitem{JS}
H. Jia and  V. \v{S}ver\'{a}k, {\em Local-in-space estimates near initial time for weak solutions of the
Navier-Stokes equations and forward self-similar solutions,} Invent. Math., 196 (2014), 233-265.

\bibitem{Katz}
N. Katz and N. Pavlovic, {\em A cheap Caffarelli-Kohn-Nirenberg inequality for the Navier-
Stokes equation with hyper-dissipation,} Geom. Funct. Anal. 12 (2002), 355-379.

\bibitem{Koch}
H. Koch and D. Tataru, {\em Well-posedness for the Navier-Stokes equations,}\, Adv. Math. 157 (2001), 22-35.

\bibitem{KT}
M. Korobkov and  T. Tsai, \newblock{\em Forward self-similar solutions of the Navier-Stokes equations in the half space,} Anal. PDE 9 (2016), 1811-1827.

 \bibitem{LXZ} B. Lai, C. Miao and  X. Zheng, {\em The forward self-similar solutions of the fractional Navier-Stokes Equations,} \,Adv. Math. 352 (2019), 981-1043.

\bibitem{Lemarie}P. G. Lemari\'e-Rieusset, {\em The Navier-Stokes problem in the 21st century,} 2016, A Chapman Hall Book.



\bibitem{Ler}
J. Leray, {\em Essai sur le mouvement d' un fluid visqueux emplissant l' espace.} Acta Math., 63 (1934), 193-248.


\bibitem{Lions} J.L. Lions, {\em Quelques methodes de resolution des problemes aux limites non lineaires,} Dunod; Gauthier-Villars, Paris,
1969.

\bibitem{Mat}
J. Mattingly and Ya. G. Sinai, {\em An elementary proof of the existence and uniqueness theorem for the Navier-Stokes
equations,} Commun. Contemp. Math., 1 (1999), 497-516.

\bibitem{Me}
J.M. Mercado, E.P. Guido, A.J. S\'{a}nchez-Sesma, M.\'{I}\~{n}iguez, A. Gonz\'{a}lez et al. {  Analysis of the Blasius formula and the Navier-Stokes fractional equation,} In: Klapp, J. (ed.) Fluid Dynamics in Physics,
Engineering and Environmental Applications, Environmental Science and Engineering,  475-480.
Springer, Berlin, Heidelberg (2013).

\bibitem{Miao-book}
C. Miao, { Lecture on modern harmonic analysis and applications}, Monographs on Modern Pure Mathematics, No.63, Higher Education Press, Beijing,2018.


\bibitem{MYZ}
C. Miao, B. Yuan and  B. Zhang, {\em Well-posedness of the Cauchy problem for the fractional power	dissipative equations,}\, Nonlinear Anal., 68 (2008), 461-484.

\bibitem{NRS} J. Necas, M. Ruzicka and V. Sverak, {\em On Leray's self-similar solutions of the Navier-Stokes equations,} Acta Math., 176 (1996), 283-294

\bibitem{Rosa} L. De Rosa,  {\em Infinitely many Leray-Hopf solutions for the fractional Navier-Stokes equations,} Comm. Partial Differential Equations 44 (2019), no. 4, 335-365.

\bibitem{stein}E. M. Stein, { Singular integrals and differentiability properties of functions,} Princeton Mathematical Series, No. 30. Princeton University
    Press, Princeton, N.J., 1970.

 \bibitem{stein-2}E. M. Stein,   Harmonic Analysis: Real-Variable Methods, Orthogonality, and Oscillatory Integrals, Princeton University Press, Princeton, 1993.


\bibitem{Tang}
L. Tang and  Y. Yu, {\em Partial regularity of suitable weak solutions to the fractional Navier-Stokes equations,} Comm. Math. Phys. 334 (2015), 1455-1482.

\bibitem{Temam}
R. Temam, { Navier-Stokes equations: Theory and Numerical Analysis,}  Amsterdam: North--Holland, 1977.

\bibitem{Tsai}
T. Tsai, {\em On Leray's self-similar solutions of the Navier-Stokes equations satisfying local energy estimates,} Arch. Rational Mech. Anal. 143 (1998), 29-51.

\bibitem{T}
T. Tsai, {\em Forward discrete self-similar solutions of the Navier-Stokes equations,} Comm. Math. Phys., 328 (2014), 29-44.






\bibitem{Zhang}
X. Zhang, {\em Stochastic Lagrangian particle approach to fractal Navier-Stokes equations,} Comm. Math. Phys. 311 (2012), 133-155.

\end{thebibliography}
\end{document}